\documentclass[10pt]{ucthesis}
\def\dsp{\def\baselinestretch{2.0}\large\normalsize}
\def\hsp{\def\baselinestretch{1.5}\large\normalsize}
\def\ssp{\def\baselinestretch{1.0}\large\normalsize}
\hsp

\usepackage{amsfonts}
\usepackage{eucal}
\usepackage{amscd}
\usepackage{amsthm}
\usepackage{amsmath}
\usepackage{amssymb}
\usepackage{listings}
\usepackage{tikz}
\usepackage{comment}
\usepackage{array} %for vertical thick lines in tables
\usepackage{multirow} %multirow tables
\usepackage{cleveref}
\usepackage[osf]{mathpazo} %for Palantino

\usepackage{listings}
\newcommand{\zz}{\mathbb Z}
\newcommand{\zp}{{\mathbb Z}_p}
\newcommand{\qq}{\mathbb Q}
\newcommand{\qp}{{\mathbb Q}_p}
\newcommand{\hh}{\mathbb H}
\newcommand{\oo}{\mathbb O}
\newcommand{\rr}{\mathbb R}
\newcommand{\cc}{\mathbb C}
\newcommand{\ii}{\vec{\textbf{\i}}}
\newcommand{\jj}{\vec{\textbf{\j}}}
\newcommand{\kk}{\vec{\textbf{k}}}
\newcommand{\ee}{\textbf{1}}
\newcommand{\val}{\text{val}}
\newcommand{\der}{\mathfrak{Der}_{\oo}}

\theoremstyle{definition}
\newtheorem{definition}{Definition}[section]

\theoremstyle{plain}
\newtheorem{theorem}{Theorem}[section]
\newtheorem{proposition}{Proposition}[section]
\newtheorem{lemma}{Lemma}[section]
\newtheorem{corollary}{Corollary}[section]

\renewcommand{\theenumi}{(\alph{enumi})}

\newenvironment{Array}{\ssp\begin{array}}{\end{array}\dsp}

\def\listofsymbols{%%%%%%%%%%%%%%%%%%%%%%%
\begin{tabbing}
$k$~~~~~~~\=\parbox{5in}{\quad A field\hfill \pageref{symbol:k}}\\
\addsymbol V,W: {General composition algebras}{symbol:GeneralCompAlg}
\addsymbol N: {Norm of a composition algebra}{symbol:N}
\addsymbol B: {Bilinear form induced by $N$}{symbol:B}
\addsymbol T: {Trace of a composition algebra (\emph{or a general torus})}{symbol:T}
\addsymbol \overline{v}: {Involution on $v \in V$}{symbol:Involution}
\addsymbol \eta: {Embedding of one composition algebra into another}{symbol:embeddings}
\addsymbol l: {Splitting element}{symbol:l}
\addsymbol K: {General quadratic composition algebras}{symbol:K}
\addsymbol \mathfrak{B}: {General quaternion composition algebras}{symbol:frakB}
\addsymbol \mathfrak{C}: {General octonion composition algebras}{symbol:frakC}
\addsymbol \theta: {An embedding of $\text{SL}_3(k)$ into $\text{Aut}(\oo)$}{symbol:theta}
\addsymbol \gamma: {A (\emph{long root}) embedding of $\text{SL}_2(k)$ into $\text{Aut}(\oo)$}{symbol:gamma}
\addsymbol \delta: {A (\emph{short root}) embedding of $\text{SL}_2(k)$ into $\text{Aut}(\oo)$}{symbol:delta}
\addsymbol \gamma_{\vec{u}}^{\pm}(s): {An automorphism of $\oo$ dependant on the embedding $\eta_{\vec{u}}$}{symbol:gammapm}
\addsymbol \delta_{\vec{u}}^{\pm}(s): {An automorphism of $\oo$ dependant on the embedding $\eta_{\vec{u}}$}{symbol:delta}
\addsymbol \Phi: {Set of roots}{symbol:Phi}
\addsymbol \Phi^+: {Set of positive roots}{symbol:Phiplus}
\addsymbol \Delta: {Set of simple roots}{symbol:Delta}
\addsymbol E: {Root space of $\Phi$}{symbol:E}
\addsymbol ( \cdot , \cdot ): {Inner product on $E$}{symbol:rootinnerproduct}
\addsymbol \left| \,\cdot\, \right|^2: {Norm-squared on $E$}{symbol:Phi}
\addsymbol W: {The Weyl group}{symbol:W}
\addsymbol s_{\alpha}: {A reflection in $W$}{symbol:reflection}
\addsymbol G_{\Phi}(k): {Chevalley group constructed from a root system $\Phi$ and a field $k$}{symbol:ChevGroup}
\addsymbol e_{\alpha}(t): {The Chevalley generators}{symbol:ChevElements}
\addsymbol w_{\alpha}(t): {A Chevalley (Weyl) element}{symbol:ChevElements}
\addsymbol h_{\alpha}(t): {A Chevalley (toral) element}{symbol:ChevElements}
\addsymbol N_{ij}: {Chevalley structure constants.  No relation to the norm $N$.}{symbol:ChevalleyStructureConstants}
\addsymbol T_{\Phi}: {The maximal torus of $G_{\Phi}(k)$ generated by the $h_{\alpha}(t)$}{symbol:Torus}
\addsymbol N_{\Phi}: {The normalizer of $T_{\Phi}$ generated by the $w_{\alpha}(t)$}{symbol:Normalizer}
\addsymbol \mathfrak{b}: {A Chevalley basis of $\mathfrak{g}$}{symbol:ChevBasis}
\addsymbol E_{\alpha}: {A Chevalley basis element}{symbol:Ealpha}
\addsymbol H_{\alpha}: {A Chevalley basis element}{symbol:Halpha}
\addsymbol \der: {The algebra of derivations on $\oo$}{symbol:Der}
\addsymbol \Phi^{\vee}: {Set of coroots}{symbol:Phivee}
\addsymbol X^\bullet (T): {Group of characters of a torus $T$}{symbol:charactersandcocharacters}
\addsymbol X_\bullet (T): {Group of cocharacters of a torus $T$}{symbol:charactersandcocharacters}
\addsymbol \langle \cdot,\cdot \rangle: {Pairing between $\Phi$ and $\Phi^\vee$}{symbol:pairing}
\addsymbol \mathcal{A}: {The affine apartment constructed from $T_\Phi$}{symbol:apartment}
\addsymbol H_{\alpha+n}: {Vanishing hyperplanes in $\mathcal{A}$}{symbol:hyperplane}
\addsymbol \mathcal{B}(G): {The Bruhat-Tits building of $G$}{symbol:building}
\addsymbol \Lambda: {Lattice in $\oo$}{symbol:lattice}
\addsymbol \Lambda^\ast: {Dual lattice in $\oo$}{symbol:duallattice}
\addsymbol \Lambda_r: {Lattice filtration in $\oo$}{symbol:latticefiltration}
\addsymbol \underline\Lambda_r: {Standard lattice sequence in $\oo$}{symbol:standardlattice}
\addsymbol \val: {Valuation on either $k$ or on an algebra}{symbol:val}
\end{tabbing}  \clearpage}
\def\addsymbol #1: #2#3{$#1$ \> \parbox{5in}{\quad #2 \hfill \pageref{#3}}\\}
\def\newnot#1{\label{#1}}

\begin{document}

% Declarations for Front Matter

\title{Lattice Filtrations for $G_2$ of a $\lowercase{p}$\,-adic Field}
\author{Paul Tokorcheck}
\degreeyear{2012}
\degreemonth{June}
\degree{DOCTOR OF PHILOSOPHY}
\chair{Martin H. Weissman}
\committeememberone{Robert Boltje}
\committeemembertwo{Hirotaka Tamanoi}
\numberofmembers{3} %% (including chair) possible: 3, 4, 5, 6
\deanlineone{Dean Tyrus Miller}
\deanlinetwo{Vice Provost and Dean of Graduate Studies}
\deanlinethree{}
\field{Mathematics}
\campus{Santa Cruz}

\begin{frontmatter}

\maketitle

%\copyrightpage

\tableofcontents
\listoffigures

%\listoftables

\newpage
\chapter*{List of Symbols\hfill} \addcontentsline{toc}{chapter}{List of Symbols}
%%%%%%%%%%%%%%%%%%%%%%%
%Sample List of Symbols
%%%%%%%%%%%%%%%%%%%%%%%
\begin{tabbing}
% YOU NEED TO ADD THE FIRST ONE MANUALLY TO ADJUST THE TABBING AND SPACES
$k$~~~~~~~\=\parbox{5in}{\quad A field\hfill \pageref{symbol:k}}\\
%ADD THE REST OF SYMBOLS WITH THE HELP OF MACRO
\addsymbol V,W: {General composition algebras}{symbol:GeneralCompAlg}
\addsymbol N: {Norm of a composition algebra}{symbol:N}
\addsymbol B: {Bilinear form induced by $N$}{symbol:B}
\addsymbol T: {Trace of a composition algebra (\emph{or a general torus})}{symbol:T}
\addsymbol \overline{v}: {Involution on $v \in V$}{symbol:Involution}
\addsymbol \eta: {Embedding of one composition algebra into another}{symbol:embeddings}
\addsymbol l: {Splitting element}{symbol:l}
\addsymbol K: {General quadratic composition algebras}{symbol:K}
\addsymbol \mathfrak{B}: {General quaternion composition algebras}{symbol:frakB}
\addsymbol \mathfrak{C}: {General octonion composition algebras}{symbol:frakC}
% .
% .
\addsymbol \theta: {An embedding of $\text{SL}_3(k)$ into $\text{Aut}(\oo)$}{symbol:theta}
\addsymbol \gamma: {A (\emph{long root}) embedding of $\text{SL}_2(k)$ into $\text{Aut}(\oo)$}{symbol:gamma}
\addsymbol \delta: {A (\emph{short root}) embedding of $\text{SL}_2(k)$ into $\text{Aut}(\oo)$}{symbol:delta}
\addsymbol \gamma_{\vec{u}}^{\pm}(s): {An automorphism of $\oo$ dependant on the embedding $\eta_{\vec{u}}$}{symbol:gammapm}
\addsymbol \delta_{\vec{u}}^{\pm}(s): {An automorphism of $\oo$ dependant on the embedding $\eta_{\vec{u}}$}{symbol:delta}
% .
% .
\addsymbol \Phi: {Set of roots}{symbol:Phi}
\addsymbol \Phi^+: {Set of positive roots}{symbol:Phiplus}
\addsymbol \Delta: {Set of simple roots}{symbol:Delta}
\addsymbol E: {Root space of $\Phi$}{symbol:E}
\addsymbol ( \cdot , \cdot ): {Inner product on $E$}{symbol:rootinnerproduct}
\addsymbol \left| \,\cdot\, \right|^2: {Norm-squared on $E$}{symbol:Phi}
\addsymbol W: {The Weyl group}{symbol:W}
\addsymbol s_{\alpha}: {A reflection in $W$}{symbol:reflection}
\addsymbol G_{\Phi}(k): {Chevalley group constructed from a root system $\Phi$ and a field $k$}{symbol:ChevGroup}
\addsymbol e_{\alpha}(t): {The Chevalley generators}{symbol:ChevElements}
\addsymbol w_{\alpha}(t): {A Chevalley (Weyl) element}{symbol:ChevElements}
\addsymbol h_{\alpha}(t): {A Chevalley (toral) element}{symbol:ChevElements}
\addsymbol N_{ij}: {Chevalley structure constants.  No relation to the norm $N$.}{symbol:ChevalleyStructureConstants}
\addsymbol T_{\Phi}: {The maximal torus of $G_{\Phi}(k)$ generated by the $h_{\alpha}(t)$}{symbol:Torus}
\addsymbol N_{\Phi}: {The normalizer of $T_{\Phi}$ generated by the $w_{\alpha}(t)$}{symbol:Normalizer}
% .
% .
\addsymbol \mathfrak{b}: {A Chevalley basis of $\mathfrak{g}$}{symbol:ChevBasis}
\addsymbol E_{\alpha}: {A Chevalley basis element}{symbol:Ealpha}
\addsymbol H_{\alpha}: {A Chevalley basis element}{symbol:Halpha}
\addsymbol \der: {The algebra of derivations on $\oo$}{symbol:Der}
% .
% .
\addsymbol \Phi^{\vee}: {Set of coroots}{symbol:Phivee}
\addsymbol X^\bullet (T): {Group of characters of a torus $T$}{symbol:charactersandcocharacters}
\addsymbol X_\bullet (T): {Group of cocharacters of a torus $T$}{symbol:charactersandcocharacters}
\addsymbol \langle \cdot,\cdot \rangle: {Pairing between $\Phi$ and $\Phi^\vee$}{symbol:pairing}
\addsymbol \mathcal{A}: {The affine apartment constructed from $T_\Phi$}{symbol:apartment}
\addsymbol H_{\alpha+n}: {Vanishing hyperplanes in $\mathcal{A}$}{symbol:hyperplane}
\addsymbol \mathcal{B}(G): {The Bruhat-Tits building of $G$}{symbol:building}
% .
% .
\addsymbol \Lambda: {Lattice in $\oo$}{symbol:lattice}
\addsymbol \Lambda^\ast: {Dual lattice in $\oo$}{symbol:duallattice}
\addsymbol \Lambda_r: {Lattice filtration in $\oo$}{symbol:latticefiltration}
\addsymbol \underline\Lambda_r: {Standard lattice sequence in $\oo$}{symbol:standardlattice}
\addsymbol \val: {Valuation on either $k$ or on an algebra}{symbol:val}
% .
% .
% ALWAYS KEEP THE FOLLOWING LINE
\end{tabbing}  \clearpage

\begin{abstract}
\thispagestyle{plain}
\pagestyle{plain}
In this work we aim to describe, in significant detail, certain filtrations and sequences of lattices inside of the split octonion algebra $\oo$, when $\oo$ is constructed over the field $\qp$.  There are several reasons why one would like to consider such lattice filtrations, not the least of which is the connection that they have to certain filtrations of subgroups of the automorphism group $G_2 = \text{Aut}(\oo)$.

It is an eventual goal to try to uncover previously unknown supercuspidal representations of $G_2$ by examining representations of the subgroups making up these subgroup filtrations, which are known as Moy-Prasad filtrations.  When $G_2$ is constructed over $\qp$, each of the filtration subgroups will be compact and open, and normal in the previous subgroup in the filtration, so that the respective quotients are all finite groups.  A basic strategy then, is to identify representations of these finite quotients, extend them to representations of the filtration subgroups, and then induce them to representations of the whole group $G_2$.

To understand the lattice filtrations that identify the Moy-Prasad filtration subgroups, the work of W.T. Gan and J.K. Yu in \cite{GanYu2003} is indispensable.  In their article, they draw connections between certain lattice filtrations, octonion orders, maximinorante norms, and points in the Bruhat-Tits building $\mathcal{B}(G_2)$.  Their main idea was to use the norm-preserving quadratic form inherent in $\oo$, along with the natural $8$-dimensional representation of $G_2 = \text{Aut}(\oo)$, to create a canonical embedding of the building into $\mathcal{B}(\text{SO}(\oo))$.   Since the latter building had been previously described as the set of ``maximinorante norms'' on $\oo$, they arrived at an explicit description of $\mathcal{B}(G_2)$ in terms of certain maximinorante norms and orders in $\oo$.

The present work then attempts to describe all of these structures in detail, beginning with general composition algebras and the construction of the group $G_2$.  Then we will construct the Bruhat-Tits building $\mathcal{B}(G_2)$ via the coroot lattice of type $G_2$, though we will mainly concern ourselves only with the standard apartment in $\mathcal{B}(G_2)$.  Finally we will define our lattice filtrations and draw the connections between them and points $\mathcal{B}(G_2)$ outlined in \cite{GanYu2003}, which will reveal the action of the group $G_2$ on its own building.  Along the way, we identify many important structures and facts about the group itself.
\end{abstract}

\begin{preface}
This document contains the entire content of the author's Ph.D. Dissertation, written at UC Santa Cruz and submitted on $14$ June $2012$.  Though the title page has not been changed, a few new elements have subsequently been added, such as a List of Symbols.  Many typographical errors have been fixed, and the overall formatting has been changed from the UC standard to improve readability.
\end{preface}

\begin{dedication}
\null\vfil
{\large
\begin{center}
To Patrice Boyle, Karen Madura, Tiffany Darden, Eve Krammer, and the rest of the Seabright neighborhood.  It was you that transformed Santa Cruz into a home that I will miss.  I could not have done this without you.
\end{center}}
\vfil\null
\end{dedication}

\begin{acknowledgements}
I would like to thank the UCSC Department of Mathematics, and in particular Robert Boltje and Geoffrey Mason, whose excellent and unforgiving instruction in all things algebraic put me on the right path over my first few years in Santa Cruz.  I am also grateful to Tony Tromba, who asked Corey Shanbrom and myself to help in the revision of his textbook, \textbf{Vector Calculus}, and also to write the companion guide attached to it.  Through that project I learned an enormous amount about writing, textbooks, and the inner workings of the publishing industry.  Finally, I thank my advisor, Marty Weissman.  Marty introduced me to many branches of mathematics that I never knew existed, and helped me along with a seemingly infinite amount of patience.  I am very lucky to have been his student.
\end{acknowledgements}

\end{frontmatter}

%%%%%%%%%%%%%%%%%%%%%%%%%%%%%%%%%%%%%%%%%%%%%%%%%%%%%%%%%%%%%%%%%%%%%%%%%%%%%%%%%%%%%%%%%%%%%%%%%%%

\part{Part I: General Fields}

\chapter{Introduction}

In Part I, we will review the general theory of composition algebras over a field $k$, and their $k$- algebra automorphisms.  Composition algebras come in several different shapes and sizes, but we generally refer to those of dimension $8$ as octionion (or Cayley) algebras.  This will lead us to an initial definition:

\begin{definition}
If $\oo$ is an octonion algebra over a field $k$, the group $G_{2,\oo}$ is the group of $k$-algebra automorphisms of $\oo$.
\end{definition}

However, the nomenclature we use implies that when constructed over the field $k=\cc$, the group $G_2$ should be associated to the exceptional complex semi-simple Lie algebra $\frak{g}_2$.  A priori, this connection is far from obvious.  We deal with this connection in Chapters \ref{chevalleygroups} and \ref{G2generators} where we first describe the Chevalley construction of a group of Lie type from a Lie algebra, and then identify particular automorphisms of our octonion algebra which act as the Chevalley generators of the group.  These associations between automorphisms and Chevalley elements will later allow us to describe other structures in $G_2$ very explicitly.  We will also be able to calculate certain structure constants for $G_2$ by specifying a choice and ordering of the generators.

Though we will later take $k$ to be the local non-archimedean field $k=\qp$, the results discussed in Part I are more general.  Therefore we may take $k$ to be any field of characteristic not equal to $2$.

\chapter{Composition Algebras and the Octonions.}
\label{compositionalgebras}

\section{Objects}
\label{objects}

Taking $k$\label{symbol:k} to be an arbitrary field of characteristic not $2$, we use standard notation from vector calculus: namely, the symbols $\ii$, $\jj$, and $\kk$ will denote the standard unit vectors in $k^3$, while $\circ$ and $\times$ will denote the usual dot product and cross product.  We begin with a preliminary definition.

\begin{definition}
\label{def:compalg}
A \textbf{composition algebra}\label{symbol:GeneralCompAlg} $V$ is a unital $k$-algebra\footnote{We do not require $V$ to be either commutative or associative.}, which is further endowed with a quadratic form\label{symbol:N} $N \colon V \rightarrow k$, with the following properties:
\begin{enumerate}
\item $N$ is multiplicative.

\item $N$ is nondegenerate.  That is, the associated symmetric bilinear form $B$ is nondegenerate, where $B$ is defined\label{symbol:B} by
\[ B(v,w) = \frac{1}{2} [N(v+w) - N(v) - N(w)]. \]
\end{enumerate}

The quadratic form $N$ is called the \textbf{norm} of the composition algebra $V$.\footnote{Though term `norm' in this context is standard, we point out that $N$ is \emph{not} a norm in the analytic sense of a normed vector space.  In fact, since $N$ maps into an arbitrary field $k$ which will not be ordered in general, concepts such as `positive definite' or `triangle inequality' are unlikely to have any meaning in our context.  Conversely, the standard norms placed on vector spaces are generally not quadratic.}
\end{definition}

The nondegenerate quadratic form $N$ places quite a bit of structure on $V$.  For example, the bilinear form $B$ can be used to define orthogonality in $V$, in the sense that $v,w \in V$ are called orthogonal if $B(v,w)=0$.  The quadratic form $N$ also leads to the following structures:

\begin{definition}
Let $V$ be a composition algebra with norm $N$. Then we have also a \textbf{trace}\label{symbol:T}, defined on $v \in V$ by \[ T(v) := N(v + \ee) - N(v) - N(\ee) = 2B(v,\ee), \]
and an \textbf{involution}\label{symbol:Involution}, defined on $v \in V$ by
\[ \overline{v} := T(v)\ee - v. \]
\end{definition}

\begin{lemma}
We have the following facts about the identity:
\begin{enumerate}
\item $N(\ee) = 1.$

\item $B(\ee,\ee) = \frac{1}{2} [N(2 \cdot \ee) - 2N(\ee)] = \frac{1}{2} [4N(\ee) - 2N(\ee)] = 1.$

\item $T(\ee) = 2B(\ee,\ee) =  2.$

\item $\overline{\ee} = T(\ee)\ee - \ee = \ee.$
\end{enumerate}
\end{lemma}

\begin{proof}
All statements follow directly from the definitions.
\end{proof}

We next justify the terminology used in our last definition, and verify a few other facts.

\begin{lemma}
\label{lem:normfacts}
Let $V$ be any composition algebra.  Then for all $v,w \in V$:
\begin{enumerate}
\item Both the trace and involution are linear.

\item The subalgebra of $V$ fixed by the involution is equal to $k\ee$.

\item $T(v) \ee = v + \overline{v}.$

\item $\overline{\overline{v}} = v.$

\item $\overline{vw} = \overline{w} \cdot \overline{v}.$

\item $B(v,w) = \frac{1}{2}T(v \overline{w}) = B(v\overline{w},1).$

\item $N(v) \ee = v \overline{v}.$

\item $v^2 - T(v)v + N(v)\ee = 0.$
\end{enumerate}
\end{lemma}

\begin{proof}
All statements and their proofs can be found in either \cite{Jacobson1958} or \cite{SpringVeld2000}.
\end{proof}

Looking carefully at Lemma \ref{lem:normfacts}, we see that all of the other structures ($N$, $T$, and $B$) can be described solely in terms of the involution.  However, it is still necessary for the resulting $N$ and $B$ to satisfy the properties of Definition \ref{def:compalg}, if $V$ is to be considered a composition algebra.

We now describe a number of examples of compositions algebras.

\begin{enumerate}
\item The field $k$ itself is a composition algebra of dimension $1$, with quadratic form $N(v) = v^2$, trace given by $T(v) = 2v$, and trivial involution.  This norm is clearly multiplicative and nondegenerate.

\item The vector space $k \times k$ is an two dimensional $k$-algebra under component-wise multiplication, with identity $\ee = (1,1)$.  The norm form $N$ is given by $N(a,b) = ab$, which is again multiplicative and non-degenerate, making this a composition algebra.  The trace is given by $T(a,b) = a+b$, and the involution is given by $\overline{(a,b)} = (b,a)$.  Note that this algebra is both associative and commutative.

\item Perhaps the most well-known composition algebra is the algebra of $2 \times 2$ matrices over a field $k$, $\text{Mat}_2(k)$.  Here, $N$ is given by the determinant, which gives rise to the usual trace and involution on $\text{Mat}_2(k)$:
    \[          \overline{\left(
                   \begin{Array}{cc}
                     a & b \\
                     c & d \\
                   \end{Array}
                 \right)} = \left(
                   \begin{Array}{cc}
                     d & -b \\
                     -c & a \\
                   \end{Array}
                 \right). \]
    It is well known that the determinant is multiplicative.  For non-degeneracy, note that for $v,w \in \text{Mat}_2(k)$,
    \begin{gather*}
    B(v,w) = \frac{1}{2}[ \det(v+w) - \det(v) - \det(w) ] = 0
    \intertext{implies that}
    \det(v+w) = \det(v) + \det(w).
    \end{gather*}
    But since the determinant is not an additive homomorphism, for this to be true for all $v \in \text{Mat}_2(k)$ we have $w=0$ necessarily.  Therefore $\text{Mat}_2(k)$ is a composition algebra, and all the appropriate relations are satisfied between these structures.  Note that this algebra is associative but \emph{not} commutative.

\item Zorn's octonions\cite{Zorn1933} $\oo$ are given as a set by:
\[ \oo := \left\{ \left(
                   \begin{Array}{cc}
                     a & \vec{v} \\
                     \vec{w} & d \\
                   \end{Array}
                 \right) \biggl\vert\; a,d \in k, \text{ and } \vec{v},\vec{w} \in k^3 \right\}. \]
Addition in this set is defined in the regular way (entry-wise), and it also has a multiplication, given by:
\[ \left(
                   \begin{Array}{cc}
                     a & \vec{v} \\
                     \vec{w} & d \\
                   \end{Array}
                 \right)\left(
                   \begin{Array}{cc}
                     \alpha & \vec{\phi} \\
                     \vec{\psi} & \delta \\
                   \end{Array}
                 \right) = \left(
                   \begin{Array}{cc}
                     a\alpha + \vec{v} \circ \vec{\psi} & a \vec{\phi} + \delta \vec{v} - \vec{w} \times \vec{\psi} \\
                     \alpha \vec{w} + \delta \vec{\psi} + \vec{v} \times \vec{\phi} & d \delta + \vec{w} \circ \vec{\phi} \\
                   \end{Array}
                 \right). \]
The identity is the usual $\ee = \left(
                   \begin{Array}{cc}
                     1 & 0 \\
                     0 & 1 \\
                   \end{Array}
                 \right)$, and $\oo$ is also equipped with a necessary norm, which is analogous to the determinant in $\text{Mat}_2(k)$:
\[ N \left(
                   \begin{Array}{cc}
                     a & \vec{v} \\
                     \vec{w} & d \\
                   \end{Array}
                 \right) = ad - \vec{v} \circ \vec{w}. \]
It follows directly from Definition $3$ that the trace $T$ on $\oo$ is the usual trace, defined as the sum of the diagonal entries, and the involution is
\[ \overline{\left(
                   \begin{Array}{cc}
                     a & \vec{v} \\
                     \vec{w} & d \\
                   \end{Array}
                 \right)} = \left(
                   \begin{Array}{cc}
                     d & -\vec{v} \\
                     -\vec{w} & a \\
                   \end{Array}
                 \right). \]

Showing that this quadratic norm form (determinant) is multiplicative is straightforward and requires only the identity
\[ (\vec{v} \times \vec{\phi}) \circ (\vec{w} \times \vec{\psi}) = (\vec{v} \circ \vec{w})(\vec{\phi} \circ \vec{\psi}) - (\vec{v} \circ \vec{\psi})(\vec{\phi} \circ \vec{w}).  \]

Showing that $N$ is nondegenerate is similar to Example (c):
\[
B \left(\left(
                   \begin{Array}{cc}
                     a & \vec{v} \\
                     \vec{w} & d \\
                   \end{Array}
                 \right),\left(
                   \begin{Array}{cc}
                     \alpha & \vec{\phi} \\
                     \vec{\psi} & \delta \\
                   \end{Array}
                 \right)\right) = 0 \]
implies that
\[ N\left(
                   \begin{Array}{cc}
                     a + \alpha & \vec{v}+\vec{\phi} \\
                     \vec{w} +\vec{\psi} & d + \delta \\
                   \end{Array}
                 \right) = N\left(
                   \begin{Array}{cc}
                     a & \vec{v} \\
                     \vec{w} & d \\
                   \end{Array}
                 \right) + N\left(
                   \begin{Array}{cc}
                     \alpha & \vec{\phi} \\
                     \vec{\psi} & \delta \\
                   \end{Array}
                 \right), \]
which implies that
\[ a \delta + \alpha d + (\vec{v} \circ \vec{\psi}) + (\vec{\phi} \circ \vec{w}) = 0.\]
This will be true for arbitrary $a,d \in k$ and $\vec{v},\vec{w} \in k^3$ if and only if $\alpha = \delta = 0 \in k$ and $\vec{\phi} = \vec{\psi} = 0 \in k^3$.  The algebra $\oo$ is therefore a composition algebra.  Note that this algebra is neither associative nor commutative.  However, it is \textbf{alternative}, in the sense of the definition below.
\end{enumerate}

\begin{definition}
A $k$-algebra $V$ is called \textbf{alternative} if, for all $v,w \in V$, we have:
\begin{gather*}
v(vw) = (vv)w, \\
v(ww) = (vw)w, \\
v(wv) = (vw)v.
\end{gather*}
Equivalently, $V$ is alternative if for all $v,w \in V$, the subalgebra generated by $v$ and $w$ is associative.
\end{definition}

\begin{definition}
A composition algebra is called \textbf{split} if it contains zero divisors.  Otherwise, it is called \textbf{non-split}.
\end{definition}

Since it is not difficult to find zero-divisors in each of the examples $2-4$ above, they are all examples of split composition algebras.  However, there are examples of composition algebras which are \emph{not} split, such as the field $\cc$ of complex numbers, Hamilton's quaternions $\hh$, or Graves' Octonions $\oo_G$, all of which are normed division algebras and may be constructed using the Cayley-Dickson construction from $k = \rr$.

\section{Morphisms}
\label{morphisms}

\begin{definition}
Let $(V,N)$ and $(V',N')$ be two composition algebras.  We call $f: (V,N) \rightarrow (V',N')$ a \textbf{morphism of composition algebras} if it is a morphism of algebras from $V$ to $V'$ which additionally preserves the norm form $N$:
\[ N'(f(v)) = N(v), \quad \forall\, v \in V. \]
\end{definition}

While this is a perfectly fine definition, we can also prove that this is always the case: any morphism of algebras will \emph{necessarily} preserve $N$, as outlined in the following theorem.

\begin{theorem}
\label{morphismspreservenorm}
Let $(V,N)$ and $(V',N')$ be two composition algebras, and $f: (V,N) \rightarrow (V',N')$ be a morphism of composition algebras. Then $N'(f(v)) = N(v)$ for all  $v \in V$.
\end{theorem}

\begin{proof}
With the hypotheses of our theorem, we choose $v \in V$.  Suppose first that $v \in k\ee \subseteq V$.  In this case, $f$ (as an algebra morphism) acts trivially on $v$, and in both $V$ and $V'$ the norm of $v$ is just $N(v)= v \overline{v} = v^2$. Thus $f$ preserves the norm in this case.

Suppose now that $v \notin k\ee$.  Examine the ideal $I = \left\{ P \in k[X] \mid P(v)=0 \right\} \subseteq V$.  Note that $I$ contains no nonzero elements of $\deg P \leq 1$ (else $v \in k\ee$).  Since $k[X]$ is a PID, $I = \langle P_v \rangle$ for some monic polynomial $P_v \in k[X]$.  In fact, by Lemma $2$(h), the characteristic polynomial $v^2 - T(v)v + N(v)\ee$ is in $I$, and this must be the unique polynomial equal to $P_v$ (else you could subtract the two monic polynomials and the result would be linear in $I$).

If we apply the algebra morphism $f$ to $P_v$, we find that
\[ f(0) = f\biggl(v^2 - T(v)v + N(v)\ee\biggr) = f(v)^2 - T(v)f(v) + N(v)\ee = 0. \]
Therefore $N(v)$ is also the constant term of the characteristic polynomial of $f(v) \in V'$, and our statement follows.
\end{proof}

Therefore, morphisms of algebras and morphisms of composition algebras are equivalent concepts, and the category of composition algebras and their morphisms is actually a full subcategory of the $k$-algebras.

Now, it happens that the examples of split composition algebras listed in Section $2.1$ are each embedded in the next as composition algebras:
\[ k \hookrightarrow k \times k \hookrightarrow \text{Mat}_2(k) \hookrightarrow \oo. \]

For example, we have the diagonal embedding\label{symbol:embeddings} $\eta_1: v \mapsto (v,v)$ of $k$ into $k \times k$.  The (other) diagonal embedding of $k \times k$ into $\text{Mat}_2(k)$ may be given by
\[ \eta_2: (a,b) \mapsto \left(
                   \begin{Array}{cc}
                     a & 0 \\
                     0 & b \\
                   \end{Array}
                 \right).
 \]

Since $\eta_1$ must preserve the identity, but must also be linear, $\eta_1$ is unique.  However, the embedding $\eta_2$ is not unique, nor are the embeddings $\text{Mat}_2(k) \hookrightarrow \oo$.   In the latter case, we may choose any unit vector $\vec{u} \in k^3$ (i.e., having $u \circ u = 1$) to yield an embedding
\[ \eta_3: \left(
     \begin{Array}{cc}
       a & b \\
       c & d \\
     \end{Array}
   \right) \mapsto \left(
     \begin{Array}{cc}
       a & b \vec{u} \\
       c \vec{u} & d \\
     \end{Array}
   \right). \]
It is easily verified that all of these maps are injective algebra morphisms, and hence they preserve the respective forms $N$.  In the last map, we will often choose $\vec{u}$ to be equal to one of the standard basis vectors $\ii$, $\jj$, or $\kk$.  When necessary, we will specify our choice of embedding by writing $\eta_{\ii}(\text{Mat}_2(k))$, $\eta_{\jj}(\text{Mat}_2(k))$, $\eta_{\kk}(\text{Mat}_2(k))$, as appropriate.

\section{Generating and Decomposing Composition Algebras}
\label{decompositions}

The goal of this section is to describe how a composition algebra can be decomposed into a composition algebra of smaller dimension and its orthogonal complement.  We will use this decomposition later to describe certain automorphisms of $\oo$, whose action is defined by the choice of decomposition.  A secondary goal is to classify all possible composition algebras by their dimension, and by whether or not they are split. This will help clarify the question of ``which $G_2$'' we refer to at any given time.

Toward these ends, we will describe how to generate a new composition algebra from an given one, using a variation of the Cayley-Dickson construction.  As always, we will refer to our set of examples for insight.

Let $(V,N)$ be any composition algebra, with its usual structure given by the norm form $N$.  We form the space $W := V \oplus Vl$, where at this point $l$ is simply a formal variable\label{symbol:l}.  For any nonzero $\mu \in k^{\times}$, multiplication in $W$ is given by \cite{Jacobson1958}
\[ (a+bl)(c+dl) = (ac + \mu \overline{d}b) + (da + b\overline{c})l. \]

Note that in our previous examples of split composition algebras we have taken $\mu = 1$.  If we take $\mu = -1$, this is simply the familiar Caley-Dickson construction which can be used with $k=\rr$ to construct the normed division algebras $\cc$, $\hh$, and $\oo_G$.

\begin{proposition}
The map $a+bl \mapsto \overline{a}-bl$ defines an involution on $W$.
\end{proposition}

\begin{proof}
The map clearly has order two, and is $k$-linear since $\overline{a}$ is linear on $V$.  We also have:
\begin{align*}
\overline{(a+bl)} \cdot \overline{(c+dl)} & = (\overline{a}-bl)(\overline{c}-dl) \\
    & = \left( \overline{a} \cdot \overline{c} + \mu \overline{d}b \right) + \left( -d\overline{a} - b\overline{\overline{c}} \right)l \\
    & = (\overline{ca} + \mu \overline{\overline{b}d}) - (bc + d\overline{a})l \\
    & = \overline{(ca + \mu \overline{b}d)} - (bc + d\overline{a})l \\
    & = \overline{(ca + \mu \overline{b}d) + (bc + d\overline{a})l} \\
    & = \overline{(c+dl)(a+bl)}. \\
\end{align*}
\end{proof}

We often write $\overline{a+bl} = \overline{a}-bl$, but note that there are two distinct involutions present in this expression, one on $V$ and one on $W$.  Thus if we associate $V = \left\{a +bl \in W \mid b =0 \right\}$, the involutions on $V$ and $W$ will coincide.

\begin{definition}
We define a norm form $N_W(w) := w \overline{w}$ for all $w \in W$, and therefore a trace $T_W$ and bilinear form $B_W$ constructed as we described previously.
\end{definition}

It is immediate that $N_W$ is multiplicative as
\[ N_W(vw) = vw\overline{vw} = vw \bar{w} \bar{v} = v N_W(w) \bar{v} = N_W(w) v\bar{v} = N_W(v)N_W(w). \]
However, at this point we do not yet know that the associated $B_W$ is non-degenerate.  In fact, this will not always be the case, and thus algebras $W$ constructed in this way may or may not be composition algebras.  Nevertheless, we collect here a number of facts about $W$ which we will have occasion to use later.  We will neglect the subscript on our maps $N$, $B$, or $T$ when the context is clear.

\begin{lemma}
For all $a \in V$, we have $al = l \overline{a}$.
\end{lemma}

\begin{proof}
Let $v \in W$ be arbitrary.  We have that $N(v)\ee = (v \overline{v})\ee$, which we can linearize $v$ as $v=a+b$, and
\begin{align*}
N(a+b)\ee &= ((a+b)\overline{(a+b)})\ee = (a\overline{a} + a\overline{b} + b\overline{a} + b\overline{b})\ee
\intertext{implies that}
2B(a,b)\ee &= a\overline{b}\ee + b\overline{a}\ee.
\end{align*}
Therefore, if we choose $a$ to be from $V$ and $b=l$, then $B(a,l)=0$ and $\overline{l} = -l$, and $0 = -al\ee + l\overline{a}\ee$.
\end{proof}

\begin{lemma}
\label{associationlemma}
For all $a,b \in V$, we have that $(ab)l = b(al)$, and $l(ab) = (lb)a$.
\end{lemma}

\begin{proof}
Let $v,w \in W$ be arbitrary.  Using the same arguments as the last lemma, we use the alternative property of $W$ to write $N(v)w = (v \overline{v})w = v (\overline{v}w)$, which we can linearize as $v=c+d$, and
\begin{align*}
N(c+d)w &= (c+d)(\overline{(c+d)}w) \\
    & = (c+d)(\overline{c}w + \overline{d}w) \\
    & = c(\overline{c}w) + c(\overline{d}w) + d(\overline{c}w) + d(\overline{d}w) \\
    & = (c\overline{c})w + c(\overline{d}w) + d(\overline{c}w) + (d\overline{d})w
\intertext{implies that}
2B(c,d)w &= c(\overline{d}w) + d(\overline{c}w).
\end{align*}
Now, we can again choose $c,w \in V$ and $d=l$, so that
\begin{align*}
0 = 2B(c,l)w &= c(\overline{l}w) + l(\overline{c}w) \\
    & = -c(lw) + l(\overline{c}w) \\
    & = -c(\overline{w}l) + \overline{(\overline{c}w)}l,
\intertext{which implies that}
c(\overline{w}l) &= (\overline{w}c)l. \\
\end{align*}
Our statement follows, and the second statement is obtained by taking conjugates of the first.
\end{proof}

\begin{lemma}
We have the following identities on the element $l \in W$:
\begin{enumerate}
\item $l^2 = \mu$.

\item $\overline{bl} = -bl$ for every $b \in V$.

\item $T_W(bl) = bl + \overline{bl} = 0$.

\item $N_W(bl) = bl\overline{bl} = -blbl = -b\overline{b}l^2 = -N_V(b)\mu$.
\end{enumerate}
\end{lemma}

\begin{proof}
All are immediate from the definitions and previous lemmas.
\end{proof}

\begin{lemma} Suppose that $W$ is a composition algebra; i.e., that $B_W$ is nondegenerate.  Associate
\[ V = \left\{ a + bl \in W \mid b=0 \right\} \quad \text{and} \quad V l = \left\{ a + bl \in W \mid a=0 \right\}.\]
Then $V l = V^{\perp}$.
\end{lemma}

\begin{proof}
For $a \in V$ and $bl \in Vl$, we have:
\begin{align*}
B_W(a , bl) & = N(a+bl) - N(a) - N(bl) \\
    & = (a+bl)\overline{(a+bl)} - a\overline{a} + N_V(b) \\
    & = (a+bl)(\overline{a}-bl) - a\overline{a} + N_V(b) \\
    & = (a\overline{a} + \mu \overline{(-b)}b) + (-ba + ba)l - (a\overline{a}) + N_V(b) \\
    & = 0.
\end{align*}
Therefore $Vl \subseteq V^{\perp}$.  Now, suppose that $w \in V^{\perp}$.  Write $w = a+bl \in V^{\perp} \subset W$.  Then
\[ B_W(w,V) = B_W(a+bl,V) = B_W(a,V) + B_W(bl,V) = B_W(a,V). \]
and this last expression equals zero if and only if $a=0$ by nondegeneracy of $B_W$.  Therefore $V l = V^{\perp}$.
\end{proof}

From these considerations, when the constructed $W$ is a composition algebra, we may write $W = V \oplus V l$, and this sum is both direct and orthogonal.  However, we have not yet determined the conditions under which $N_W$ will be nondegenerate.  The following theorem was originally proven by A. Hurwitz in \cite{Hurwitz1898} in the case of the normed division algebras over $\rr$ but has since been extended to include arbitrary fields.

\begin{theorem}
All composition algebras are obtained by repeated doubling, starting from $k\ee$. Composition algebras of dimension $1$ or $2$ are commutative and associative, those of dimension $4$ are associative but not commutative, and those of dimension $8$ are alternative, but neither commutative nor associative.  The constructed norm $N_W$ on $W$ will be nondegenerate if and only if $V$ was associative. Therefore the possible dimensions of a composition algebra are $1$, $2$, $4$, and $8$.
\end{theorem}

\begin{proof}
For general fields, proofs of these statements may be found in either \cite{Jacobson1958}, \cite{Kaplansky1953}, or \cite{SpringVeld2000}.
\end{proof}

\begin{comment}
(Sketch)\footnote{This could use work.  Also look in [SV] p14.} Starting with any composition algebra $W$ and a finite dimensional, non-isotropic subalgebra $V \subset W$ (for example, $k \ee \subset W$), we identify an $l \in V^{\perp}$ with the desired splitting properties.  Then we construct $V_1 := V \oplus V l$, and show that this is again a finite dimensional non-isotropic subalgebra contained in $W$.  If it is equal to $W$, we are done; if not, we may repeat this process on $V_1$.  At some point however, the process must terminate, for after at most the fourth iteration the subalgebra $V_{i+1} := V_i \oplus V_i l_i$ will be a non-alternative subalgebra of an alternative algebra $W$, which is impossible.
\end{comment}

With this result, for a fixed field $k$ we may now refer to composition algebras of dimension $2$ as \textbf{quadratic} $k$-algebras (and denote them by $K$\label{symbol:K}), composition algebras of dimension $4$ as \textbf{quaternion} $k$-algebras (and denote them by $\mathfrak{B}$\label{symbol:frakB}), and composition algebras of dimension $8$ as \textbf{Cayley} or \textbf{octonion} $k$-algebras (and denote them by $\mathfrak{C}$\label{symbol:frakC}).  We will continue to refer to general composition algebras as either $V$ or $W$.

In our examples of split composition algebras, we see that $k \times k$ can indeed be decomposed in this way; given $(a,b) \in k \times k$, we can write:
\[  (a,b) = \left( \frac{a+b}{2},\frac{a+b}{2} \right) + \left( \frac{a-b}{2},\frac{a-b}{2} \right)(1,-1). \]
The splitting element $l$ in this case is $(1,-1)$, and we can check that for all $a,b \in k$ and corresponding $(a,a),(b,b) \in \eta_1(k) \subset k \times k$,
\begin{align*}
B\bigl( \left( a,a \right) , \left( b,b \right)(1,-1) \bigr) & = \frac{1}{2}T\bigl( \left( a,a \right) \overline{\left( b,-b \right)} \bigr) \\
    & = \frac{1}{2}T\bigl( \left( a,a \right) \left( -b,b \right) \bigr) \\
    & = \frac{1}{2}T\bigl( \left( -ab,ab \right) \bigr) \\
    & = \frac{1}{2} \left( -ab + ab \right) \\
    & = 0.
\end{align*}
Therefore the two images $V,Vl$ are indeed orthogonal.

Also, we may associate $k \times k$ with its image (the diagonal matrices) in $\text{Mat}_2(k)$, and have the following decomposition:
\[  \left(
                   \begin{Array}{cc}
                     a & b \\
                     c & d \\
                   \end{Array}
                 \right) = \left(
                   \begin{Array}{cc}
                     a & 0 \\
                     0 & d \\
                   \end{Array}
                 \right) +\left(
                   \begin{Array}{cc}
                     b & 0 \\
                     0 & c \\
                   \end{Array}
                 \right)\left(
                   \begin{Array}{cc}
                     0 & 1 \\
                     1 & 0 \\
                   \end{Array}
                 \right). \]
The splitting element here is $l = \left(
                   \begin{Array}{cc}
                     0 & 1 \\
                     1 & 0 \\
                   \end{Array}
                 \right)$.

\begin{comment}
Note that it is again orthogonal to all elements $\left(
                   \begin{Array}{cc}
                     a & 0 \\
                     0 & d \\
                   \end{Array}
                 \right) \in \eta_2(k \times k)$.
\end{comment}

Finally, we examine the algebra $\oo$.  Here, the splitting element will depend on the choice of embedding, and for our three chosen embeddings (with $\vec{u}$ equal to either $\ii$, $\jj$, or $\kk$), we use (respectively) the splitting elements:
\[ J := \left(
     \begin{Array}{cc}
       0 & \jj \\
       \jj & 0 \\
     \end{Array}
   \right), \quad K := \left(
     \begin{Array}{cc}
       0 & \kk \\
       \kk & 0 \\
     \end{Array}
   \right), \quad I := \left(
     \begin{Array}{cc}
       0 & \ii \\
       \ii & 0 \\
     \end{Array}
   \right).
 \]
Again, the images $\eta_{\ii}(\text{Mat}_2(k))$ and $\eta_{\ii}(\text{Mat}_2(k))J$ are orthogonal, and respectively for the other choices of embedding.  Explicitly, an element of $\oo$ may be decomposed in any of the following ways:
\begin{align}
\left( \begin{Array}{cc}
       a & \langle v_1,v_2,v_3 \rangle \\
       \langle w_1,w_2,w_3 \rangle & d \\
     \end{Array} \right) & = \left(
     \begin{Array}{cc}
       a & v_1 \ii \\
       w_1 \ii & d \\
     \end{Array}
   \right) + \left(
     \begin{Array}{cc}
       v_2 & w_3 \ii \\
       -v_3 \ii & w_2 \\
     \end{Array}
   \right)\left(
     \begin{Array}{cc}
       0 & \jj \\
       \jj & 0 \\
     \end{Array}
   \right) \\
& = \left(
     \begin{Array}{cc}
       a & v_2 \jj \\
       w_2 \jj & d \\
     \end{Array}
   \right) + \left(
     \begin{Array}{cc}
       v_3 & w_1 \jj \\
       -v_1 \jj & w_3 \\
     \end{Array}
   \right)\left(
     \begin{Array}{cc}
       0 & \kk \\
       \kk & 0 \\
     \end{Array}
   \right) \\
& = \left(
     \begin{Array}{cc}
       a & v_3 \kk \\
       w_3 \kk & d \\
     \end{Array}
   \right) + \left(
     \begin{Array}{cc}
       v_1 & w_2 \kk \\
       -v_2 \kk & w_1 \\
     \end{Array}
   \right)\left(
     \begin{Array}{cc}
       0 & \ii \\
       \ii & 0 \\
     \end{Array}
   \right).
\end{align}
These decompositions will be used extensively in later computations.

\newpage

\chapter{Automorphisms of the Split Octonion Algebra.}
\label{automorphisms}

\section{Generalities}
\label{generalities}

Now the we have defined the split octonion algebra $\oo$, we shift our focus to its automorphism group $\text{Aut}(\oo)$. It will be assumed that any composition subalgebras of $\oo$ that are mentioned here are precisely those split algebras described in our previous examples.  We remind the reader that $k$ may be any field of $\text{char}\, k \neq 2$.

\begin{definition}
Let $W$ be a composition algebra and $V$ be a composition subalgebra of $W$.  We will use $\text{Aut}(W/V)$ to denote the set of $k$-algebra automorphisms of $W$ which act trivially on $V$.  It is clear that these form a subgroup of the group $\text{Aut}(W)$.  However, since the composition algebras $V,W$ will not be fields in general, we will \textbf{not} refer to these as ``Galois groups'' despite the similarities in notation and definition.\footnote{These groups are referred to in this way in \cite{Jacobson1958} and some other sources.}
\end{definition}

Note that any automorphism of any composition algebra $W$ is unital and linear, and thus preserves the embedded copy of $k \subset W$. Therefore we can write $\text{Aut}(W/k)$ or $\text{Aut}(W)$ interchangeably, without loss of meaning.

In order to understand the fairly complicated group $\text{Aut}(\oo/k) = \text{Aut}(\oo)$, it is helpful to consider some of its more notable subgroups.  We list a few important subgroups below, and describe how they arise.  To save on notation, we will denote by $g^{-t}$ the transpose inverse of a matrix $g$.  There will be nothing lost by this abbreviation, since these two operations commute.

\section{A subgroup isomorphic to $\text{SL}_{3}(k)$}
\label{SL3}

We begin with a definition, the claims of which will be verified in the following propositions.

\begin{definition}
The group $\text{SL}_{3}(k)$ is embedded in $\text{Aut}(\oo/k)$ via the morphism\label{symbol:theta}
\[ \theta: \text{SL}_{3}(k) \hookrightarrow \text{Aut}(\oo/k),\]
where, for $g \in \text{SL}_{3}(k)$, the element $\theta(g)$ is defined by
\[ [\theta(g)]\left(
                   \begin{Array}{cc}
                     a & \vec{v} \\
                     \vec{w} & d \\
                   \end{Array}
                 \right) = \left(
                   \begin{Array}{cc}
                     a & g\vec{v} \\
                     g^{-t}\vec{w} & d \\
                   \end{Array}
                 \right).
 \]
\end{definition}

\begin{proposition}
The maps $\theta(g): \oo \rightarrow \oo$ are automorphisms for each $g \in \text{SL}_3(k)$.
\end{proposition}

\begin{proof}
It is easily verified that each $\theta(g)$ is linear, and invertibility follows from $g$ itself being invertible.

To show that $\theta(g)$ is multiplicative, we first note that given any two vectors $\vec{v},\vec{w} \in k^3$, and any invertible matrix $g \in \text{GL}_3(k)$, we have
\begin{align*}
g\vec{v} \circ g^{-t}\vec{w} & =(g\vec{v})^{t} (g^{-t}\vec{w}) \\
    & = \vec{v}^{t} g^{t} (g^{-t})\vec{w} \\
    & = \vec{v}^{t} \vec{w} \\
    & = \vec{v} \circ \vec{w}.
\end{align*}
It is less obvious (though still true) that under the same conditions one has
\[ g\vec{v} \times g\vec{w} = (\det g) (g^{-t}) (\vec{v} \times \vec{w}). \]
However, in our case we assume $g \in \text{SL}_3(k)$, so the above determinant is equal to one.  Armed with these facts, we now verify that
\begin{align*} [\theta(g)] \left(
                   \begin{Array}{cc}
                     a & \vec{v} \\
                     \vec{w} & d \\
                   \end{Array}
                 \right) [\theta(g)] \left(
                   \begin{Array}{cc}
                     \alpha & \vec{\phi} \\
                     \vec{\psi} & \delta \\
                   \end{Array}
                 \right) & = \left(
                   \begin{Array}{cc}
                     a & g\vec{v} \\
                     g^{-t} \vec{w} & d \\
                   \end{Array}
                 \right) \left(
                   \begin{Array}{cc}
                     \alpha & g\vec{\phi} \\
                     g^{-t} \vec{\psi} & \delta \\
                   \end{Array}
                 \right) \\
        & \hspace{-1cm} = \left(
                   \begin{Array}{cc}
                     a\alpha + g\vec{v} \circ g^{-t}\vec{\psi} & a g\vec{\phi} + \delta g\vec{v} - g^{-t}\vec{w} \times g^{-t}\vec{\psi} \\
                     \alpha g^{-t}\vec{w} + \delta g^{-t}\vec{\psi} + g\vec{v} \times g\vec{\phi} & d \delta + g^{-t}\vec{w} \circ g\vec{\phi} \\
                   \end{Array}
                 \right) \\
        & \hspace{-1cm} = \left(
                   \begin{Array}{cc}
                     a\alpha + \vec{v} \circ \vec{\psi} & g(a \vec{\phi} + \delta \vec{v} - \vec{w} \times \vec{\psi}) \\
                     g^{-t} (\alpha \vec{w} + \delta \vec{\psi} + \vec{v} \times \vec{\phi}) & d \delta + \vec{w} \circ \vec{\phi} \\
                   \end{Array}
                 \right) \\
        & \hspace{-1cm} = [\theta(g)] \left( \left(
                   \begin{Array}{cc}
                     a & \vec{v} \\
                     \vec{w} & d \\
                   \end{Array}
                 \right) \left(
                   \begin{Array}{cc}
                     \alpha & \vec{\phi} \\
                     \vec{\psi} & \delta \\
                   \end{Array}
                 \right) \right).
\end{align*}
\end{proof}

\begin{proposition}
The map $\theta$ is an injective morphism of groups.
\end{proposition}

\begin{proof}
This verification is elementary, and requires only the the fact that matrix inversion and transposition commute.
\end{proof}

We can easily see that the action of $\theta$ leaves the diagonal entries of an octonion element unchanged, which means that for $K = k \times k \subset \oo$ we have that $\theta(\text{SL}_3(k)) \subseteq \text{Aut}(\oo/K)$.  In fact, we can do better:

\begin{proposition}
Fix a tower of composition algebras satisfying $k \subset K \subset \mathfrak{B} \subset \mathfrak{C}$.  The image $\theta(\text{SL}_3(k))$ in $\text{Aut}(\oo)$ is equal to $\text{Aut}(\oo/K)$.
\end{proposition}

\begin{proof}
A proof of this may be found on pages $71-72$ of \cite{Jacobson1958}.
\end{proof}
% or in Lausch, p79.

\section{Some subgroups isomorphic to $\text{SL}_{2}(k)$}
\label{LongSL2}

In this definition, we begin with the particular embedding $\eta_{\ii}: \text{Mat}_2(k) \hookrightarrow \oo$, which corresponds to a choice of decomposition of $v \in \oo$ as $v = \eta_{\ii}(a)+\eta_{\ii}(b)J$, as discussed in Section \ref{decompositions}.  We will discuss the other embeddings of $\text{Mat}_2(k)$ into $\oo$ later on.

\begin{definition}
For the choice of embedding $\eta_{\ii}$, we have a related embedding of groups\label{symbol:gamma}
\[ \gamma_{\ii}: \text{SL}_2(k) \hookrightarrow \text{Aut}(\oo/k), \]
where $\gamma_{\ii}(g)$ is defined in the following way: after decomposing $v \in \oo$ into $v = \eta_{\ii}(a)+\eta_{\ii}(b)J$, with $a,b \in \text{Mat}_2(k)$, apply the element $g$:
\[ [\gamma_{\ii}(g)] \bigl( \eta_{\ii}(a)+\eta_{\ii}(b)J \bigr) = \eta_{\ii}(a)+\eta_{\ii}(gb)J. \]
\end{definition}

\begin{proposition}
The maps $\gamma_{\ii}(g): \oo \rightarrow \oo$ are automorphisms for each $g \in SL_2(k)$.
\end{proposition}

\begin{proof}
Again, it is easily verified that $\gamma_{\ii}(g)$ is linear, and invertibility follows from $g$ itself being invertible.  That it is also multiplicative depends upon the fact that $g \in \text{SL}_2(k)$ and $\det g = 1$, for in such a case the involution on $\text{Mat}_2(k) = \text{Mat}_2(k)$ corresponds to inversion: $\overline{g} = g^{-1}$.  Therefore,
\begin{align*}
 [\gamma_{\ii}(g)]((a+bl)(c+dl)) & = [\gamma_{\ii}(g)]((ac + \overline{d}b) + (da+bc)l) \\
    & = (ac + \overline{d}b) + (g(da+bc))l \\
    & = (ac + \overline{d} (g^{-1} g) b) + (g(da)+g(bc))l \\
    & = (ac + (\overline{d} \overline{g}) (g b)) + ((gd)a+(gb)c)l \\
    & = (ac + \overline{(gd)}(g b)) + ((gd)a+(gb)c)l \\
    & = (a + (g b)l)(c + (g d)l) \\
    & = [\gamma_{\ii}(g)](a+bl) \cdot [\gamma_{\ii}(g)](c+dl).
\end{align*}
\end{proof}

\begin{proposition}
The map $\gamma_{\ii}$ is an isomorphism of groups:
\[ \gamma_{\ii} \colon \text{SL}_2(k) \rightarrow \text{Aut}\bigl( \oo/\eta_{\ii}(\text{Mat}_2(k) )\bigr). \]
\end{proposition}

 \begin{proof}
 Since the matrices $a,b,g$ all associate, $\gamma_{\ii}$ is a homomorphism of groups from $\text{SL}_2(k) \rightarrow \text{Aut}(\oo)$.  It is obvious that the map $\gamma_{\ii}$ has trivial kernel.  Also, from the definition of $\gamma_{\ii}$ the element $g$ acts only on the second component of the decomposed $a + bl$, and therefore acts trivially on the image $\eta_{\ii}(\text{Mat}_2(k) )$ in $\oo$.  Therefore, $\text{SL}_2(k) \subseteq \text{Aut}\bigl(\oo/\text{Mat}_2(k)\bigr)$.

 Conversely, let $\phi \in \text{Aut}\bigl(\oo/\eta_{\ii}(\text{Mat}_2(k) )\bigr)$.  We already have that $\oo \cong \eta_{\ii}(\text{Mat}_2(k) ) \oplus \eta_{\ii}(\text{Mat}_2(k) )l$ for some element $l \in \text{Mat}_2(k)^{\perp}$ with $N(l) \neq 0$.  Since algebra automorphisms preserve norms and thus orthogonal subspaces, $\phi$ maps $\text{Mat}_2(k)^{\perp} \cong \eta_{\ii}(\text{Mat}_2(k) )l$ into itself, and therefore $\phi(l) = ul$, for some $u \in \text{Mat}_2(k)$.  In fact, $N(l) = N(\phi(l)) = N(ul) = N(u)N(l)$ implies that $N(u)=1$.  Therefore, for $a+bl \in \oo$,
 \[ \phi(a+bl) = \phi(a) + \phi(b)\phi(l) = a + b(ul) = a + (ub)l. \]
 The last equality in the above follows from Lemma \ref{associationlemma}.
 \end{proof}
% Springer/Veldkamp do this in a way independent of embeddings.

Recall that this definition began with a choice of decomposition $a+bl$ in the sense of the last section.  Likewise, we can start with any embedding of $\text{Mat}_2(k)$ into $\oo$, and end up with a similarly defined embedding of $\text{SL}_2(k)$ into $\text{Aut}(\oo)$.  All of the verifications we have just performed can also be verified for other embeddings.  In this way, each embedding of $\text{Mat}_2(k)$ into $\oo$ will yield a distinct copy of $SL_2$ into $\text{Aut}(\oo)$.

\begin{comment}
It should also be noted that $\text{Aut}(\oo / \text{Mat}_2(k)) \subset \text{Aut}(\oo / K)$, regardless of the embedding of $\text{Mat}_2(k)$, so all automorphisms of this type are actually contained in those of Section $3.2$.  (Marty:  Only if $K \subseteq \text{Mat}_2(k)$. More accurately, given a tower of composition algebras
\[ k \hookrightarrow K \hookrightarrow \mathfrak{B} \hookrightarrow \mathfrak{C}, \]
we have containments
\[ \text{Aut}(\mathfrak{C}/\mathfrak{B}) \subseteq \text{Aut}(\mathfrak{C}/K) \subseteq \text{Aut}(\mathfrak{C}). \]
\end{comment}

\section{Some (other) subgroups isomorphic to $\text{SL}_{2}(k)$}
\label{ShortSL2}

For each embedding of $\text{Mat}_2(k)$ into $\oo$ we now describe another, distinct way of embedding $\text{SL}_2(k)$ into $\text{Aut}(\oo)$.  Again, we describe the construction with respect to the embedding $\eta_{\ii}$, though there is no reason to prefer this choice.
\begin{definition}
For the choice of embedding $\eta_{\ii}$, we have a related embedding of groups\label{symbol:delta},
\[ \delta_{\ii}: \text{SL}_2(k) \hookrightarrow \text{Aut}(\oo/k). \]

Here, after decomposing an element $v \in \oo$ into $v = \eta_{\ii}(a)+\eta_{\ii}(b)J$, with $a,b \in \text{Mat}_2(k)$ as before, we apply the element $g$ in a different way:
\[ [\delta_{\ii}(g)](\eta_{\ii}(a)+\eta_{\ii}(b)J) = \eta_{\ii}(gag^{-1})+\eta_{\ii}(bg^{-1})J. \]
\end{definition}

That $\delta_{\ii}(\text{SL}_2(k)) \subseteq \text{Aut}(\oo/k)$, and that $\delta_{\ii}$ is an injective morphism of groups, are both verifications very similar to those of Section \ref{LongSL2}.  Again, all verifications hold regardless of the choice of initial decomposition.

Unlike the previous cases however, we note that $\delta_{\ii}(g)$ will act trivially on $a \in \text{Mat}_2(k)$ for all $g \in \text{SL}_2(k)$ if and only if $a$ is a scalar matrix.  Thus, the automorphisms constructed in this way all preserve the field $k$ (or rather, its isomorphic copy in $\oo$; after all, they are linear), but do not preserve $K$ or any copy of $\text{Mat}_2(k)$, and this copy of $\text{SL}_2(k)$ is not subgroup of the form $\text{Aut}(W/V)$ as described before.

\begin{comment}
\[ \text{SO}_4(\zz_p) \cong SL_2(\zz_p) \times SL_2(\zz_p) \hookrightarrow G_2(\qq_p) \]
\end{comment}

\section{Matters of Notation}
\label{notation}

While the embedding of $SL_3(k)$ into $G_2$ was relatively straightforward, the various embeddings of $SL_2(k)$ required us to make a number of choices.  We wish to create a notation for these associated automorphisms that will be even more descriptive, and be consistent with notation that we will use in later chapters.  First, note that so long as only one choice of embedding $\text{Mat}_2(k) \hookrightarrow \oo$ is used, composition of the automorphisms of type $\gamma$ and $\delta$ will simply correspond to matrix multiplication, since all of the involved matrices will be part of the same embedded copy of $\text{Mat}_2(k) \subset \oo$, and will associate.

Second, we note that $SL_2(k)$ is generated by matrices of the form:
\[ \left(
    \begin{Array}{cc}
        1 & x \\
        0 & 1 \\
    \end{Array}
\right)  \text{ and }  \left(
    \begin{Array}{cc}
        1 & 0 \\
        y & 1 \\
    \end{Array}
\right), \text{ with } x,y \in k. \]
Therefore, each embedded copy $\gamma(SL_2) \subset \text{Aut}(\oo)$ is likewise generated by those automorphisms of $\oo$ which act by these matrices, that is, by the automorphisms:

\[ \gamma\left(
    \begin{Array}{cc}
        1 & x \\
        0 & 1 \\
    \end{Array}
\right) \text{ and } \gamma\left(
    \begin{Array}{cc}
        1 & 0 \\
        y & 1 \\
    \end{Array}
\right), \text{ with } x,y \in k. \]
Each embedded copy of $\delta(SL_2)$ is also generated by the automorphisms analogous to these.

\begin{definition}
\label{additivity}
We will denote the following embeddings of the additive group $k^+$ into the multiplicative group $SL_2(k)$ as follows:
\[
e^{+}:  s \mapsto \left(
                                \begin{Array}{cc}
                                  1 & s \\
                                  0 & 1 \\
                                \end{Array}
                              \right)
\quad \text{and} \quad
e^{-}:  s \mapsto \left(
                                \begin{Array}{cc}
                                  1 & 0 \\
                                  s & 1 \\
                                \end{Array}
                              \right).
\]
So, for a fixed $\eta: \text{Mat}_2(k) \hookrightarrow \oo$, we consolidate all of our notation\label{symbol:gammapm}\label{symbol:deltapm} by defining the following maps from additive group $k^+$ to the group $\text{Aut}(\oo)$:
\begin{center}
\begin{tabular}{cc}
  $\gamma^{+} = \gamma \circ e^{+}$, \; & \;$\delta^{+} = \delta \circ e^{+}$, \\
  $\gamma^{-} = \gamma \circ e^{-}$, \; & \;$\delta^{-} = \delta \circ e^{-}$. \\
\end{tabular}
\end{center}
\end{definition}

Defined by a composition of injective morphisms, the maps above are all themselves injective morphisms from the additive group $k^+$ to the group $\text{Aut}(\oo)$, and there are $12$ such maps. As an example, with $x \in k$, the automorphism $\delta^{-}_{\jj}(s)$ acts on an element $a+bK \in \oo$ as
\[ \delta^{-}_{\jj} (x) (a+bK) = \left(\left(\begin{Array}{cc}
        1 & 0 \\
        s & 1 \\
    \end{Array}
\right) \biggl(\: a \:\biggl) \left(\begin{Array}{cc}
        1 & 0 \\
        -s & 1 \\
    \end{Array}
\right)\right) + \left( \biggl(\: b \:\biggl) \left(\begin{Array}{cc}
        1 & 0 \\
        -s & 1 \\
    \end{Array}
\right)\right)\left(\begin{Array}{cc}
        0 & \kk \\
        \kk & 0 \\
    \end{Array}
\right). \]

It is precisely these automorphisms $\gamma^{\pm}(s)$ and $\delta^{\pm}(t)$, for $s,t \in k$, that will act as our Chevalley generators in the remaining chapters.  The automorphisms $\theta(g)$ corresponding to the $SL_3(k)$ subgroup will make a significant contribution as well.

\newpage
\chapter{Chevalley Groups}
\label{chevalleygroups}

\section{Background}
\label{background}

Historically, the identification of many abstract simple groups was done via a careful consideration of simple \emph{complex} Lie groups.  For example, if one defines the group $SL_2(\cc)$ to be the group of two-by-two matrices of determinant one, then the field $\cc$ may be easily replaced by any field or commutative ring (with identity) without alteration of the definition.

But this strategy requires one to construct and study each simple Lie group individually.  Moreover, many other groups which were known (at least formally) to exist via an association to a Lie algebra or root system, failed to admit such a simple definition.  Theoretically, all such algebraic groups can be considered as subgroups of $GL_n$ satisfying some polynomial conditions on the entries, though the size of $n$ and the nature of the polynomial conditions proved to be elusive.  At the time of Chevalley's work in \cite{Chevalley1955}, only the exceptional group $G_2$ had admitted\footnote{The group $G_2$ was treated for the first time in \cite{Dickson1901} and \cite{Dickson1905}.  In Chevalley's introduction of \cite{Chevalley1955}, he states that he has himself applied this method to the groups $F_4$, $E_6$, and $E_7$ in an unedited and unpublished work.} such an identification.

The strength of Chevalley's construction is that it allows us to study all groups which arise from an abstract root system \emph{simultaneously}, regardless of their association to any particular Lie algebra or field.  In its basic form, this construction will begin with a root system, and then identify a set of generators for the group which are in families corresponding to the roots in the root diagram.  One then imposes a set of relations on these generators which will yield the desired group structure.  These relations are explicitly described in \cite{Steinberg1967}, from which the main part of this section is taken.  A group constructed in this manner will be called a \textbf{Chevalley group}.

Our goal in this chapter is to justify the use of the name $G_2$ for $\text{Aut}(\oo)$ by identifying the generators of $\text{Aut}(\oo)$ established in the last section with the Chevalley generators coming from the root diagram of the Lie Algebra $\frak{g}_2$, and then verifying that our generators satisfy the needed Chevalley relations.

\section{Abstract Root Systems}
\label{abstractrootsystems}

To define a Chevalley group $G$, we require only an abstract, reduced root system.  We therefore proceed with a few definitions.

\begin{definition}[\cite{Knapp2002}]
Let $E$\label{symbol:E} be a finite-dimensional real inner product space with inner product $( \cdot , \cdot )$\label{symbol:rootinnerproduct} and norm squared $\left| \,\cdot\, \right|^2$\label{symbol:rootnorm}.  An \textbf{abstract, reduced root system} in $E$ is a finite set $\Phi$\label{symbol:Phi} of nonzero elements of $E$ such that
\begin{enumerate}
\item $\Phi$ spans $E$.
\item If $\alpha$ and $\beta$ are in $\Phi$, then $\displaystyle \frac{2( \beta , \alpha )}{\left| \alpha \right|^2}$ is an integer.
\item For each $\alpha, \beta \in \Phi$, the reflection $\displaystyle s_{\alpha}(\beta) := \beta - \frac{2( \beta , \alpha )}{\left| \alpha \right|^2}\alpha$ is also a root in $\Phi$.
\item For each $\alpha \in \Phi$, if $n\alpha \in \Phi$ then $n = \pm 1$.
\end{enumerate}
\end{definition}

We may choose a decomposition of $\Phi$ into \textbf{positive}\label{symbol:Phiplus} and \textbf{negative} roots $\Phi = \Phi^+ \sqcup \Phi^-$.  The way that we choose this decomposition is unimportant, so long as the positive roots satisfy the following properties:
\begin{enumerate}
\item For any nonzero $\alpha \in \Phi$, exactly one of $\alpha$ and $-\alpha$ is positive.
\item The sum of positive roots is positive (when it is a root at all).
\end{enumerate}

Then, among the positive roots we can find the set of \textbf{simple} roots\label{symbol:Delta} $\Delta$, which are defined by the property that if $\delta \in \Delta$, then $\delta$ cannot be written $\delta = \alpha + \beta$ for $\alpha,\beta \in \Phi^+$.

\begin{definition}
Let $W$ denote the \textbf{Weyl group}\label{symbol:W} of isometries of $E$, generated by the elements $s_{\alpha}$\label{symbol:reflection} (from root axiom (c)) which reflect a root $\beta$ across the hyperplane orthogonal to $\alpha$ in $E$.
\end{definition}

\begin{definition}
An abstract root system is said to be \textbf{reduced} if $\alpha \in \Delta$ implies $2\alpha \notin \Delta$.
\end{definition}

After making some choice of simple roots $\Delta = \{\alpha_{1},\alpha_{2},\dots,\alpha_{n}\}$ (here $n = \dim E$), we define the factors
\[ A_{ij} = \frac{2( \alpha_{i} , \alpha_{j} )}{\left| \alpha_{i} \right|^2}. \]

From the root axiom (b), we are guaranteed that $A_{ij} \in \zz$.  Then, the \textbf{Cartan matrix} of $\Phi$ (with respect to $\Delta$) is defined to be the $n \times n$ matrix $A = \left( A_{ij} \right)$.  Although this matrix will depend not only on our choice of simple roots, but also on their enumeration, distinct choices will lead to matrices which are conjugate by some permutation matrix.

The combinatorial properties imposed on the entries of this matrix by the root axioms lead to the classification of all such Cartan matrices, and to the notion of the Dynkin diagram.  The classification of all Dynkin diagrams, and therefore all abstract reduced root systems, is well-known and can be found in references such as \cite{Humphreys1972} or \cite{Knapp2002}.

\section{The Chevalley Construction}
\label{chevconstruction}

We build our Chevalley group\label{symbol:ChevGroup} $G = G_{\Phi}(k)$ from a root system $\Phi$ and a field $k$.  The Chevalley generators of $G$ will each be elements of the form $e_{\alpha}(t)$\label{symbol:ChevElements}, parameterized by $t \in k$ and $\alpha \in \Phi$.  From these basic generators one can also construct the following secondary elements:

\begin{definition}
For $t \in k^{\times}$, define:
\begin{align*}
w_{\alpha}(t) & := e_{\alpha}(t)e_{-\alpha}(-t^{-1})e_{\alpha}(t), \\
h_{\alpha}(t) & := w_{\alpha}(t)w_{\alpha}(-1).
\end{align*}
\end{definition}

Having these generators and group elements, we describe the relations which define the group $G$, taken from \cite[pg $66$]{Steinberg1967}:

\begin{enumerate}
\item The $e_{\alpha}$ are each homomorphisms from the additive group of $k$ into $G$, that is:
\[ e_{\alpha}(s+t) = e_{\alpha}(s)e_{\alpha}(t) \text{ for all } s,t \in k. \]

\item If $\alpha, \beta \in \Phi$ with $\alpha + \beta \neq 0$, then
\[ [e_{\beta}(t) , e_{\alpha}(s)] = \prod e_{i \alpha + j \beta}(N_{ij} s^i t^j), \]
where the product is taken over all (strictly) positive integers $i,j \in \zz$ such that $i \alpha + j \beta \in \Phi$, and the $N_{ij}$\label{symbol:ChevalleyStructureConstants} are each integers depending on $\alpha,\beta$, but not on $s,t$.

\item Each $h_{\alpha}$ is multiplicative in $k^{\times}$; i.e., $h_{\alpha}(s)h_{\alpha}(t) = h_{\alpha}(st)$ for all $s,t \in k^{\times}$.
\end{enumerate}

Though $G$ is presently only an abstract group, devoid of any topological or geometric properties, the relations just given still impart a significant amount of structure onto the group.  For example, one can show that structure constants $N_{ij}$ described in (b) will be equal to either $\pm1$,$\pm2$, or $\pm3$ for any initial $\Phi$ or choice of $\Delta$.  There is great deal more theory which comes directly from the Chevalley construction alone, the best treatments of which can be found in \cite{Chevalley1955} and \cite{Steinberg1967}.

From our generators, we define the following subgroups:

\begin{definition}[\cite{Rabinoff2003}]  % pg 11.
If $G = G_{\Phi}(k)$ is a Chevalley group constructed as we have just described, then \label{symbol:Torus}\label{symbol:Normalizer}
\begin{gather*}
T_{\Phi} := \bigl\langle\, h_{\alpha}(t) \,\bigr\rangle, \\
N_{\Phi} := \bigl\langle\, w_{\alpha}(t) \,\bigr\rangle,
\end{gather*}
where $\alpha \in \Phi$ and $t \in k^{\times}$.  Then $T_{\Phi} \leq G$ is called the \textbf{Cartan subgroup} of $G$.
\end{definition}

The subgroup $N_{\Phi}$ will also play a role, as evidenced by the following proposition.

\begin{proposition}    % Also Rabinoff, pg 14.
Let $T_{\Phi}$ and $N_{\Phi}$ be as described above, and $W$ be the Weyl group of $\Phi$ described in Definition $13$.  Then:
\begin{enumerate}
\item $T_{\Phi}$ is normal in $N_{\Phi}$.     %In fact, $N = N_G(T)$ if $k$ has more than three elements.
\item The map $\phi: N_{\Phi} \rightarrow W$ that sends $w_{\alpha}(t)$ to the reflection $s_{\alpha}$ induces an isomorphism of $N_{\Phi}/T_{\Phi}$ with $W$.
\end{enumerate}
\end{proposition}

\begin{proof}
Proof of this may be found in \cite{Steinberg1967}, Lemma $22$.
\end{proof}

\section{When More is Known}
\label{moreinfo}

Though the theory of abstract Chevalley groups is extensive, quite a bit more can be said if we associate to the root system its appropriate Lie group or Lie algebra.  For example, at the end of the last section, we were able to define the subgroup $T_{\Phi}$ and the Weyl group $N_{\Phi}/T_{\Phi}$ of $G = G_{\Phi}$ without any topological references whatsoever.  Of course, in the case that $k=\cc$ and $G$ is a semisimple complex Lie group, $T_{\Phi}$ will be the familiar maximal torus of $G$, and its Lie algebra will be a Cartan subalgebra of $\mathfrak{g}$.  These considerations are unnecessary in the abstract case, but in this setting, one often considers the roots to be homomorphisms defined on such a maximal torus.  We will describe these homomorphism in Chapter \ref{standardapartment} where we define the coroot system.

In the case that a root system $\Phi$ can be identified with a well-known algebraic group \textbf{G}, then we might identify known generators of $G$ which are in correspondence with our Chevalley generators $e_{\alpha}(t)$.  Describing the precise correspondence between various generators and roots usually involves some (sometimes lengthy) calculation.  A simple example of such an identification is described in Section \ref{example}.

On the other hand, if one attaches to the root system its appropriate semisimple complex Lie algebra $\mathfrak{g}$, there is an explicit connection of the root $\alpha$ to the generator $e$ via the exponential map.  One can choose a root space decomposition
\[ \mathfrak{g} = \mathfrak{h} \bigoplus_{\alpha \in \Phi} \mathfrak{g}_{\alpha} \]
with a suitable choice of root vectors $E_{\alpha}$ from their respective root spaces that satisfy certain algebraic properties. Chevalley showed in \cite{Chevalley1955} that such a basis $\{ H_{\delta}, E_{\alpha} \}$ of $\mathfrak{g}$ (appropriately called a \textbf{Chevalley basis}) always exists.  One can then formally exponentiate each basis element $E_{\alpha}$ against an invariant $t$ to arrive at elements $e_{\alpha}(t): = \exp(tX_{\alpha})$, which will necessarily satisfy the Chevalley group relations and thus generate\footnote{This is true so long as $k$ is algebraically closed or we generate the simply-connected form.  The latter will be true in our case.} the group $G_{\Phi}$.  This group $G_{\Phi}$ is precisely the Chevalley group corresponding to $\mathfrak{g}$.  Since $t$ is an invariant which can be taken from any field, this approach is general.  We will discuss this exponentiation strategy in more detail in Chapter \ref{G2generators}.

\section{An Example}
\label{example}

The group $SL_2(k)$ can be constructed from rank one root diagram having two roots, $\Phi = \{ \alpha, -\alpha \}$, which is shown in Figure \ref{rootsA} along with its single line of reflection.  However, we have already discussed (in Section \ref{notation}) two generators of $SL_2(k)$, when considered as group of matrices.  Therefore we can make the association, for $t \in k$:
\[ e_{\alpha}(t) = \left(
     \begin{Array}{cc}
       1 & t \\
       0 & 1 \\
     \end{Array}
   \right), \quad e_{-\alpha}(t) = \left(
     \begin{Array}{cc}
       1 & 0 \\
       t & 1 \\
     \end{Array}
   \right).
 \]

\begin{figure}[ht!]
\begin{center}
\begin{tikzpicture}[thin,scale=2]%
    \draw (0:-1) -- (0:1);
    \draw[dashed] (90:0.5) -- (270:0.5);
    \draw{   (0:1) node[circle, draw, fill=black!50,
                        inner sep=0pt, minimum width=4pt, label=right:$\alpha$]{}
            (0:-1) node[circle, draw, fill=black!50,
                        inner sep=0pt, minimum width=4pt, label=left:$-\alpha$]{}
            (0:0) node[circle, draw, fill=black!50,
                        inner sep=0pt, minimum width=4pt]{}
};
\end{tikzpicture}
\caption{The root diagram of type $A_1$.}\label{rootsA}
\end{center}
\end{figure}
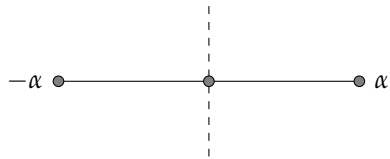

From these generators, we calculate that
\begin{align*}
w_{\alpha}(t) & = e_{\alpha}(t)e_{-\alpha}(-t^{-1})e_{\alpha}(t) \\
    & =  \left(
                           \begin{Array}{cc}
                             1 & t \\
                             0 & 1 \\
                           \end{Array}
                         \right)\left(
                           \begin{Array}{cc}
                             1 & 0 \\
                             -t^{-1} & 1 \\
                           \end{Array}
                         \right)\left(
                           \begin{Array}{cc}
                             1 & t \\
                             0 & 1 \\
                           \end{Array}
                         \right) \\
    & = \left(
                           \begin{Array}{cc}
                             0 & t \\
                             -t^{-1} & 0 \\
                           \end{Array}
                         \right), \\
h_{\phi}(t) & = w_{\phi}(t)w_{\phi}(-1) \\
    & = \left(
                           \begin{Array}{cc}
                             0 & t \\
                             -t^{-1} & 0 \\
                           \end{Array}
                         \right)\left(
                           \begin{Array}{cc}
                             0 & -1 \\
                             1 & 0 \\
                           \end{Array}
                         \right) \\
    & = \left(
                           \begin{Array}{cc}
                             t & 0 \\
                             0 & t^{-1} \\
                           \end{Array}
                         \right).
\end{align*}
Note that the Chevalley relations (a) and (c) are easily verified, while relation (b) is trivial in this case.  We see that the $h_{\alpha}(t)$ indeed generate the maximal torus of diagonal matrices in $SL_2(k)$.  We can also calculate that $W = N_{\Phi}/T_{\Phi} = \zz/2\zz$, which corresponds to our intuition from the root diagram $A_1$.

One may consult \cite{Rabinoff2003} for a similar discussion of the group $\text{Sp}_4(k)$, corresponding to the root system $C_2$.  In our next chapter we will make an identification of Chevalley generators $e_\alpha(t)$ to automorphisms in $\text{Aut}(\oo)$, though in this case it will not be so obvious that our choices of automorphisms actually generate this entire automorphism group.  We will therefore need to apply some extra argument.

\chapter{Chevalley Generators of $G_2$}
\label{G2generators}

\section{Background and Strategy}

As mentioned in Section \ref{moreinfo}, one strategy for identifying Chevalley generators of a Chevalley group $G$ is to first identify a Chevalley basis of the corresponding Lie algebra $\mathfrak{g}$.  For us, with $\mathfrak{g} = \mathfrak{g}_2$, we can allow a choice of simple roots $\Delta = \{ \gamma,\delta \}$, with $\delta$ the short root as shown in Figure \ref{rootsG2}, and a corresponding root space decomposition.

\begin{figure}[ht!]
\begin{center}
\begin{tikzpicture}[thin,scale=1.5]%
    \draw \foreach \x in {0,120,240} {
        (\x+90:1.732) -- (\x+210:1.732)
        (\x-90:1.732) -- (\x+30:1.732)
        (0:0) -- (0:1)
        (0:0) -- (150:1.732)
        (0:0) node[circle, draw, fill=black!50,
                        inner sep=0pt, minimum width=4pt]{}
        (0:1) node[circle, draw, fill=black!50,
                        inner sep=0pt, minimum width=4pt, label=right:$\delta$]{}
        (150:1.732) node[circle, draw, fill=black!50,
                        inner sep=0pt, minimum width=4pt, label=above left:$\gamma$]{}
};
\end{tikzpicture}
\caption{The root diagram of type $G_2$.}\label{rootsG2}
\end{center}
\end{figure}

\begin{definition}
A \textbf{Chevalley basis}\label{symbol:ChevBasis} for the Lie algebra $\mathfrak{g} = \mathfrak{g}_2$ is a basis
\[ \mathfrak{b} = \left\{ H_\gamma,H_\delta,E_\phi \right\}_{\phi \in \Phi} \]
which satisfies the following axioms:\label{symbol:Ealpha}\label{symbol:Halpha}
\begin{enumerate}
\item $[H_\gamma,H_\delta] = 0$.

\item $[H_\gamma,E_\phi] = \frac{2( \gamma, \phi )}{(\phi,\phi)} E_\phi$, and $[H_\delta,X_\phi] = \frac{2( \delta, \phi )}{(\phi,\phi)} E_\phi$.

\item $[E_\phi,E_{-\phi}] = aH_\gamma + bH_\delta$, where $\phi = a\gamma + b\delta$.

\item $[E_\phi,E_\psi] = 0$, if $\phi + \psi \neq 0$ and $\phi + \psi \notin \Phi$.

\item If $\phi + \psi \neq 0$ and $\phi + \psi \in \Phi$, and $r$ is the integer occurring in the $\phi$-string through $\psi$ given by $\psi - r\phi, \dots, \psi, \dots, \psi + q\phi$, then
    \[ [E_\phi,E_\psi] = \pm (r+1) E_{\phi + \psi}. \]
\end{enumerate}
\end{definition}

Exponentiating each basis element $E_{\phi}$ will give group elements $e_\phi(t) = \exp(t \cdot E_\phi)$ in $G_2^{Chev}$ and exponentiating the Chevalley basis relations will show that the $e_\phi(t)$ necessarily satisfy the Chevalley group relations.  Examples of this construction can be found in either \cite{Chevalley1955} or \cite{Steinberg1967}.

However, in our case we already have candidates for our Chevalley generators, the $\gamma^\pm$ and $\delta^\pm$ automorphisms, and we will soon show directly that they satisfy the Chevalley group relations.  Therefore we would like to use a slightly reversed argument.  We will rather describe a specific Lie algebra $\mathfrak{g}$ of Cartan type $G_2$, and a basis of this algebra, such that when our basis is exponentiated it gives precisely our automorphisms $\gamma^\pm$ and $\delta^\pm$.  Since we will have verified the Chevalley group relations, our choice of basis will necessarily be a Chevalley basis in $\mathfrak{g}$, and our $\gamma^\pm$ and $\delta^\pm$ automorphisms will generate a Chevalley group of type $G_2$.

To summarize, we will have the following structures:
\begin{enumerate}
\item A Lie algebra $\mathfrak{g}_2$.

\item A Chevalley basis $\mathfrak{b}$ of $\mathfrak{g}_2$.

\item The Chevalley group $G_2$ generated from exponentiated Chevalley basis elements.

\item The automorphism group $\text{Aut}(\oo)$.

\item The subgroup $G_{\gamma,\delta} \subseteq \text{Aut}(\oo)$ generated by the automorphisms $\gamma^\pm$ and $\delta^{\pm}$.
\end{enumerate}

Our arguments will show that $G_{\gamma,\delta} \cong G_2$.  The fact that $G_{\gamma,\delta}$ is actually all of $\text{Aut}(\oo)$ will be addressed using a result of George Seligman from \cite{Seligman1960}.  Seligman uses as $\mathfrak{g}_2$ the algebra of derivations of $\oo$, and therefore we will follow suit.

\section{Associations and Verifications}
\label{sec:associationsverifications}

We note again that the root diagram of type $G_2$ shows $12$ roots, $6$ long and $6$ short.  Similarly, we also have the $12$ maps of type $\gamma$ and $\delta$, each of which are homomorphisms from the additive group $k^+$ into the group $\text{Aut}(\oo)$.  Therefore we can identify each $\gamma$ with a long root in $\Phi$, and each $\delta$ with a short root.  This identification is shown in Figure \ref{autosandrootsG2}.  We should note that the relative associations are precise; once a choice of association is made between the two simple roots and particular automorphisms, the rest follow necessarily.

\begin{figure}[ht]
\begin{center}
\begin{tikzpicture}[thin,scale=1.5]%
    \draw \foreach \x in {0,120,240} {
        (\x+90:1.732) -- (\x+210:1.732)
        (\x-90:1.732) -- (\x+30:1.732)
        (0:0) -- (0:1)
        (0:0) -- (150:1.732)
        (0:0) node[circle, draw, fill=black!50,
                        inner sep=0pt, minimum width=4pt]{}
        (0:1) node[circle, draw, fill=black!50,
                        inner sep=0pt, minimum width=4pt, label=right:$\delta_{\ii}^{+}$]{}
        (60:1) node[circle, draw, fill=black!50,
                        inner sep=0pt, minimum width=4pt, label=above right:$\delta_{\kk}^{-}$]{}
        (120:1) node[circle, draw, fill=black!50,
                        inner sep=0pt, minimum width=4pt, label=above left:$\delta_{\jj}^{+}$]{}
        (180:1) node[circle, draw, fill=black!50,
                        inner sep=0pt, minimum width=4pt, label=left:$\delta_{\ii}^{-}$]{}
        (240:1) node[circle, draw, fill=black!50,
                        inner sep=0pt, minimum width=4pt, label=below left:$\delta_{\kk}^{+}$]{}
        (300:1) node[circle, draw, fill=black!50,
                        inner sep=0pt, minimum width=4pt, label=below right:$\delta_{\jj}^{-}$]{}
        (30:1.732) node[circle, draw, fill=black!50,
                        inner sep=0pt, minimum width=4pt, label=above right:$\gamma_{\jj}^{-}$]{}
        (90:1.732) node[circle, draw, fill=black!50,
                        inner sep=0pt, minimum width=4pt, label=above:$\gamma_{\ii}^{+}$]{}
        (150:1.732) node[circle, draw, fill=black!50,
                        inner sep=0pt, minimum width=4pt, label=above left:$\gamma_{\kk}^{-}$]{}
        (210:1.732) node[circle, draw, fill=black!50,
                        inner sep=0pt, minimum width=4pt, label=below left:$\gamma_{\jj}^{+}$]{}
        (270:1.732) node[circle, draw, fill=black!50,
                        inner sep=0pt, minimum width=4pt, label=below:$\gamma_{\ii}^{-}$]{}
        (330:1.732) node[circle, draw, fill=black!50,
                        inner sep=0pt, minimum width=4pt, label=below right:$\gamma_{\kk}^{+}$]{}
};
\end{tikzpicture}
\caption{The association of the automorphisms of $\oo$ to the roots of $G_2$.}\label{autosandrootsG2}
\end{center}
\end{figure}
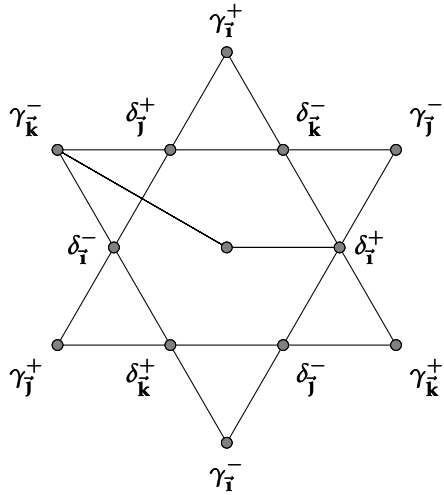

\begin{proposition}
\label{prop:autoverify}
The automorphisms $\gamma^\pm$ and $\delta^\pm$, as placed in Figure \ref{autosandrootsG2}, satisfy the Chevalley group relations.
\end{proposition}

\begin{proof}
First, each automorphism $\gamma$ and $\delta$ is a homomorphism from the additive group of $k$ to the multiplicative group of $\text{Aut}(\oo)$ by Definition \ref{additivity}.

According to the third Chevalley relation, we wish to show that the automorphisms corresponding to the $h_{\alpha}$ are multiplicative.  Since $w_{\alpha}$ and $h_{\alpha}$ are defined in terms of a single $\alpha \in \Phi$, the associated automorphisms will be built using a single choice of embedding of $\text{Mat}_2(k)$ into $\oo$.  Therefore composition of the involved automorphisms just corresponds to multiplication of the involved matrices.  Moreover, all matrices involved will associate with each other.

For example, if $e_{\alpha} = \delta^{+}_{\ii}$, then $w_{\alpha}(t)$ corresponds to an automorphism which acts on an element of $\oo$, split with respect to $\ii$:
\begin{align*}
[w_{\alpha}(t)](a+bJ) & = a + (e_{\alpha}(t)e_{-\alpha}(-t^{-1})e_{\alpha}(t)b)J \\
    & = a + \left( \left(
                           \begin{Array}{cc}
                             1 & t \ii \\
                             0 & 1 \\
                           \end{Array}
                         \right)\left(
                           \begin{Array}{cc}
                             1 & 0 \\
                             -t^{-1} \ii & 1 \\
                           \end{Array}
                         \right)\left(
                           \begin{Array}{cc}
                             1 & t \ii \\
                             0 & 1 \\
                           \end{Array}
                         \right) b \right) J \\
    & = a + \left( \left(
                           \begin{Array}{cc}
                             0 & t \ii \\
                             -t^{-1} \ii & 0 \\
                           \end{Array}
                         \right) b \right)J. \\
\end{align*}
Likewise, in this case $h_{\alpha}$ will correspond to an action by a matrix of the form
\begin{align*}
[h_{\alpha}(t)](a+bJ) & = [w_{\alpha}(t)w_{\alpha}(-1)](a+bJ) \\
    & = a + \left(\left(
                           \begin{Array}{cc}
                             0 & t\ii \\
                             -t^{-1}\ii & 0 \\
                           \end{Array}
                         \right)\left(
                           \begin{Array}{cc}
                             0 & -\ii \\
                             \ii & 0 \\
                           \end{Array}
                         \right)b\right)J \\
    & = a + \left(\left(
                           \begin{Array}{cc}
                             t & 0 \\
                             0 & t^{-1} \\
                           \end{Array}
                         \right)b \right)J.
\end{align*}
Again, since a single root is fixed in the Chevalley definition of $h_{\alpha}$, the action of the corresponding automorphism will always be with respect to a diagonal matrix in $\text{SL}_2(k)$.  Therefore actions of $h_{\alpha}(s)h_{\alpha}(t)$ and $h_{\alpha}(st)$ will be identical.

The computations necessary to verify the second Chevalley relation and identify the necessary $N_{ij}$ are given in Appendices A and B, completing our proof.
\end{proof}

For now, denote by $G_{\gamma,\delta} \subseteq \text{Aut}(\oo)$ the group generated by our $\gamma^\pm$ and $\delta^\pm$.  Though we have not yet proven that these automorphisms actually generate all of $\text{Aut}(\oo)$, we still have a valid definition of the subgroup $T_{\Phi} \subset G_{\gamma,\delta}$ generated by
\[ T_{\Phi} := \bigl\langle h_\alpha(t) \mid \alpha \in \Phi, t \in k^\times \bigl\rangle. \]
This leads us to the following proposition.

\begin{proposition}
\label{prop:thetorus}
Let $T_s$ be the maximal torus of diagonal matrices in $\text{SL}_3(k)$.  Then the image $\theta \left(T_s\right) \subset \text{Aut}(\oo)$ is equal to $T_{\Phi} = \bigl\langle h_\alpha(t) \mid \alpha \in \Phi, t \in k^\times \bigl\rangle \subset G_{\gamma,\delta}$.
\end{proposition}

\begin{proof}
Since $\theta$ is an morphism of groups the image $\theta(T_s)$ can be generated by the two elements
\[ \theta\left(
           \begin{Array}{ccc}
             t &  &  \\
              & t^{-1} &  \\
              &  & 1 \\
           \end{Array}
         \right) \;\text{ and }\; \theta\left(
           \begin{Array}{ccc}
             1 &  &  \\
              & t &  \\
              &  & t^{-1} \\
           \end{Array}
         \right),
 \]
for $t \in k^\times$.  It is then an easy calculation that
\begin{align*}
h_{\gamma^+_{\kk}}(t) & = \gamma_{\kk}\left(
           \begin{Array}{cc}
               t &  \\
                & t^{-1} \\
           \end{Array}
         \right)\left(
           \begin{Array}{cc}
               a & \vec{v} \\
               \vec{w} & d \\
           \end{Array}
         \right) \\
    & = \left(
           \begin{Array}{cc}
               a & v_3 \\
               w_3 & d \\
           \end{Array}
         \right) + \left( \left(
           \begin{Array}{cc}
               t &  \\
                & t^{-1} \\
           \end{Array}
         \right) \left(
           \begin{Array}{cc}
               v_1 & w_2 \\
               -v_2 & w_1 \\
           \end{Array}
         \right)  \right) \cdot \left(
           \begin{Array}{cc}
               0 & \ii \\
               \ii & 0 \\
           \end{Array}
         \right) \\
    & = \left(
           \begin{Array}{cc}
               a & \langle tv_1,t^{-1}v_2,v_3 \rangle \\
               \langle t^{-1}w_1,tw_2,w_3 \rangle & d \\
           \end{Array}
         \right) \\
    & = \theta\left(
           \begin{Array}{ccc}
             t &  &  \\
              & t^{-1} &  \\
              &  & 1 \\
           \end{Array}
         \right)\left(
           \begin{Array}{cc}
               a & \vec{v} \\
               \vec{w} & d \\
           \end{Array}
         \right),
\end{align*}
and likewise that as automorphisms,
\[ \theta\left(
           \begin{Array}{ccc}
             1 &  &  \\
              & t &  \\
              &  & t^{-1} \\
           \end{Array}
         \right) = \gamma_{\ii}\left(
           \begin{Array}{cc}
               t &  \\
                & t^{-1} \\
           \end{Array}
         \right) = h_{\gamma^+_{\ii}}(t). \]
Therefore $T_s \subseteq T_{\Phi}$.  Conversely, while the subgroup $T_{\Phi}$ is defined to be generated by the $h_\alpha$ for $\alpha \in \Phi$, it may actually be generated by a smaller subset of these $h_\alpha$.  For example, it is immediate from the definitions that
\begin{align*}
h_{\gamma^+_{\ii}}(t) & = h_{\gamma^-_{\ii}}(t^{-1}), \\
h_{\gamma^+_{\jj}}(t) & = h_{\gamma^-_{\jj}}(t^{-1}), \\
h_{\gamma^+_{\kk}}(t) & = h_{\gamma^-_{\kk}}(t^{-1}).
\end{align*}
Also, from the above considerations we can show that
\[ h_{\gamma^+_{\kk}}(t)h_{\gamma^+_{\ii}}(t) = \theta\left(
           \begin{Array}{ccc}
             t &  &  \\
              & t^{-1} &  \\
              &  & 1 \\
           \end{Array}
         \right)\theta\left(
           \begin{Array}{ccc}
             1 &  &  \\
              & t &  \\
              &  & t^{-1} \\
           \end{Array}
         \right) = \theta\left(
           \begin{Array}{ccc}
             t &  &  \\
              & 1 &  \\
              &  & t^{-1} \\
           \end{Array}
         \right) = h_{\gamma^-_{\jj}}(t).\]
For the toral elements corresponding to short roots $\delta$, we calculate:
\begin{align*}
h_{\delta^+_{\ii}}(t) & = \delta_{\ii}\left(
           \begin{Array}{cc}
               t &  \\
                & t^{-1} \\
           \end{Array}
         \right)\left(
           \begin{Array}{cc}
               a & \vec{v} \\
               \vec{w} & d \\
           \end{Array}
         \right) \\
    & = \left(
           \begin{Array}{cc}
               t &  \\
                & t^{-1} \\
           \end{Array}
         \right)\left(
           \begin{Array}{cc}
               a & v_1 \\
               w_1 & d \\
           \end{Array}
         \right)\left(
           \begin{Array}{cc}
               t^{-1} &  \\
                & t \\
           \end{Array}
         \right) + \left( \left(
           \begin{Array}{cc}
               v_2 & w_3 \\
               -v_3 & w_2 \\
           \end{Array}
         \right)\left(
           \begin{Array}{cc}
               t^{-1} &  \\
                & t \\
           \end{Array}
         \right)  \right) \left(
           \begin{Array}{cc}
               0 & \jj \\
               \jj & 0 \\
           \end{Array}
         \right) \\
    & = \left(
           \begin{Array}{cc}
               a & \langle t^2 v_1,t^{-1}v_2,t^{-1}v_3 \rangle \\
               \langle t^{-2}w_1,tw_2,tw_3 \rangle & d \\
           \end{Array}
         \right) \\
    & = \theta\left(
           \begin{Array}{ccc}
             t &  &  \\
              & 1 &  \\
              &  & t^{-1} \\
           \end{Array}
         \right)\theta\left(
           \begin{Array}{ccc}
             t &  &  \\
              & t^{-1} &  \\
              &  & 1 \\
           \end{Array}
         \right)\left(
           \begin{Array}{cc}
               a & \vec{v} \\
               \vec{w} & d \\
           \end{Array}
         \right) \\
    & = h_{\gamma^-_{\jj}}(t)h_{\gamma^+_{\kk}}(t)\left(
           \begin{Array}{cc}
               a & \vec{v} \\
               \vec{w} & d \\
           \end{Array}
         \right).
\end{align*}
In this way all $h_\delta$ may be generated from different combinations of the $h_\alpha$, or equivalently, by the image $\theta(T_s)$.  Therefore $T_{\Phi} \subseteq T_s$ and our statement follows.
\end{proof}

\section{The Algebra of Derivations of $\oo$}

Our next task is to specify a Lie algebra of Cartan type $G_2$, and identify an appropriate Chevalley basis of that algebra.  Toward this, we make the following definition.

\begin{definition}
We denote by $\der$ the set of \textbf{derivations}\label{symbol:Der} on $\oo$, that is, the set of $k$-linear operators $D: \oo \rightarrow \oo$ satisfying, for all $x,y \in \oo$
\[ D\left( x  y \right) = x D(y) + D(x)  y. \]
\end{definition}

It is elementary to show that $\der$ is closed under commutation $[d_1,d_2] = d_1d_2 - d_2d_1$ and thus $\der$ is a Lie algebra over $k$.  In \cite{Jacobson1939}, N. Jacobson begins with this definition and proves that $\der$ is a simple, $14$-dimensional Lie algebra of Cartan type $G_2$.  In fact, this statement can be extended to any field of characteristic not $2$ or $3$; c.f. \cite[Theorem $5$]{Seligman1960}.  Another portion of Seligman's work proves the following:

\begin{proposition}[\cite{Seligman1960}]
Let $\text{char}\; k \neq 2,3,5,7,11,13,17$, and $\oo$ be the octonion algebra constructed over $k$.  Then the representation
\begin{align*}
\rho \colon \text{Aut}(\oo) & \longrightarrow \text{Aut}(\der) \\
    g & \longmapsto \left( D \mapsto g^{-1}Dg \right)
\end{align*}
is an isomorphism of groups.
\end{proposition}

\begin{definition}
Let $ad$ denote the \textbf{adjoint representation} of $\der$,
\begin{align*}
\text{ad} \colon \der & \longrightarrow \text{End}(\der) \\
    D & \longmapsto [D,-],
\end{align*}
and by the simplicity of $\der$, this is faithful.
\end{definition}

\begin{definition}
Let $\der^{nil} : = \{ D \in \der \mid \text{ad}\,D \text{ is nilpotent.} \}$  We then define the \textbf{exponential map} into the automorphism group of $\der$:
\begin{align*}
\exp \colon \der^{nil} & \longrightarrow \text{Aut}(\der) \\
     D & \longmapsto \exp(t \cdot \text{ad}\,D ),
\end{align*}
where $\exp$ is defined by
\[ \exp(X) := \sum_{k=0}^{\infty} \frac{X^k}{k!}, \]
where the sum terminates as $X$ is assumed to be ad-nilpotent.
\end{definition}

This gives us a diagram which describes the relationships we have so far:
\[ \der^{nil} \underset{\exp}{\longrightarrow} \text{Aut}(\der) \overset{\sim}{\underset{\rho}{\longleftarrow}} \text{Aut}(\oo). \]

Now suppose that $\mathfrak{b} = \{ H_{\alpha},H_{\beta},E_{\phi} \}$ is a Chevalley basis of $\mathfrak{g}_2$.  If the $E_phi$ are ad-nilpotent (as will be the case for us), then we exponentiate our basis elements and define
\[ e_{\alpha}(t) := \exp(t \cdot \text{ad}\, E_\alpha ) \in \text{Aut}(\der). \]
In this case we have the following theorem.

\begin{theorem}[\cite{Seligman1960}]
Let $\text{char}\; k \neq 2,3,5,7,11,13,17$, and $\oo$ be the octonion algebra constructed over $k$.  Let $G_2 \subset \text{Aut}(\der)$ be the group generated by all automorphisms of $\der$ of the form
\[ e_{\alpha}(t) = \exp(t \cdot \text{ad}\, E_\alpha), \]
where $E_\alpha$ runs through all the root vectors relative to all standard Cartan subalgebras of $\der$.  Then $G_2$ coincides with the full automorphism group $\text{Aut}(\der)$, and therefore $G_2 \cong \text{Aut}(\der) \cong \text{Aut}(\oo)$.
\end{theorem}

We note that Seligman proves this theorem by describing an explicit Chevalley basis of $\der$, and exponentiating that specific basis.  Also, \cite{Seligman1960} does not explicitly describe any individual automorphisms in $\text{Aut}(\oo)$.  We will be choosing our own basis for $\der$, but by \cite[page $6$]{Steinberg1967}, any Chevalley basis of $\der$ is unique up to some sign changes and automorphisms of $\der$, so Seligman's results will still be valid for our choices.

\section{A Chevalley Basis of $\der$}

According to our last theorem, we are left to show that our automorphisms $\gamma^\pm$ and $\delta^\pm$ are of the required form for some Chevalley basis elements $E_\alpha$, in which case we will have that $G_{\gamma,\delta} = G_2^ = \text{Aut}(\oo)$.  We do this by taking the algebraic derivative of each of the automorphisms we have.

We begin with the long roots automorphisms $\gamma^\pm$.  Define $E_{\gamma^{\pm}} \in \text{End}(\oo)$ by:
\[ E_{\gamma^{\pm}}(x) : = \lim_{t \rightarrow 0} \frac{\gamma^\pm(t)x - x}{t}. \]
That is, for each $u \in \{ \pm\ii,\pm\jj,\pm\kk \}$ we have $x = a + b l_u$ and
\begin{align*}
E_{\gamma_u^{+}}(a + bl_u) & = \lim_{t \rightarrow 0} \frac{\gamma_u^{+}(t)x - x}{t} \\
    & = \lim_{t \rightarrow 0} \; \frac{1}{t} \; \left( \left(
                                                         \begin{Array}{cc}
                                                           0 & t \\
                                                           0 & 0 \\
                                                         \end{Array}
                                                       \right) b \right)l_u \\
    & = \left( \left(
                                                         \begin{Array}{cc}
                                                           0 & 1 \\
                                                           0 & 0 \\
                                                         \end{Array}
                                                       \right) b \right)l_u.
\end{align*}
Likewise, for $E_{\gamma_u^{-}} \in \text{End}(\oo)$ we have
\[ E_{\gamma_u^{-}}(a + bl_u) = \left( \left(
                                                         \begin{Array}{cc}
                                                           0 & 0 \\
                                                           1 & 0 \\
                                                         \end{Array}
                                                       \right) b \right)l_u. \]

Next we apply the same operation to the short roots $\delta^\pm$, using the same definitions:
\[ E_{\delta^{\pm}}(x) : = \lim_{t \rightarrow 0} \frac{\delta^\pm(t)x - x}{t}. \]
Again, for each $u \in \{ \pm\ii,\pm\jj,\pm\kk \}$ and $x = a + b l_u$ we have
\begin{align*}
E_{\delta_u^{+}}(a + bl_u) & = \lim_{t \rightarrow 0} \frac{\delta_u^{+}(t)x - x}{t} \\
    & = \lim_{t \rightarrow 0} \; \frac{1}{t} \; \left[ \left( \left(
                                                         \begin{Array}{cc}
                                                           1 & t \\
                                                           0 & 1 \\
                                                         \end{Array}
                                                       \right)(a_{ij}) \left(
                                                         \begin{Array}{cc}
                                                           1 & -t \\
                                                           0 & 1 \\
                                                         \end{Array}
                                                       \right) - (a_{ij}) \right)+ \left( (b_{ij}) \left(
                                                         \begin{Array}{cc}
                                                           0 & -t \\
                                                           0 & 0 \\
                                                         \end{Array}
                                                       \right)  \right)l_u \right] \\
    & = \lim_{t \rightarrow 0} \; \frac{1}{t} \; \left[ \left( \left(
                                                         \begin{Array}{cc}
                                                           1 & t \\
                                                           0 & 1 \\
                                                         \end{Array}
                                                       \right)(a_{ij}) - (a_{ij})\left(
                                                         \begin{Array}{cc}
                                                           1 & t \\
                                                           0 & 1 \\
                                                         \end{Array}
                                                       \right) \right)\left(
                                                         \begin{Array}{cc}
                                                           1 & -t \\
                                                           0 & 1 \\
                                                         \end{Array}
                                                       \right)+ \left( (b_{ij}) \left(
                                                         \begin{Array}{cc}
                                                           0 & -t \\
                                                           0 & 0 \\
                                                         \end{Array}
                                                       \right)  \right)l_u \right] \\
    & = \lim_{t \rightarrow 0} \; \frac{1}{t} \; \left[ \left(
                                                         \begin{Array}{cc}
                                                           ta_{21} & ta_{22} - ta_{11} \\
                                                           0 & - ta_{21}\\
                                                         \end{Array}
                                                       \right)  \left(
                                                         \begin{Array}{cc}
                                                           1 & -t \\
                                                           0 & 1 \\
                                                         \end{Array}
                                                       \right)+ \left( (b_{ij}) \left(
                                                         \begin{Array}{cc}
                                                           0 & -t \\
                                                           0 & 0 \\
                                                         \end{Array}
                                                       \right)  \right)l_u \right] \\
    & = \lim_{t \rightarrow 0} \; \left[ \left(
                                                         \begin{Array}{cc}
                                                           a_{21} & a_{22} - a_{11} \\
                                                           0 & - a_{21}\\
                                                         \end{Array}
                                                       \right)  \left(
                                                         \begin{Array}{cc}
                                                           1 & -t \\
                                                           0 & 1 \\
                                                         \end{Array}
                                                       \right)+ \left( (b_{ij}) \left(
                                                         \begin{Array}{cc}
                                                           0 & -1 \\
                                                           0 & 0 \\
                                                         \end{Array}
                                                       \right)  \right)l_u \right] \\
    & =  \left(
                                                         \begin{Array}{cc}
                                                           a_{21} & a_{22} - a_{11} \\
                                                           0 & - a_{21}\\
                                                         \end{Array}
                                                       \right) + \left( (b_{ij}) \left(
                                                         \begin{Array}{cc}
                                                           0 & -1 \\
                                                           0 & 0 \\
                                                         \end{Array}
                                                       \right)  \right)l_u  \\
    & = \left[ \left( \begin{Array}{cc}
                                                           0 & 1 \\
                                                           0 & 0 \\
                                                         \end{Array}
                                                       \right), (a_{ij}) \right] + \left((b_{ij}) \left(
                                                         \begin{Array}{cc}
                                                           0 & -1 \\
                                                           0 & 0 \\
                                                         \end{Array}
                                                       \right)  \right)l_u.
\end{align*}
In this expression the square brackets denote the commutator bracket.  Likewise, for $E_{\delta_u^{-}} \in \text{End}(\oo)$ we have
\[ E_{\delta_u^{-}}(a + bl_u) = \left[ \left( \begin{Array}{cc}
                                                           0 & 0 \\
                                                           1 & 0 \\
                                                         \end{Array}
                                                       \right), (a_{ij}) \right] + \left((b_{ij}) \left(
                                                         \begin{Array}{cc}
                                                           0 & 0 \\
                                                           -1 & 0 \\
                                                         \end{Array}
                                                       \right)  \right)l_u. \]

These define a set $\left\{ E_\phi \right\}_{\phi \in \Phi}$.  It can be checked directly that each of these $E_\phi$ is are derivations, and nilpotent with $E_\phi^2 = 0$, so for any other $D \in \der$:
\begin{align*}
\text{ad}\,(E_\phi)^3(D) & = \text{ad}(E_\phi)^2(E_\phi D-D E_\phi) \\
    & = \text{ad}(E_\phi)(-2E_\phi D E_\phi) \\
    & = 0.
\end{align*}
Thus each $\text{ad}\,(E_\phi)^3 = 0$ and $E_\phi$ is ad-nilpotent.  It is also clear that the $E_\phi$ exponentiate to our original octonion automorphisms.

For the two simple roots $\delta = \delta_{\ii}^+$ and $\gamma = \gamma_{\kk}^-$ in $\Delta$, define
\[ H_\gamma := [E_\gamma, E_{-\gamma}] \quad\text{and}\quad H_\delta := [E_\delta, E_{-\delta}].\]

\begin{proposition}
The combined set $\mathfrak{b} : = \{ H_\gamma, H_\delta, E_\phi \mid \phi \in \Phi \}$ defines a Chevalley basis of $\der$.
\end{proposition}

\begin{proof}
The Chevalley group relations can be recovered by the Chevalley basis relations via exponentiation, and vice versa via differentiation.  Therefore the elements $E_\phi$ will satisfy the relevant Chevalley basis relations, while $H_\gamma$, $H_\delta$ were constructed in order to satisfy their needed relations.

Now, let $\mathfrak{b}'$ be any other Chevalley basis for $\der$.  We have a Lie algebra morphism
\[ f: \mathfrak{b}' \longrightarrow \mathfrak{b} \]
which maps basis elements to corresponding basis elements, and this map is injective as $\der$ is simple.  Therefore $\mathfrak{b}$ is linearly independent and our statement follows.
\end{proof}

%%%%%%%%%%%%%%%%%%%%%%%%%%%%%%%%%%%%%%%%%%%%%%%%%%%%%%%%%%%%%%%%%%%%%%%%%%%%%%%%%%%%%%%%%%%%%%%%%%%%%%

\part{Part II: Local Fields}

\chapter{Introduction}

In Part II we discuss certain structures in $G_2$ that arise only when it is constructed over a local, non-archimedean field.  In particular, we will take $k = \qp$, though many results will be true in a more general context.  We will assume that the reader is familiar with basic notions related to local non-archimedean fields.  If not, \cite{Serre1968} or \cite{Serre1973} are standard references.

One important fact about $G_2$ constructed over $\qp$ is the following, from \cite{Bruhat1987-2}:

\begin{theorem}
For $k = \qq_p$, any octonion algebra $\oo$ constructed over $k$ will necessarily be split.  Therefore, there is only one possible octonion algebra over $k = \qq_p$, up to isomorphism.
\end{theorem}

\begin{comment}
\begin{proof}
Omitted.  Don't prove this; Gan-Yu say ``we remind the reader that if... then such a group is necessarily split \footnote{Jacobson bottom of page $64$ has a similar statement for finite fields (without proof).  See Marty's email for a sketch of the proof.  Note also the use of Phister forms.  [SV] refers to O'Meara for this proof.}
\footnote{Note also that Jacobson, p $63$, states that any two split composition algebras of the same dimension are isomorphic.  Therefore, if we can show that any octionion algebra of $k = \qq_p$ has zero divisors, then this octionion algebra will be unique up to isomorphism (of composition algebras!)  Also, it is Hurwitz's Theorem that any normed (in the analyitic sense!) division algebra is isomorphic to either $\rr$, $\cc$, $\hh$, or $\oo$.}
\end{proof}
\end{comment}

Therefore in our case, we need only consider the split form of $G_2$.  We continue to follow the work of Bruhat and Tits in Chapter \ref{standardapartment}, where we construct the standard apartment of the Bruhat-Tits building $\mathcal{B}(G_2)$, based on our previous choices of maximal torus and root systems.  Though the building itself is a larger structure, most of our work will be focused within the apartment itself.

In Chapter \ref{orders} we define octonion orders, lattice filtrations and sequences, and valuations on $\oo$.  The narrative there will follow along very closely with \cite{GanYu2003}, and our attempt will be to place their work in a more concrete setting using the explicit automorphisms we now have.

\chapter{Construction of the Standard Apartment}
\label{standardapartment}

\section{Coroots and the Coroot Diagram}
\label{corootdefinitions}

We again let $\Phi$ be the root system used in the Chevalley construction of our group $G_2$.  We will choose as simple roots $\delta := \delta_{\ii}^{+}$ and $\gamma := \gamma_{\kk}^{-}$ and the set of positive roots $\Phi^{+}$ corresponding to this choice.  We begin by describing the purely abstract coroot system.

\begin{definition}
From each root $\alpha \in \Phi$, define the corresponding \textbf{coroot}
\[ \alpha^{\vee} : = \dfrac{2}{( \alpha, \alpha )}\alpha = \dfrac{2}{\| \alpha \|^2}\alpha. \]
These coroots will also form an abstract root system\label{symbol:Phivee} $\Phi^{\vee}$ of the same rank as $\Phi$, which we will call the dual root system or \textbf{coroot system}, and denote it by $\Phi^{\vee}$.
\end{definition}

Since each coroot is just a scaled root, the angle between our simple coroots is equal to that of the simple roots.   It is also easy to see that the coroot corresponding to a short root will be long, and the coroot corresponding to a long root will be short.  The coroot system for our group $G_2$ is shown in Figure \ref{corootsG2}, with the simple coroots labeled.

\begin{figure}[ht]
\begin{center}
\begin{tikzpicture}[thin,scale=1.5]%
    \draw \foreach \x in {0,120,240} {
        (\x+120:1.732) -- (\x+240:1.732)
        (\x-60:1.732) -- (\x+60:1.732)
        (0:0) -- (0:1.732)
        (0:0) -- (150:1)};
    \draw{ (0:0) node[circle, draw, fill=black!50,
                        inner sep=0pt, minimum width=4pt]{}
        (0:1.732) node[circle, draw, fill=black!50,
                        inner sep=0pt, minimum width=4pt, label=right:$\delta^{\vee}$]{}
        (150:1) node[circle, draw, fill=black!50,
                        inner sep=0pt, minimum width=4pt, label=above left:$\gamma^{\vee}$]{}
    };
\end{tikzpicture}
\label{corootsG2}\caption{The coroot diagram of $\mathfrak{g}_2$.}
\end{center}
\end{figure}
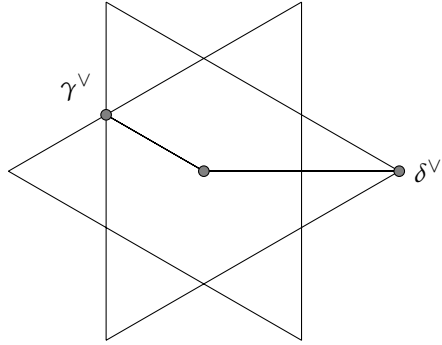

In the setting of algebraic groups, we often like to consider our coroots to to be homomorphisms from $k^{\times} \rightarrow T$, where $T$ is some maximal torus of the group.  Luckily, we already have homomorphisms of this type, namely, the $h_{\alpha}$ described in the Chevalley construction, and which generate such a subgroup $T_{\Phi}$.  Therefore we can define the action of $\alpha^{\vee}$ on $k^{\times}$ to be equal to that of $h_{\alpha}$, for each root $\alpha \in \Phi$:
\[ \alpha^{\vee} = h_{\alpha} \colon k^{\times} \rightarrow T_{\Phi}. \]

In the same setting we would also like to consider our roots themselves to be homomorphisms $\alpha \colon T \rightarrow k^{\times}$ for each $\alpha \in \Phi$.  This association is a bit more complex.

\begin{proposition}
For each pair of roots $\alpha,\beta \in \Phi$ we have the following identity of Chevalley elements, where $u,t \in k^\times$:
\[ h_{\beta}(t)e_{\alpha}(u)h_{\beta}(t)^{-1} = e_{\alpha}(t^{2(\alpha,\beta )/( \beta,\beta )}u). \]
This defines a map on generators $\alpha \colon h_\beta(t) \mapsto t^{2(\alpha,\beta )/( \beta,\beta )}$ which then extends to a homomorphism $\alpha \colon T_{\Phi} \longrightarrow k^{\times}$.
\end{proposition}

\begin{proof}
We refer the reader to \cite[Lemmas $19$ and $20$]{Steinberg1967} and their proofs.
\end{proof}

\begin{definition}\footnote{We use this definition as our group $G_2$ is simply-connected and adjoint.}\label{symbol:charactersandcocharacters}
For $T = T_{\Phi}$, we denote by $X^\bullet (T)$ the group of characters from $T \rightarrow k^\times$ generated by the roots and by $X_\bullet (T)$ the group of cocharacters from $k^\times \rightarrow T$ generated by the coroots.
\end{definition}

We next establish a connection between the roots and coroots.  It is well known, or easily calculated, that the root system of type $G_2$ has Cartan matrix $A = \left(
                   \begin{Array}{cc}
                     2 & -3 \\
                     -1 & 2 \\
                   \end{Array}
                 \right)$.  We use these Cartan integers to make the following definiton.

\begin{definition}\label{symbol:pairing}
For $\delta,\gamma \in \Delta$ the simple roots in the $G_2$ diagram, define the pairings
\begin{align*}
\langle \delta, \delta^{\vee} \rangle & := \frac{2( \delta,\delta )}{( \delta,\delta )} = 2, \\
\langle \delta, \gamma^{\vee} \rangle & := \frac{2( \gamma,\delta )}{( \delta,\delta )} = -1, \\
\langle \gamma, \delta^{\vee} \rangle & := \frac{2( \delta,\gamma )}{( \gamma,\gamma )} = -3, \\
\langle \gamma, \gamma^{\vee} \rangle & := \frac{2( \gamma,\gamma )}{( \gamma,\gamma )} = 2.
\end{align*}
These identifications extend to a non-symmetric bilinear pairing between $\Phi$ and $\Phi^\vee$.
\end{definition}

To close this section, we remind ourselves that each coroot is simply a re-scaled root, so that an explicit connection between roots and coroots can be seen from comparing the root and coroot diagrams.  Seeing that the highest long root, $3\delta + 2\gamma$, lies at the same relative angle as the coroot $\delta^{\vee} + 2\gamma^{\vee}$ should lead us to believe the identity $(3\delta + 2\gamma)^{\vee} = \delta^{\vee} + 2\gamma^{\vee}$.  Indeed, we can verify these connections algebraically by using the following relation:

\begin{proposition}
Let $\alpha,\beta \in \Phi$, and $s_\alpha(\beta) \in W$ be the reflection of $\beta$ through the hyperplane orthogonal to $\alpha$ in the root diagram.  Then
\[ (s_\alpha(\beta))^{\vee} = \beta^{\vee} - \langle \alpha,\beta^{\vee}\rangle\alpha^{\vee}. \]
\end{proposition}

\begin{proof}
Recall that the coroots have exactly the same relative angles as the roots they are built from, so from the definition of $s_\alpha(\beta)$ (root axiom (c)), we have
\[ (s_\alpha(\beta))^{\vee} = \beta^{\vee} - \frac{2(\beta,\alpha)}{(\alpha,\alpha)}\alpha^{\vee}. \]
The definition of $\langle \alpha,\beta^{\vee}\rangle$ completes the identity.
\end{proof}

Now, suppose that $\alpha$ is any root.  Then negative roots and coroots coincide:
\begin{align*}
(s_\alpha(\alpha))^{\vee} & = \alpha^{\vee} - \langle \alpha,\alpha^{\vee}\rangle\alpha^{\vee} \\
    & = \alpha^{\vee} - 2\alpha^{\vee} \\
    & = -\alpha^{\vee},
\intertext{and therefore}
(-\alpha)^{\vee} & = (s_\alpha(\alpha))^{\vee} = -(\alpha^{\vee}).
\end{align*}

Likewise, for $\gamma,\delta$ our long and short (respectively) simple roots:
\begin{align*}
(s_\delta(\gamma))^{\vee} & = \gamma^{\vee} - \langle\delta,\gamma^{\vee}\rangle\delta^{\vee} \\
    & = \gamma^{\vee} - (-1)\delta^{\vee} \\
    & = \delta^{\vee} + \gamma^{\vee}.
\end{align*}
Therefore, since the reflection of $\gamma$ across the hyperplane orthogonal to $\delta$ is $3\delta + \gamma$,
\[ (3\delta + \gamma)^{\vee} = (s_\delta(\gamma))^{\vee} = \delta^{\vee} + \gamma^{\vee}. \]

Then, we can use the last fact to show that
\begin{align*}
(s_\gamma(3\delta + \gamma))^{\vee} & = (3\delta + \gamma)^{\vee} - \langle\gamma,(3\delta + \gamma)^{\vee}\rangle\gamma^{\vee} \\
    & = \delta^{\vee} + \gamma^{\vee} - \langle\gamma,\delta^{\vee} + \gamma^{\vee}\rangle\gamma^{\vee} \\
    & = \delta^{\vee} + \gamma^{\vee} - \left(\langle\gamma,\delta^{\vee}\rangle + \langle\gamma,\gamma^{\vee}\rangle\right)\gamma^{\vee} \\
    & = \delta^{\vee} + \gamma^{\vee} - ((-3) + (2))\gamma^{\vee} \\
    & = \delta^{\vee} + 2\gamma^{\vee}.
\end{align*}
Since the reflection of $3\delta + \gamma$ across the hyperplane orthogonal to $\gamma$ is $3\delta + 2\gamma$,
\[ (3\delta + 2\gamma)^{\vee} = s_\gamma(3\delta + \gamma))^{\vee} = \delta^{\vee} + 2\gamma^{\vee}. \]

Working in this way we can calculate all of the following identities:
\begin{align*}
(\delta + \gamma)^{\vee} & = \delta^{\vee} + 3\gamma^{\vee}, \\
(2\delta + \gamma)^{\vee} & = 2\delta^{\vee} + 3\gamma^{\vee}, \\
(3\delta + \gamma)^{\vee} & = \delta^{\vee} + \gamma^{\vee}, \\
(3\delta + 2\gamma)^{\vee} & = \delta^{\vee} + 2\gamma^{\vee}. \\
\end{align*}

Again, since $-(\alpha^{\vee}) = (-\alpha)^{\vee}$ for all roots $\alpha$, this gives the coroot corresponding to every root.

\section{Affine Roots, the Apartment, and Hyperplanes}

\begin{definition}\label{symbol:apartment}
With $T = \bigl\langle h_\alpha(t) \bigl\rangle$ and the corresponding choice of roots and coroots, we define the \textbf{affine apartment} $\mathcal{A}$ to be the full coroot space $\mathcal{A} = E^{\star} = X_{\bullet}(T) \otimes \rr$.
\end{definition}

From this definition alone, the apartment of any rank two coroot system (of any Cartan type) will be identical.  However, there is much substructure on $\mathcal{A}$ that will serve to distinguish between such apartments.  Toward this, we make some new definitions:

\begin{definition}\newnot{symbol:hyperplane}
Let $\alpha \in \Phi$, and $n \in \zz$.  We define a new (affine) functional on $\mathcal{A}$:
\begin{align*}
\alpha + n \colon \mathcal{A} & \longrightarrow \rr \\
    x & \longmapsto \langle\alpha,x\rangle + n,
\end{align*}
and define
\begin{align*}
H_{\alpha+n} & = \{ x \in \mathcal{A} \mid ( \alpha+n )( x ) = 0 \} \\
    & = \{ x \in \mathcal{A} \mid \langle \alpha, x \rangle = -n \}.
\end{align*}
\end{definition}

Because $\langle -,- \rangle$ is bilinear, this zero locus $(\alpha + n)^{-1}(0)$ forms an affine subspace of codimension $1$ in the plane of coroots.  For example, consider the hyperplane $H_{\delta}$.  Since $\langle \delta,\delta^{\vee} \rangle= 2$, the coroot $\delta^{\vee}$ will not lie in $H_{\delta}$.  However, we do have:
\begin{align*}
\langle \delta,\delta^{\vee}+2\gamma^{\vee} \rangle & = \langle \delta,\delta^{\vee} \rangle + 2\langle \delta,\gamma^{\vee} \rangle \\
    & = (2) + 2(-1) \\
    & = 0.
\end{align*}
Since this pairing is bilinear, we also know that $\langle \delta,-\delta^{\vee}-2\gamma^{\vee} \rangle = -\langle \delta,\delta^{\vee}+2\gamma^{\vee} \rangle = 0$, and this is enough to determine the hyperplane $H_{\delta} \subseteq \mathcal{A}$, shown in Figure \ref{deltahyperplane}.

\begin{figure}[ht!]
\begin{center}
\begin{tikzpicture}[thin,scale=1.5]%
    \draw \foreach \x in {0,120,240} {
        (\x+120:1.732) -- (\x+240:1.732)
        (\x-60:1.732) -- (\x+60:1.732)
        (0:0) -- (0:1.732)
        (0:0) -- (150:1)};
    \draw[dashed]{
        (0:0) -- (90:3) node[above]{$H_{\delta}$}
        (0:0) -- (90:-3)
    };
    \draw{ (0:0) node[circle, draw, fill=black!50,
                        inner sep=0pt, minimum width=4pt]{}
        (0:1.732) node[circle, draw, fill=black!50,
                        inner sep=0pt, minimum width=4pt, label=right:$\delta^{\vee}$]{}
        (150:1) node[circle, draw, fill=black!50,
                        inner sep=0pt, minimum width=4pt, label=above left:$\gamma^{\vee}$]{}
    };
\end{tikzpicture}
\caption{A hyperplane corresponding to $\delta$.}\label{deltahyperplane}
\end{center}
\end{figure}
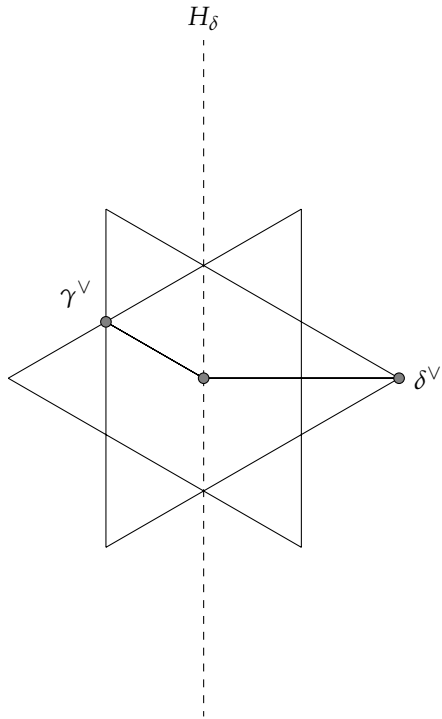

\newpage

Likewise, since $\langle \delta, \gamma^{\vee} \rangle = \delta(\gamma^{\vee}) = -1$, we have that $\gamma^{\vee}$ lies in the hyperplane $H_{\delta+1}$.  We also have that
\begin{align*}
\langle \delta,-\delta^{\vee}-\gamma^{\vee} \rangle & = -\langle \delta,\delta^{\vee} \rangle - \langle \delta,\gamma^{\vee} \rangle \\
    & = -(2) -(-1) \\
    & = -1.
\end{align*}

Therefore the coroot $-\delta^{\vee}-\gamma^{\vee}$ also lies in the hyperplane $H_{\delta+1}$.  Again, since our pairing is bilinear, any linear combination of these two coroots will also lie in our hyperplane.  In Figure \ref{alldeltahyperplanes}, we display a number of hyperplanes corresponding to the root $\delta$.

\begin{figure}[ht!]
\begin{center}
\begin{tikzpicture}[thin,scale=1.5]%
    \draw \foreach \x in {0,120,240} {
        (\x+120:1.732) -- (\x+240:1.732)
        (\x-60:1.732) -- (\x+60:1.732)
        (0:0) -- (0:1.732)
        (0:0) -- (150:1)};
    \draw[dashed] \foreach \x in {-2,-1,0,1,2}{
        (0:.866*\x) -- ++(90:3) node[above]{$H_{\delta - (\x)}$}
        (0:.866*\x) -- ++(90:-3)
    };
    \draw{ (0:0) node[circle, draw, fill=black!50,
                        inner sep=0pt, minimum width=4pt]{}
        (0:1.732) node[circle, draw, fill=black!50,
                        inner sep=0pt, minimum width=4pt, label=right:$\delta^{\vee}$]{}
        (150:1) node[circle, draw, fill=black!50,
                        inner sep=0pt, minimum width=4pt, label=above left:$\gamma^{\vee}$]{}
    };
\end{tikzpicture}
\caption{The hyperplanes corresponding to $\delta$.}\label{alldeltahyperplanes}
\end{center}
\end{figure}

We use the same strategy to identify hyperplanes $H_{\gamma + n}$ corresponding to the other simple root $\gamma \in \Phi$.  These hyperplanes are shown in Figure \ref{gammahyperplanes} below.  Note that the coroot $\gamma^{\vee}$ lies on the hyperplane $H_{\gamma - 2}$ as expected, and that the origin is again contained in the hyperplane $H_{\gamma - 0}$.

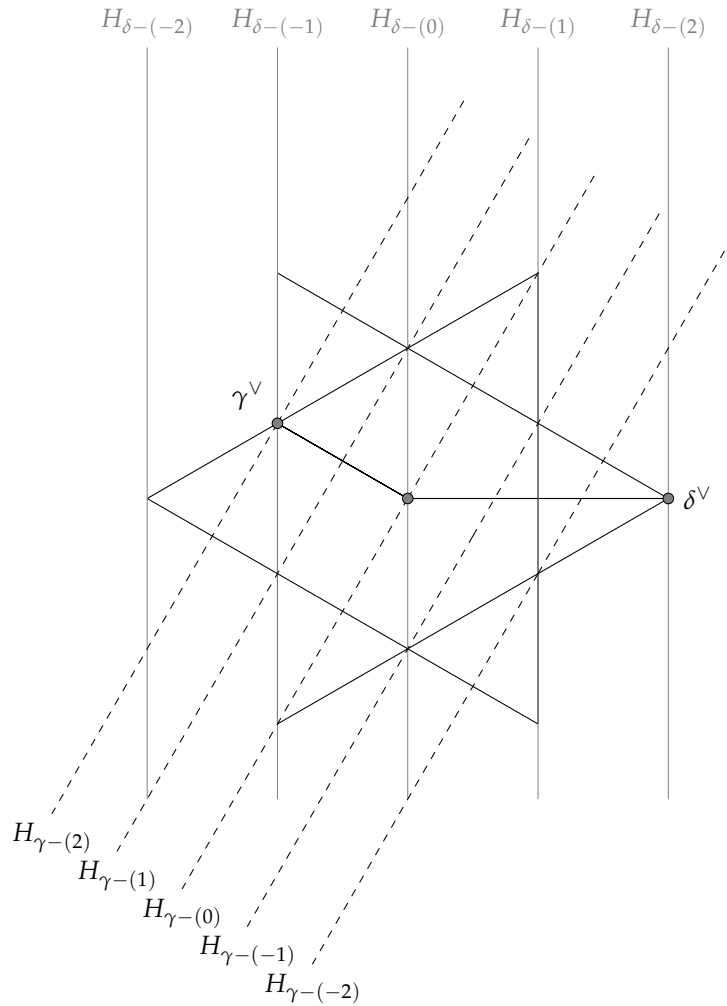
\begin{figure}[ht!]
\begin{center}
\begin{tikzpicture}[thin,scale=2.0]%
    \draw \foreach \x in {0,120,240} {
        (\x+120:1.732) -- (\x+240:1.732)
        (\x-60:1.732) -- (\x+60:1.732)
        (0:0) -- (0:1.732)
        (0:0) -- (150:1)};
    \draw[help lines] \foreach \x in {-2,-1,0,1,2}{
        (0:.866*\x) -- ++(90:3) node[above]{$H_{\delta - (\x)}$}
        (0:.866*\x) -- ++(90:-2)
    };
    \draw[dashed] \foreach \x in {-2,-1,0,1,2}{
        (150:.5*\x) -- ++(60:2.5)
        (150:.5*\x) -- ++(60:-3) node[below]{$H_{\gamma - (\x)}$}
    };
    \draw{ (0:0) node[circle, draw, fill=black!50,
                        inner sep=0pt, minimum width=4pt]{}
        (0:1.732) node[circle, draw, fill=black!50,
                        inner sep=0pt, minimum width=4pt, label=right:$\delta^{\vee}$]{}
        (150:1) node[circle, draw, fill=black!50,
                        inner sep=0pt, minimum width=4pt, label=above left:$\gamma^{\vee}$]{}
    };
\end{tikzpicture}
\caption{The hyperplanes corresponding to $\gamma$.}\label{gammahyperplanes}
\end{center}
\end{figure}

\newpage

In Figure \ref{deltagammahyperplanes}, we show the hyperplanes corresponding to the non-simple root $3\delta + \gamma$.  Note that according to our calculations in Section \ref{corootdefinitions}, or by comparing the root and coroot diagrams, we have $(3\delta + \gamma)^{\vee} = \delta^{\vee} + \gamma^{\vee}$.  We find that
\begin{align*}
\langle 3\delta + \gamma,\delta^{\vee}+\gamma^{\vee} \rangle & = \langle 3\delta,\delta^{\vee}+\gamma^{\vee} \rangle + \langle \gamma,\delta^{\vee}+\gamma^{\vee} \rangle \\
    & = 3\langle \delta,\delta^{\vee} \rangle + 3\langle \delta,\gamma^{\vee} \rangle + \langle \gamma,\delta^{\vee} \rangle + \langle \gamma,\gamma^{\vee} \rangle \\
    & = 3(2) + 3(-1) + (-3) + (2) \\
    & = 2,
\end{align*}
as we should expect.  The remainder of the $H_{3\delta + \gamma + n}$ hyperplanes are labeled as well.

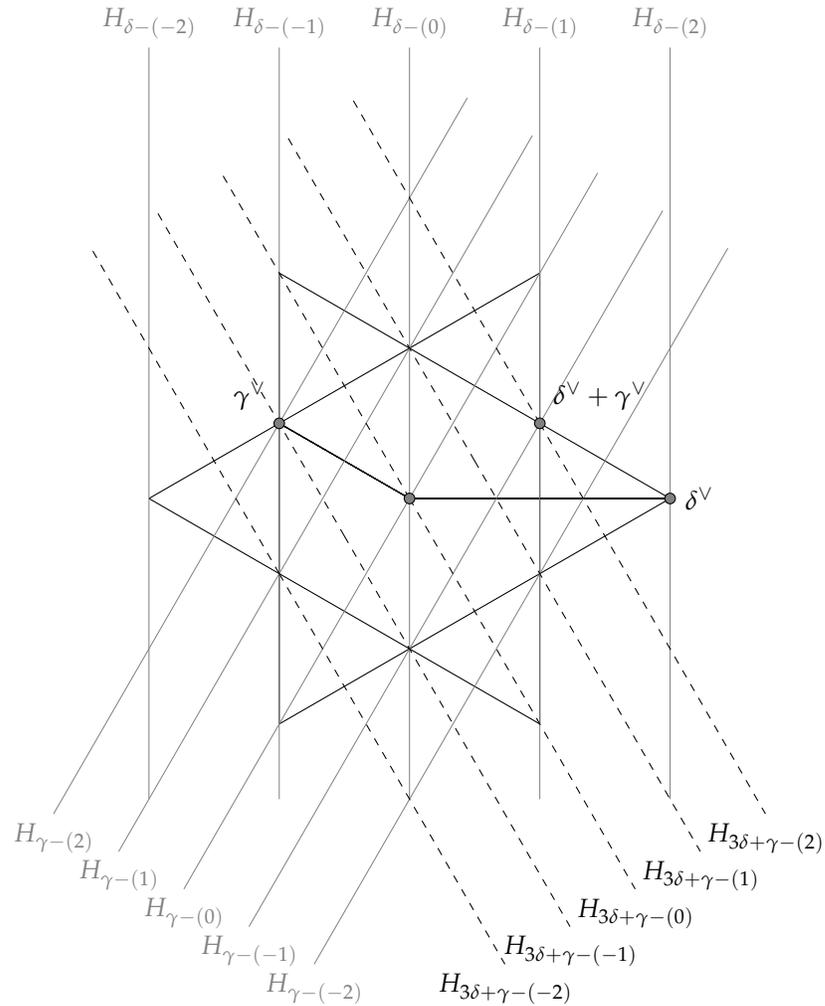
\begin{figure}[ht!]
\begin{center}
\begin{tikzpicture}[thin,scale=2.0]%
    \draw \foreach \x in {0,120,240} {
        (\x+120:1.732) -- (\x+240:1.732)
        (\x-60:1.732) -- (\x+60:1.732)
        (0:0) -- (0:1.732)
        (0:0) -- (150:1)};
    \draw[help lines] \foreach \x in {-2,-1,0,1,2}{
        (0:.866*\x) -- ++(90:3) node[above]{$H_{\delta - (\x)}$}
        (0:.866*\x) -- ++(90:-2)
    };
    \draw[help lines] \foreach \x in {-2,-1,0,1,2}{
        (150:.5*\x) -- ++(60:2.5)
        (150:.5*\x) -- ++(60:-3) node[below]{$H_{\gamma - (\x)}$}
    };
    \draw[dashed] \foreach \x in {-2,-1,0,1,2}{
        (30:.5*\x) -- ++(120:2.5)
        (30:.5*\x) -- ++(120:-3) node[below]{$H_{3\delta + \gamma - (\x)}$}
    };
    \draw{ (0:0) node[circle, draw, fill=black!50,
                        inner sep=0pt, minimum width=4pt]{}
        (0:1.732) node[circle, draw, fill=black!50,
                        inner sep=0pt, minimum width=4pt, label=right:$\delta^{\vee}$]{}
        (150:1) node[circle, draw, fill=black!50,
                        inner sep=0pt, minimum width=4pt, label=above left:$\gamma^{\vee}$]{}
        (30:1) node[circle, draw, fill=black!50,
                        inner sep=0pt, minimum width=4pt, label=above right:$\delta^{\vee} + \gamma^{\vee}$]{}
    };
\end{tikzpicture}
\caption{The hyperplanes corresponding to $3\delta + \gamma$.}\label{deltagammahyperplanes}
\end{center}
\end{figure}

After all hyperplanes $H_{\alpha + n}$ for all $12$ roots $\alpha \in \Phi$ have been positively identified, we see in Figure \ref{apartmentG2} that the standard affine apartment is divided up into infinitely many $30,60,90$ triangles, each of which is called a \textbf{chamber} of the apartment.  For purposes of later identification, we will refer to the vertices that connect edges at an incidence angle of $30^\circ$ as \textbf{type $1$} points.  Vertices which have incidence angles of $90^\circ$ will be called \textbf{type $2$} points, and those that have incidence angles of $60^\circ$ will be called \textbf{type $3$} points.

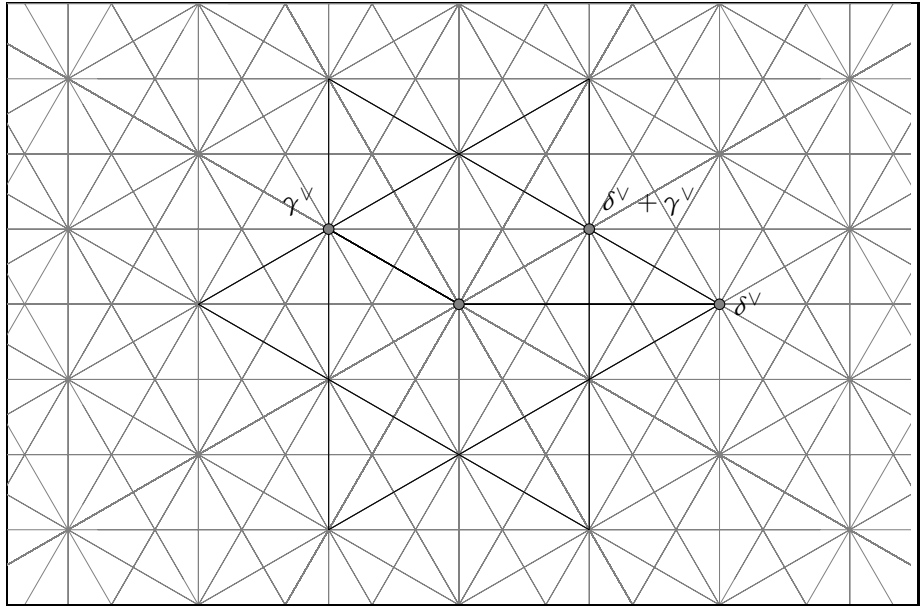
\begin{figure}[ht!]
\begin{center}
\frame{%
\begin{tikzpicture}[thin,scale=2.0]%
    \clip (-3,-2) rectangle (3,2);
    \draw[help lines] \foreach \z in {0,1,...,5} {
            \foreach \y in {0,1,...,5} {
              \foreach \x in {0,1,...,11} {
        (0:0) -- ++(30*\x:5)
        (60*\y + 30:\z) -- ++(30*\x:5)
    }}};
    \draw \foreach \x in {0,120,240} {
        (\x+120:1.732) -- (\x+240:1.732)
        (\x-60:1.732) -- (\x+60:1.732)
        (0:0) -- (0:1.732)
        (0:0) -- (150:1)};
    \draw{ (0:0) node[circle, draw, fill=black!50,
                        inner sep=0pt, minimum width=4pt]{}
        (0:1.732) node[circle, draw, fill=black!50,
                        inner sep=0pt, minimum width=4pt, label=right:$\delta^{\vee}$]{}
        (150:1) node[circle, draw, fill=black!50,
                        inner sep=0pt, minimum width=4pt, label=above left:$\gamma^{\vee}$]{}
        (30:1) node[circle, draw, fill=black!50,
                        inner sep=0pt, minimum width=4pt, label=above right:$\delta^{\vee} + \gamma^{\vee}$]{}
    };
\end{tikzpicture}
}
\caption{The Standard Apartment of $G_2$.}\label{apartmentG2}
\end{center}
\end{figure}

Recall that the construction of this apartment was contingent on a choice of maximal torus $T \subset G$, or equivalently from a choice of Chevalley generators $e_{\alpha}$ for $\alpha \in \Phi$.  An alternate choice of torus will yield a distinct apartment, though the two will have identical structures.  These distinct apartments can be ``glued together'' along the hyperplane edges, and the greater structure thus obtained is known as the \textbf{Bruhat-Tits building}\label{symbol:building} $\mathcal{B}(G_2)$ of $G_2$.   However, in this current work we have already made a steadfast choice of Chevalley generators, so we will have little need of this larger structure for now, and will focus on the affine apartment obtained from our choices.

\chapter{Lattice Filtrations and Octonion Orders}
\label{orders}

\section{Lattices and Orders in $\oo$}
\label{latticesandorders}

For the current section we will briefly step away from our discussion of roots, coroots, and the affine apartment.  We outline here some general definitions and facts about octonion orders, in the sense of \cite{GanYu2003}.  We will quickly return to our affine apartment at the end of Section \ref{algebravaluations}.

Also, for the first time the discrete valuation of the field $k = \qp$ will come into play.  Therefore, in this section and for the remainder of this paper, the field over which we work will be fixed, with $k = \qp$ and the uniformizer $p \neq 2$.

\begin{definition}\label{symbol:lattice}
A \textbf{lattice} in $\oo$ is a finitely generated $\zz_{p}$-submodule $\Lambda$ of $\oo$, such that $\Lambda \otimes \qq_{p} = \oo$.  An \textbf{order} in $\oo$ is a $\zp$-lattice which is also a unital subring.
\end{definition}

\begin{definition}\label{symbol:latticefiltration}
A \textbf{lattice filtration} in $\oo$ is a set of lattices $\{\Lambda_r\}_{r \in \rr}$, totally ordered with respect to containment, satisfying the following axioms:
\begin{enumerate}
\item $\bigcup_{r} \Lambda_{r} = \oo.$

\item $\bigcap_{r} \Lambda_{r} = \{ 0 \}.$

\item If $r,s \in \rr$ and $r \leq s$, then $\Lambda_{r} \supset \Lambda_{s}.$

\item If $r \in \rr$, then $\Lambda_{s} = \bigcap_{r < s} \Lambda_{r}.$
\end{enumerate}
\end{definition}

\begin{definition}
Let $\{ \Lambda_r \}$ be a lattice filtration in $\oo$.  For any $r \in \rr$, we take
\[ \Lambda_{r^{+}} = \bigcup_{r<s} \Lambda_{s}. \]
The \textbf{jumps} of the lattice $(\Lambda_r)$ are the real numbers $r \in \rr$ for which $\Lambda_{r^{+}} \neq \Lambda_{r}$.
\end{definition}

\begin{definition}
A lattice filtration is called a \textbf{lattice sequence} if, for every $r \in \rr$, we have $\Lambda_{r+1} = p \cdot \Lambda_{r}$.
\end{definition}

We should note that a given lattice sequence $\{ \Lambda_r \}$ may or may not have lattices intermediate to $\Lambda_{r}$ and $\Lambda_{r+1} = p \cdot \Lambda_{r}$.  However, there will be finitely many jumps in the filtration between $0$ and $1$.  Let $j_i \in \rr$ such that $0 < j_1 < j_2 < \dots < j_{r-1} < 1$, and we have containments
\[ \dots \subsetneq p \,\Lambda_{j_1} \subsetneq p \,\Lambda_{0} \subsetneq \Lambda_{j_{r-1}} \subsetneq \dots \subsetneq \Lambda_{j_2} \subsetneq \Lambda_{j_1} \subsetneq \Lambda_{0} \subsetneq p^{-1} \Lambda_{j_{r-1}} \subsetneq \dots \]

The maximum number of intermediate lattices for a given sequence is called the \textbf{rank} of the sequence.

As an example, we may consider the standard lattice $\underline\Lambda := \oo(\zp)$ of octonion elements with purely integer entries.  This is certainly a $\zp$-module of $\oo$ which, after extending scalars, is equal to $\oo$, but it is also closed under multiplication and contains the identity. Therefore it is an order in $\oo$.

We may also define the individual lattices of a lattice filtration, with $r \in \rr$, by
\[ \underline\Lambda_r := p^{\lfloor r \rfloor} \underline\Lambda = \oo(p^{\lfloor r \rfloor} \zp).  \]

Note that $\{ \underline\Lambda_r \}$ is a lattice sequence by definition.  For $r \geq 0$, each of these lattices will be an order in $\oo$, but not so for negative $r$.  Also, because of the discrete valuation on $\qp$, the jumps of this lattice filtration will occur precisely at $r \in \zz$.  We will refer to $\{ \underline\Lambda_r \}$ as the \textbf{standard lattice sequence}\label{symbol:standardlattice}.

\begin{proposition}
\label{orderconditions}
Let $\Lambda$ be the following lattice in $\oo$:
\[ \Lambda = \left(
  \begin{Array}{cc}
    p^{a_1}\zp & \bigl\langle p^{a_2}\zp,p^{a_3}\zp,p^{a_4}\zp \bigr\rangle \\
    \bigl\langle p^{a_5}\zp,p^{a_6}\zp,p^{a_7}\zp \bigr\rangle & p^{a_8}\zp \\
  \end{Array}
\right). \]
That is, the $i^{\text{th}}$ entry of an element in this lattice should have valuation at minimum $a_i$.  Then $\Lambda$ is an order in $\oo$ if and only if all of the following relations are satisfied:
\begin{align*}
a_2 + a_3 & \geq a_7, \quad\quad a_2+a_5 \geq 0, \quad\quad a_1 = a_8 = 0.\\
a_2 + a_4 & \geq a_6, \quad\quad a_3+a_6 \geq 0, \\
a_3 + a_4 & \geq a_5, \quad\quad a_4+a_7 \geq 0, \\
a_5 + a_6 & \geq a_4, \\
a_5 + a_7 & \geq a_3, \\
a_6 + a_7 & \geq a_2, \\
\end{align*}
\end{proposition}

\begin{proof}
We begin with the last condition, that $a_1 = a_8 = 0$.  For $\Lambda$ to be an order, it must first contain the identity, which forces $a_1$ and $a_8$ to be nonpositive.  Furthermore, if $a_1 < 0$, then by performing the multiplication
\[ \Lambda^2 = \left(
  \begin{Array}{cc}
    p^{a_1}\zp & \bigl\langle p^{a_2}\zp,p^{a_3}\zp,p^{a_4}\zp \bigr\rangle \\
    \bigl\langle p^{a_5}\zp,p^{a_6}\zp,p^{a_7}\zp \bigr\rangle & p^{a_8}\zp \\
  \end{Array}
\right)^2, \]
we see that the upper-left entry of the product will be in
\[ p^{2a_1}\zp + p^{a_2+a_5}\zp + p^{a_3+a_6}\zp + p^{a_4+a_7}\zp. \]
Since $2a_1 < a_1$, this is not contained in $p^{a_1}\zp$ for any values of the remaining exponents. Therefore $\Lambda$ is not closed under multiplication and is not an order in $\oo$.  The lower-right entry of the product can also be examined to show that $a_8\nless0$.

Now suppose that $a_1 = a_8 = 0$.  Therefore $\ee \in \oo$, and we check to see when $\Lambda$ is closed under multiplication by performing the multiplication $\Lambda^2$.  We will examine each entry of this product individually.

In the upper-left entry, we find
\[ \zp + p^{a_2+a_5}\zp + p^{a_3+a_6}\zp + p^{a_4+a_7}\zp, \]
which will be contained in $\zp$ so long as $a_2+a_5$, $a_3+a_6$, $a_4+a_7$ are all nonnegative.  This gives us three more of our relations, and computing the lower-right entry will yield the same result.

Next, computing the upper-right (vector) entry we find
\[ \bigl\langle p^{a_2}\zp,p^{a_3}\zp,p^{a_4}\zp \bigr\rangle + \bigl\langle p^{a_2}\zp,p^{a_3}\zp,p^{a_4}\zp \bigr\rangle - \bigl\langle p^{a_6+a_7}\zp,p^{a_5+a_7}\zp,p^{a_5+a_6}\zp \bigr\rangle \]
This will be contained in $\Lambda$ if and only if $a_6+a_7 \geq a_2$, $a_5+a_7 \geq a_3$, and $a_5+a_6 \geq a_4$, which gives us three more of our relations.  The final three relations can be found be examining the lower-left (vector) entry in the same way.
\end{proof}

It is clear from this argument that our standard lattice sequence $\underline\Lambda_r$ satisfies all of these relations trivially. It is also clear that writing out octonion elements in this way is notationally cumbersome.  We introduce some new notation to help:

\begin{definition}
We will use the following notation to describe the relevant lattices in $\oo$:
\begin{align*}
\Lambda \left(
  \begin{Array}{ccc|c}
    a_2 & a_3 & a_4 & a_1 \\
    a_5 & a_6 & a_7 & a_8 \\
  \end{Array}
\right)
& = \left(
  \begin{Array}{cc}
    p^{a_1}\zp & \bigl\langle p^{a_2}\zp,p^{a_3}\zp,p^{a_4}\zp \bigr\rangle \\
    \bigl\langle p^{a_5}\zp,p^{a_6}\zp,p^{a_7}\zp \bigr\rangle & p^{a_8}\zp \\
  \end{Array}
\right).
\end{align*}
If the diagonal valuations are absent, they are to be interpreted as zeros.  Likewise, we will denote lattice filtrations in $\oo$, indexed by $r \in\rr$, by:
\begin{align*}
\Lambda_r \left(
  \begin{Array}{ccc|c}
    a_2 & a_3 & a_4 & a_1 \\
    a_5 & a_6 & a_7 & a_8 \\
  \end{Array}
\right) & = \left\{ p^{\lfloor r \rfloor} \cdot \Lambda \left(
  \begin{Array}{ccc|c}
    a_2 & a_3 & a_4 & a_1 \\
    a_5 & a_6 & a_7 & a_8 \\
  \end{Array}
\right) \right\}_{r \in \rr} \\
& = \left\{ \Lambda \left(
  \begin{Array}{ccc|c}
    a_2+ \lfloor r \rfloor & a_3+ \lfloor r \rfloor & a_4+ \lfloor r \rfloor & a_1+ \lfloor r \rfloor \\
    a_5+ \lfloor r \rfloor & a_6+ \lfloor r \rfloor & a_7+ \lfloor r \rfloor & a_8+ \lfloor r \rfloor \\
  \end{Array}
\right) \right\}_{r \in \rr}.
\end{align*}
\end{definition}

Once again, due to the fact that $k=\qp$ has a discrete valuation, all of these lattice filtrations are also lattice sequences.  The individual lattices in the sequence may or may not be orders, based on whether or not the indices satisfy the relations in Proposition \ref{orderconditions}.

\begin{definition}\label{symbol:duallattice}
Given a lattice $\Lambda$ in $\oo$ as defined above, the \textbf{dual lattice} associated to $\Lambda$ is defined to be
\[ \Lambda^\ast  = \{ x \in \oo \mid T(xy) \in \zp, \forall y \in \Lambda \}. \]
\end{definition}

\begin{proposition}
Let $\Lambda$ be a lattice in $\oo$:
\[ \Lambda = \Lambda \left(
  \begin{Array}{ccc|c}
    a_2 & a_3 & a_4 & a_1 \\
    a_5 & a_6 & a_7 & a_8 \\
  \end{Array}
\right). \]
Then
\[ \Lambda^\ast = \Lambda \left(
  \begin{Array}{ccc|c}
    -a_5 & -a_6 & -a_7 & -a_8 \\
    -a_2 & -a_3 & -a_4 & -a_1 \\
  \end{Array}
\right). \]
\end{proposition}

\begin{proof}
We perform a simple calculation (on the level of sets) to determine the octonion elements which satisfy this condition. Set
\[ \Lambda^\ast = \left(
  \begin{Array}{cc}
    p^{d_1}\zp & \bigl\langle p^{d_2}\zp,p^{d_3}\zp,p^{d_4}\zp \bigr\rangle \\
    \bigl\langle p^{d_5}\zp,p^{d_6}\zp,p^{d_7}\zp \bigr\rangle & p^{d_8}\zp \\
  \end{Array}
\right). \]
Then we multiply $\Lambda^\ast \cdot \Lambda$, writing our vectors as columns for convenience:
\[
\left(
  \begin{Array}{cc}
    p^{d_1}\zp & \left(
                   \begin{Array}{c}
                     p^{d_2}\zp \\
                     p^{d_3}\zp \\
                     p^{d_4}\zp \\
                   \end{Array}
                 \right) \\
    \left(
                   \begin{Array}{c}
                     p^{d_5}\zp \\
                     p^{d_6}\zp \\
                     p^{d_7}\zp \\
                   \end{Array}
                 \right) & p^{d_8}\zp
  \end{Array}
\right)\left(
  \begin{Array}{cc}
    p^{a_1}\zp & \left(
                   \begin{Array}{c}
                     p^{a_2}\zp \\
                     p^{a_3}\zp \\
                     p^{a_4}\zp \\
                   \end{Array}
                 \right) \\
    \left(
                   \begin{Array}{c}
                     p^{a_5}\zp \\
                     p^{a_6}\zp \\
                     p^{a_7}\zp \\
                   \end{Array}
                 \right) & p^{a_8}\zp
  \end{Array}
\right).
\]
Since we are only concerned with the trace of this product, we only consider its diagonal entries.  For the upper-left entry, we find that
\[ p^{d_1+a_1}\zp + p^{d_2 + a_5}\zp + p^{d_3 + a_6}\zp + p^{d_4 + a_7}\zp, \]
which will be in $\zp$ for all element of $\Lambda$ precisely when:
\begin{align*}
d_1 & = -a_1, \quad\quad d_3 = -a_6, \\
d_2 & = -a_5, \quad\quad d_4 = -a_7.
\end{align*}
Likewise, the lower-right entry will yield
\[ p^{d_5 + a_2}\zp + p^{d_6 + a_3}\zp + p^{d_7 + a_4}\zp + p^{d_8 + a_8}\zp, \]
which will be in $\zp$ for all element of $\Lambda$ precisely when:
\begin{align*}
d_5 & = -a_2, \quad\quad d_7 = -a_4,\\
d_6 & = -a_3, \quad\quad d_8 = -a_8.
\end{align*}
Translating this result to the notation we have defined gives us the desired result.
\end{proof}

\section{Maximinorante Algebra Valuations}
\label{algebravaluations}

In a general context, let $V$ be any unital $k$-algebra with $k$ a local non-archimedean field.  We would like to extend our regular valuation $\val \colon k \rightarrow \zz$ to a valuation on all of $V$, denoted also by $\val$, in such a way that the two valuations coincide on $k \subseteq V$.

\begin{definition}[\cite{GanYu2003}]\label{symbol:val}
Let $V$ be any unital $k$-algebra with $k$ a local non-archimedean field.  A \textbf{valuation}\footnote{In \cite{GanYu2003}, these maps are called a \textbf{norm} on $V$. Since we will work with composition algebras which already have a norm $N$, we will avoid this term in the current context.} on $V$ is a map $\val \colon V \rightarrow \rr \cup \{ \infty \}$ such that for all $v,w \in V$:
\begin{enumerate}
\item $\val(v+w) \geq \inf \{ \val(v),\val(w) \}$.

\item $\val(av) = \val(a) + \val(v)$ for all $a \in k$.

\item $\val(v) = \infty$ if and only if $v=0$.
\end{enumerate}
We call $\val$ an \textbf{algebra valuation} if it additionally satisfies:
\begin{enumerate}
\addtocounter{enumi}{3}
\item $\val(v \cdot w) \geq \val(v) + \val(w)$.
\end{enumerate}
\end{definition}

We reiterate that in expression (b), the valuation on $a$ is the regular valuation defined on $k$, while the valuation on $v$ and $av$ is the one defined on $V$.

Since the $k$-algebra that we would like to discuss is also a composition algebra with a quadratic norm $N$ and an associated nondegenerate bilinear form $B$, we give some adjectives to describe the relationship between these structures.

\begin{definition}(\cite{GanYu2003})
Let $(V,N)$ be a composition algebra over $k$, with $\text{char}\, k \neq 2$.  We say that an algebra valuation $\val$ on $V$ \textbf{minorizes} $N$ and $B$ if, for all $v,w \in V$,
\[ \val(B(v,w)) \geq \val(v) + \val(w). \]
Among valuations, we say that $\val_1 \geq \val_2$ if $\val_1(v) \geq \val_2(v)$ for all $v \in V$.  Thus, if $\val$ is a maximal element in the set of algebra valuations minorizing $N$ and $B$, we say that $\val$ is \textbf{maximinorante}.
\end{definition}

\begin{comment}(No need for gradings in this context.)
Now, we again restrict our attention to the algebra $V = \oo$.  To establish a connection between the algebra valuations we have just defined and the lattice sequences of Section \ref{latticesandorders}, we describe one more structure that can be placed on the lattice sequences:

\begin{definition}
Let $\{ \Lambda_r \}$ be a lattice sequence in $\oo$.  A \textbf{grading} on $\{ \Lambda_r \}$ is a strictly decreasing (with respect to containment in the sequence) map $\Gamma \colon \{ \Lambda_r \} \rightarrow \rr$ such that for each $a \in k$ and $\Lambda \in \{ \Lambda_r \}$:
\[ \Gamma( a \cdot \Lambda ) = \val(a) + \Gamma( \Lambda ). \]
\end{definition}
\end{comment}

\begin{proposition}
\label{prop:valuationbijection}
There is a bijection between the set of valuations $\val$ on $\oo$, and the set of lattice sequences on $\oo$.
\end{proposition}

\begin{proof}
Given a valuation $\val \colon \oo \rightarrow \rr \cup \{ \infty \}$, we construct the individual lattices
\[ \Lambda^{\val}_r := \{ x \in \oo \mid \val(x) \geq r \}. \]
Since $\val(p \cdot \Lambda) = \val(p) + \val(\Lambda) = 1 + \val(\Lambda)$, we have that
\[ \Lambda^{\val}_{r+1} = \{ x \in \oo \mid \val(x) \geq r+1 \} = p \cdot \Lambda^{\val}_{r}, \]
and this is indeed a lattice sequence in $\oo$.

\begin{comment}
A grading on this lattice sequence is given by
\[ \Gamma( \Lambda^{\val}_r ) := \inf_{x \in \Lambda^{\val}_r } \val(x). \]
We verify that for every lattice $\Lambda$ in the sequence, we have
\begin{align*}
\Gamma( a\Lambda ) & = \inf_{x \in a\Lambda } \val(x) \\
    & = \inf_{x \in \Lambda } \val(ax) \\
    & = \inf_{x \in \Lambda } \val(a) + \val(x) \\
    & = \val(a) + \inf_{x \in \Lambda } \val(x) \\
    & = \val(a) + \Gamma( \Lambda ).
\end{align*}
\end{comment}

Conversely, let $\{ \Lambda_r \}$ be a lattice sequence, and let $x \in \oo$.  Let $\Lambda_x$ be the smallest (with respect to containment in the lattice sequence) member of $\{ \Lambda_r \}$ containing $x$, and let $\val(x) : = r$.  Note that since $0 \in \oo$ is contained in every lattice of every sequence, there is no such `smallest member' of $\{ \Lambda_r \}$ which contains it, and from this we get condition (c) of the definition of algebra valuations.  Condition (b) follows from the fact that for $x \in \oo$ and $a \in \rr$ we have $\Lambda_{ax} = a\Lambda_{x}$.

Finally, for condition (a), let $x,y \in \oo$.  Then in the sequence $\{ \Lambda_r \}$, we may assume (without loss of generality) that $\Lambda_x \subset \Lambda_y$ and $\val(x) \geq \val(y)$.  Then, since $\Lambda_y$ contains both $x$ and $y$ and thus $x+y$, and $\Lambda_{x+y}$ is the smallest such lattice, we have $\Lambda_{x+y} \subset \Lambda_y$.  Therefore
\[ \val(x+y) = \inf\{\val(x),\val(y) \}. \]
\end{proof}

Given the above bijection, we find that the definitions of this chapter are all connected via the following important theorem.

\begin{theorem}[{\cite[Theorem $7.3$]{GanYu2003}}]
\label{ganyumaintheorem}
There is a bijection between the points in the building $\mathcal{B}(G)$ of $G = \text{Aut}(\oo)$ and the set of maximinorante algebra valuations for $(\oo,N)$.  In this bijection, the type $1$ points of $\mathcal{B}(G)$ correspond to those algebra valuations which take values in $\val(k)$, which in turn correspond to maximal orders in $\oo$.
\end{theorem}

We now have several different but related structures: maximinorante algebra valuations on $\oo$, certain lattice sequences and orders in $\oo$, and points in our apartment $\mathcal{A}$.  Theorem \ref{ganyumaintheorem} establishes a direct connection between these objects, and will allow us to label our apartment $\mathcal{A}$ in significant detail.

\section{Orders of Type $1$}
\label{type1points}

By Theorem \ref{ganyumaintheorem}, the type $1$ points are in bijection with the maximal orders in $\oo$.  We will denote the set of maximal orders by $\mathcal{V}_1$, and we describe some of these maximal orders below.

We recall that $\text{SL}_3(k) \subset G_2$ acts on octonion elements by the map $\theta$, as defined in Section \ref{SL3}.  Therefore it also acts on lattices in $\oo$.  In particular, consider the following toral elements of $T_s \subset \text{SL}_3(k)$:
\[ g = \left(
     \begin{Array}{ccc}
       p & 0 & 0 \\
       0 & p^{-1} & 0 \\
       0 & 0 & 1 \\
     \end{Array}
   \right) \quad\text{and}\quad g^{-T} = g^{-1} = \left(
     \begin{Array}{ccc}
       p^{-1} & 0 & 0 \\
       0 & p & 0 \\
       0 & 0 & 1 \\
     \end{Array}
   \right).
 \]
Applying the associated automorphism $\theta(g)$ to the standard lattice sequence gives the following result:
\begin{align*}
\theta(g)\left( \underline\Lambda \right) & = \left(
  \begin{Array}{cc}
    \zp & g \cdot \zp^3 \\
    g^{-T} \cdot \zp^3 & \zp \\
  \end{Array}
\right) \\
& = \left(
  \begin{Array}{cc}
    \zp & \bigl\langle p\,\zp,p^{-1}\zp,\zp \bigr\rangle  \\
    \bigl\langle p^{-1}\zp,p\,\zp,\zp \bigr\rangle & \zp \\
  \end{Array}
\right) \\
& = \Lambda\left(
  \begin{Array}{ccc|c}
    1 & -1 & 0 & 0 \\
    -1 & 1 & 0 & 0 \\
  \end{Array}
\right).
\end{align*}

\begin{proposition}
The lattice
\[ \Lambda = \Lambda\left(
  \begin{Array}{ccc|c}
    1 & -1 & 0 & 0 \\
    -1 & 1 & 0 & 0 \\
  \end{Array}
\right) \]
is a maximal order in $\oo$.
\end{proposition}

\begin{proof}
It is easily checked that $\Lambda$ satisfies the properties of Proposition \ref{orderconditions}.  In particular, the sum of each column in the matrix of valuations is equal to zero.  Any lattice containing $\Lambda$ must have an decreased valuation in at least one entry, which would cause the relevant sum to become negative, in which case the lattice in question would not be an order.
\end{proof}

Note that there are six orders (total) which are analogous to this one, each produced by an analogous toral element $h_\alpha(p)$, where $p$ is the uniformizer and $\alpha$ ranges over the six long roots, associated to the six short coroots. These $h_\alpha(p)$ can be interpreted as acting by $\theta$ automorphisms via the matrices
\[ \left(
     \begin{Array}{ccc}
       p & 0 & 0 \\
       0 & p^{-1} & 0 \\
       0 & 0 & 1 \\
     \end{Array}
   \right), \left(
     \begin{Array}{ccc}
       1 & 0 & 0 \\
       0 & p & 0 \\
       0 & 0 & p^{-1} \\
     \end{Array}
   \right), \left(
     \begin{Array}{ccc}
       p^{-1} & 0 & 0 \\
       0 & 1 & 0 \\
       0 & 0 & p \\
     \end{Array}
   \right), \]
and their inverses.  The six maximal orders created by these toral elements are identified in yellow in Figure \ref{Type1Vertices}.  In this way, the torus of $\text{SL}_3(k)$, which is equal to the torus $T$ in $G_2$ by Proposition \ref{prop:thetorus} and the same torus used to identify the coroot lattice and the apartment, acts on its apartment by translation.  More specifically, $T$ acts transitively on the type $1$ vertices of $\mathcal{A}$.

\newpage

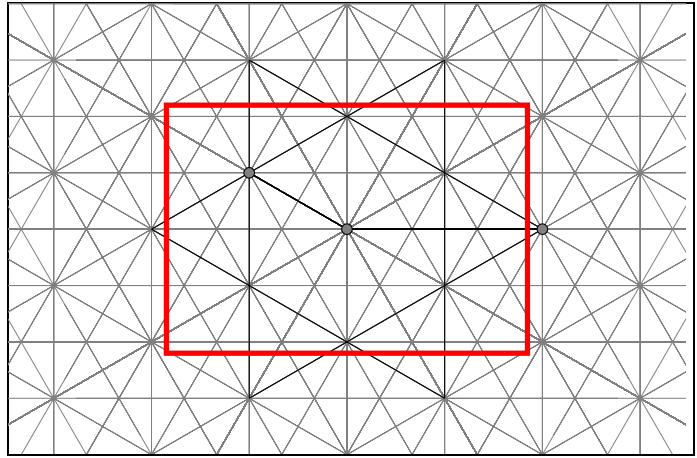
\begin{figure}[ht!]
\begin{center}
\frame{%
\begin{tikzpicture}[thin,scale=1.5]%
    \clip (-3,-2) rectangle (3,2);
    \draw[help lines] \foreach \z in {0,1,...,5} {
            \foreach \y in {0,1,...,5} {
              \foreach \x in {0,1,...,11} {
        (0:0) -- ++(30*\x:5)
        (60*\y + 30:\z) -- ++(30*\x:5)
    }}};
    \draw \foreach \x in {0,120,240} {
        (\x+120:1.732) -- (\x+240:1.732)
        (\x-60:1.732) -- (\x+60:1.732)
        (0:0) -- (0:1.732)
        (0:0) -- (150:1)};
    \draw{ (0:0) node[circle, draw, fill=black!50,
                        inner sep=0pt, minimum width=4pt]{}
        (0:1.732) node[circle, draw, fill=black!50,
                        inner sep=0pt, minimum width=4pt]{}
        (150:1) node[circle, draw, fill=black!50,
                        inner sep=0pt, minimum width=4pt]{}
    };
    \begin{scope}[color=red,line width=2pt]
    \draw (-1.6,-1.1) rectangle (1.6,1.1);
    \end{scope}
\end{tikzpicture}
}
\caption{The Highlighted Area of the Apartment}\label{zoom}
\end{center}
\end{figure}

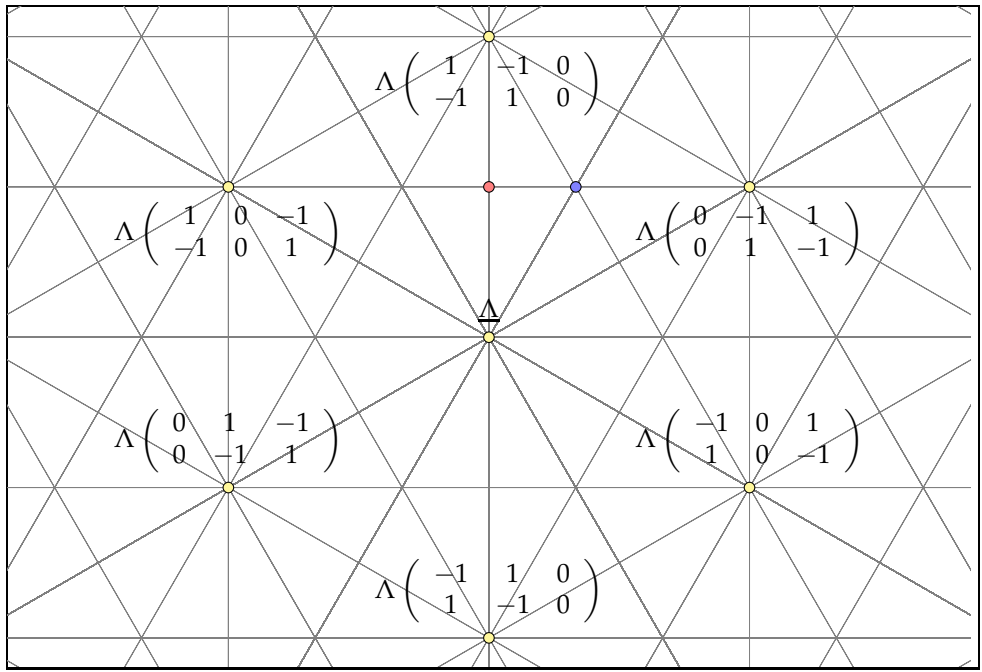
\begin{figure}[ht!]
\begin{center}
\frame{%
\begin{tikzpicture}[thin,scale=2]%
    \clip (-3.2,-2.2) rectangle (3.2,2.2);
    \draw[help lines] \foreach \z in {0,1,...,5} {
            \foreach \y in {0,1,...,5} {
              \foreach \x in {0,1,...,11} {
        (0:0) -- ++(30*\x:10)
        (60*\y + 30:2*\z) -- ++(30*\x:10)
    }}};
    \draw{ (0:0) node[circle, draw, fill=yellow!50,
                        inner sep=0pt, minimum width=4pt, label=above: $\underline\Lambda$]{}
        (30:2) node[circle, draw, fill=yellow!50,
                        inner sep=0pt, minimum width=4pt, label=below: $\Lambda\left(
  \begin{Array}{ccc}
    0 & -1 & 1 \\
    0 & 1 & -1 \\
  \end{Array}
\right)$]{}
        (90:2) node[circle, draw, fill=yellow!50,
                        inner sep=0pt, minimum width=4pt, label=below: $\Lambda\left(
  \begin{Array}{ccc}
    1 & -1 & 0 \\
    -1 & 1 & 0 \\
  \end{Array}
\right)$]{}
        (150:2) node[circle, draw, fill=yellow!50,
                        inner sep=0pt, minimum width=4pt, label=below: $\Lambda\left(
  \begin{Array}{ccc}
    1 & 0 & -1 \\
    -1 & 0 & 1 \\
  \end{Array}
\right)$]{}
        (210:2) node[circle, draw, fill=yellow!50,
                        inner sep=0pt, minimum width=4pt, label=above: $\Lambda\left(
  \begin{Array}{ccc}
    0 & 1 & -1 \\
    0 & -1 & 1 \\
  \end{Array}
\right)$]{}
        (270:2) node[circle, draw, fill=yellow!50,
                        inner sep=0pt, minimum width=4pt, label=above: $\Lambda\left(
  \begin{Array}{ccc}
    -1 & 1 & 0 \\
    1 & -1 & 0 \\
  \end{Array}
\right)$]{}
        (330:2) node[circle, draw, fill=yellow!50,
                        inner sep=0pt, minimum width=4pt, label=above: $\Lambda\left(
  \begin{Array}{ccc}
    -1 & 0 & 1 \\
    1 & 0 & -1 \\
  \end{Array}
\right)$]{}
    };
    \draw{ (90:1) node[circle, draw, fill=red!50,
                        inner sep=0pt, minimum width=4pt]{}
    };
    \draw{ (60:1.1547) node[circle, draw, fill=blue!50,
                        inner sep=0pt, minimum width=4pt]{}
    };
\end{tikzpicture}
}
\caption{Maximal orders associated to the vertices of type $1$.}\label{Type1Vertices}
\end{center}
\end{figure}

\newpage

While the conditions given by Gan and Yu have been sufficient to determine the maximal orders associated to each of our type $1$ vertices, according to Theorem \ref{ganyumaintheorem}, there are also lattice sequences (resp. maximinorante valuations) associated to each of these points of $\mathcal{A}$.

Creating a lattice sequence from each of our type $1$ orders $\Lambda$ is straightforward, by noting that
\[ \cdots \subsetneq p^2 \Lambda \subsetneq p \Lambda \subsetneq \Lambda \subsetneq p^{-1} \Lambda \subsetneq p^{-2} \Lambda \subsetneq \cdots \]

\begin{comment}
Since any grading $\Gamma$ on this sequence should satisfy $\Gamma(a \cdot \Lambda) = \val(a) + \Gamma(\Lambda)$, there is also an obvious choice of grading, where
\[ \Lambda_i = p^{i} \Lambda \quad \text{and} \quad \Gamma(\Lambda_i) = i. \]
\end{comment}

This construction is independent of the type $1$ maximal order $\Lambda$ that we choose to start from.  Now that we have a lattice sequence identified with each $\Lambda$, we also have an associated algebra valuation as defined in Proposition \ref{prop:valuationbijection}; for each $x \in \mathcal{A}$:
\[ \val_1(x) = \sup \{ r \in \rr \mid x \in \Lambda_r). \]

In this case, we have identified all our structures by first starting with a maximal order in $\oo$, constructing a lattice sequence, and then identifying an algebra valuation $\val_1$ constructed from that lattice sequence.  For later points in $\mathcal{A}$, we will need to do this in a reverse order, by first defining an algebra valuation and then constructing our lattice sequence.  When that occurs, it will be useful for us to have a more detailed description of $\val_1$.  Toward that end, we choose a standard basis $\mathfrak{s} = \{ b_{\pm 1},b_{\pm 2},b_{\pm 3},b_{\pm 4} \}$ for $\oo$ as an $8$-dimensional $k$-vector space, with
\begin{align*}
b_1 & = \left(
        \begin{Array}{cc}
          0 & \langle 1,0,0 \rangle \\
          0 & 0 \\
        \end{Array}
      \right), \quad\quad\quad b_{-1} = \left(
        \begin{Array}{cc}
          0 & 0 \\
          \langle 1,0,0 \rangle & 0 \\
        \end{Array}
      \right), \\
b_2 & = \left(
        \begin{Array}{cc}
          0 & \langle 0,1,0 \rangle \\
          0 & 0 \\
        \end{Array}
      \right), \quad\quad\quad b_{-2} = \left(
        \begin{Array}{cc}
          0 & 0 \\
          \langle 0,1,0 \rangle & 0 \\
        \end{Array}
      \right), \\
b_3 & = \left(
        \begin{Array}{cc}
          0 & \langle 0,0,1 \rangle \\
          0 & 0 \\
        \end{Array}
      \right), \quad\quad\quad b_{-3} = \left(
        \begin{Array}{cc}
          0 & 0 \\
          \langle 0,0,1 \rangle & 0 \\
        \end{Array}
      \right), \\
b_4 & = \left(
        \begin{Array}{cc}
          1 & 0 \\
          0 & 0 \\
        \end{Array}
      \right),  \quad\quad\quad \hspace{10mm} b_{-4} = \left(
        \begin{Array}{cc}
          0 & 0 \\
          0 & 1 \\
        \end{Array}
      \right).
\end{align*}

For each maximal order $\Lambda$, define an intermediate function (depending on $\Lambda$) $v_1^{\Lambda} \colon \mathfrak{s} \rightarrow \zz$ on these basis elements, and extend to scalar multiples using the identity $v_1^{\Lambda}(a_i b_i) = \val(a_i) + v_1^{\Lambda}(b_i)$.  Finally, define our valuation:
\[ \val_1(x) := \min_i \left\{ v_1^{\Lambda}(a_i b_i) \right\}, \quad \text{for } x = \sum a_i b_i \in \oo. \]

For example, given the lattice sequence constructed from the order
\[ \Lambda = \Lambda\left(
  \begin{Array}{ccc}
    1 & -1 & 0 \\
    -1 & 1 & 0 \\
  \end{Array}
\right),\]
we define the intermediate function
\begin{align} v_1^{\Lambda}(b_i) & = \left\{
  \begin{Array}{rl}
    0, & \quad\text{if } i \in \{\pm4,\pm3\} \\
    1, & \quad\text{if } i \in \{2,-1\} \\
    -1, & \quad\text{if } i \in \{1,-2\}.
  \end{Array}
\right.
\end{align}

Then, for any lattice element $x \in \Lambda_0 = \Lambda\left(
  \begin{Array}{ccc}
    1 & -1 & 0 \\
    -1 & 1 & 0 \\
  \end{Array}
\right)$, we have that
\begin{align*}
\val_1(x) & = \min_i \left\{ v_1^{\Lambda}(a_i b_i) \right\} \\
    & = \min_i \left\{ \val(a_i) + v_1^{\Lambda}(b_i) \right\} \\
    & = \min \left\{ \val(a_1) -1, \val(a_2) + 1, \val(a_{-1}) + 1, \val(a_{-2}) -1, \val(a_{\pm3}), \val(a_{\pm4}) \right\} \\
    & \geq 0.
\end{align*}

Likewise, for any lattice element $x \in \Lambda_2 = \Lambda\left(
  \begin{Array}{ccc|c}
    3 & 1 & 2 & 2 \\
    1 & 3 & 2 & 2 \\
  \end{Array}
\right)$, we have that
\begin{align*}
\val_1(x) & = \min_i \left\{ v_1^{\Lambda}(a_i b_i) \right\} \\
    & = \min_i \left\{ \val(a_i) + v_1^{\Lambda}(b_i) \right\} \\
    & = \min \left\{ \val(a_1) -1, \val(a_2) + 1, \val(a_{-1}) + 1, \val(a_{-2}) -1, \val(a_{\pm3}), \val(a_{\pm4}) \right\} \\
    & \geq 2.
\end{align*}

As another example, given the lattice sequence constructed from the order
\[ \Lambda = \Lambda\left(
  \begin{Array}{ccc}
    0 & -1 & 1 \\
    0 & 1 & -1 \\
  \end{Array}
\right),\]
we define a distinct intermediate function
\[ v_1^{\Lambda}(b_i) = \left\{
  \begin{Array}{rl}
    0, & \quad\text{if } i \in \{\pm4,\pm1\} \\
    1, & \quad\text{if } i \in \{2,-3\} \\
    -1, & \quad\text{if } i \in \{3,-2\}.
  \end{Array}
\right. \]
This will define another algebra valuation consistent with this new type $1$ vertex.

As a final example, for the standard lattice sequence $\underline\Lambda_r$, we can simply take $v_1^{\Lambda}(b_i) = 0$ for all $b_i$, in which case our valuation becomes $\val_1(x) = \min\left\{ \val(a_i) \right\}$.  This defines an algebra valuation unique to the origin of $\mathcal{A}$.

It is easily verified that these algebra valuations are consistent with those that we have defined formally in Proposition \ref{prop:valuationbijection}, and satisfy all the needed properties and necessary relations to their respective lattice sequences.  Other, analogous algebra valuations may be identified for the lattice sequences corresponding to any other type $1$ point in $\mathcal{A}$.

\section{Orders of Type $2$}
\label{type2points}

With regard to the type $2$ vertices of the apartment $\mathcal{A}$, we use the following proposition which is a restatement of Lemma $9.4$ and Theorem $9.14$ of \cite{GanYu2003}.

\begin{proposition}[\cite{GanYu2003}]
\label{type2conditions}
Let $\mathcal{V}_2$ be the set of orders $\Lambda$ in $\oo$ satisfying:
\begin{enumerate}
\item $\Lambda \subsetneq \Lambda^\ast \subsetneq p^{-1}\Lambda$.

\item $\Lambda^{\ast 2} \subset p^{-1}\Lambda$.
\end{enumerate}
Then $\mathcal{V}_2$ is in bijection with the set of vertices of type $2$ in $\mathcal{B}(G_2)$.  Let $\Lambda_1 \in \mathcal{V}_1$ and $\Lambda_2 \in \mathcal{V}_2$.  Let $x_1$ be the type $1$ vertex in $\mathcal{B}(G_2)$ associated to $\Lambda_1$, $x_2$ be the type $2$ vertex associated to $\Lambda_2$.  Then $x_1$ is incident to $x_2$ if and only if $\Lambda_2 \subset \Lambda_1$.
\end{proposition}

\begin{proposition}
Let $\Lambda$ be the following lattice in $\oo$:
\[ \Lambda = \Lambda\left(
  \begin{Array}{ccc|c}
    1 & 0 & 0 & 0\\
    0 & 1 & 0 & 0\\
  \end{Array}
\right). \]
Then $\Lambda$ is a type $2$ order in $\oo$; i.e., it satisfies the conditions of Proposition \ref{type2conditions}.
\end{proposition}

\begin{proof}
That $\Lambda$ is an order is easily checked by verifying the conditions of Proposition \ref{orderconditions}. We also have $\Lambda^\ast = \Lambda\left(
  \begin{Array}{ccc|c}
    0 & -1 & 0 & 0\\
    -1 & 0 & 0 & 0\\
  \end{Array}
\right)$,
and clearly
\[ \Lambda\left(
  \begin{Array}{ccc|c}
    1 & 0 & 0 & 0\\
    0 & 1 & 0 & 0\\
  \end{Array}
\right) \subsetneq \Lambda\left(
  \begin{Array}{ccc|c}
    0 & -1 & 0 & 0\\
    -1 & 0 & 0 & 0\\
  \end{Array}
\right) \subsetneq \Lambda\left(
  \begin{Array}{ccc|c}
    0 & -1 & -1 & -1\\
    -1 & 0 & -1 & -1\\
  \end{Array}
\right). \]
so condition (a) is satisfied.  Next we calculate $\Lambda^{\ast 2}$:
\begin{align*}
\Lambda^{\ast 2} & = \Lambda\left(
  \begin{Array}{ccc|c}
    0 & -1 & 0 & 0\\
    -1 & 0 & 0 & 0\\
  \end{Array}
\right)^2 \\
    & = \left(
  \begin{Array}{cc}
    \zp & \left(
                   \begin{Array}{c}
                     \zp \\
                     p^{-1}\zp \\
                     \zp \\
                   \end{Array}
                 \right) \\
    \left(
                   \begin{Array}{c}
                     p^{-1}\zp \\
                     \zp \\
                     \zp \\
                   \end{Array}
                 \right) & \zp
  \end{Array}
\right)\left(
  \begin{Array}{cc}
    \zp & \left(
                   \begin{Array}{c}
                     \zp \\
                     p^{-1}\zp \\
                     \zp \\
                   \end{Array}
                 \right) \\
    \left(
                   \begin{Array}{c}
                     p^{-1}\zp \\
                     \zp \\
                     \zp \\
                   \end{Array}
                 \right) & \zp
  \end{Array}
\right) \\
    & = \left(
  \begin{Array}{cc}
    p^{-1}\zp & \left(
                   \begin{Array}{c}
                     \zp \\
                     p^{-1}\zp \\
                     p^{-1}\zp \\
                   \end{Array}
                 \right) \\
    \left(
                   \begin{Array}{c}
                     p^{-1}\zp \\
                     \zp \\
                     p^{-1}\zp \\
                   \end{Array}
                 \right) & p^{-1}\zp
  \end{Array}
\right) \\
    & = \Lambda\left(
  \begin{Array}{ccc|c}
    0 & -1 & -1 & -1\\
    -1 & 0 & -1 & -1\\
  \end{Array}
\right).
\end{align*}
In this case, we have that $\Lambda^{\ast 2} = p^{-1}\Lambda$, so condition (b) is also satisfied and $\Lambda$ is a type $2$ order in $\oo$.
\end{proof}

Since this order is contained in both the standard lattice sequence $\underline\Lambda$ and also in the sequence $\Lambda\left(
  \begin{Array}{ccc|c}
    1 & 0 & 0 & 0\\
    0 & 1 & 0 & 0\\
  \end{Array}
\right)$, the vertex associated to it in $\mathcal{A}$ is incident to both of the associated vertices of the other two.  But there is only one such vertex in the building, labeled in red in Figure \ref{Type1Vertices}.

There will likewise be five more orders of a similar form, which we also identify in red and place in their appropriate positions in Figure \ref{Type2Vertices}.  The type $1$ vertices continue to be shown in yellow for reference.

\begin{figure}[ht!]
\begin{center}
\frame{%
\begin{tikzpicture}[thin,scale=2]%
    \clip (-3.2,-2.2) rectangle (3.2,2.2);
    \draw[help lines] \foreach \z in {0,1,...,5} {
            \foreach \y in {0,1,...,5} {
              \foreach \x in {0,1,...,11} {
        (0:0) -- ++(30*\x:10)
        (60*\y + 30:2*\z) -- ++(30*\x:10)
    }}};
    \draw{ (0:0) node[circle, draw, fill=yellow!50,
                        inner sep=0pt, minimum width=4pt, label=above: $\underline\Lambda$]{}
        (30:2) node[circle, draw, fill=yellow!50,
                        inner sep=0pt, minimum width=4pt]{}
        (90:2) node[circle, draw, fill=yellow!50,
                        inner sep=0pt, minimum width=4pt]{}
        (150:2) node[circle, draw, fill=yellow!50,
                        inner sep=0pt, minimum width=4pt]{}
        (210:2) node[circle, draw, fill=yellow!50,
                        inner sep=0pt, minimum width=4pt]{}
        (270:2) node[circle, draw, fill=yellow!50,
                        inner sep=0pt, minimum width=4pt]{}
        (330:2) node[circle, draw, fill=yellow!50,
                        inner sep=0pt, minimum width=4pt]{}
    };
    \draw{ (30:1) node[circle, draw, fill=red!50,
                        inner sep=0pt, minimum width=4pt, label=right: $\Lambda\left(
  \begin{Array}{ccc|c}
    0 & 0 & 1 & 0\\
    0 & 1 & 0 & 0\\
  \end{Array}
\right)$]{}
        (90:1) node[circle, draw, fill=red!50,
                        inner sep=0pt, minimum width=4pt, label=above: $\Lambda\left(
  \begin{Array}{ccc|c}
    1 & 0 & 0 & 0\\
    0 & 1 & 0 & 0\\
  \end{Array}
\right)$]{}
        (150:1) node[circle, draw, fill=red!50,
                        inner sep=0pt, minimum width=4pt, label=left: $\Lambda\left(
  \begin{Array}{ccc|c}
    1 & 0 & 0 & 0\\
    0 & 0 & 1 & 0\\
  \end{Array}
\right)$]{}
        (210:1) node[circle, draw, fill=red!50,
                        inner sep=0pt, minimum width=4pt, label=left: $\Lambda\left(
  \begin{Array}{ccc|c}
    0 & 1 & 0 & 0\\
    0 & 0 & 1 & 0\\
  \end{Array}
\right)$]{}
        (270:1) node[circle, draw, fill=red!50,
                        inner sep=0pt, minimum width=4pt, label=below: $\Lambda\left(
  \begin{Array}{ccc|c}
    0 & 1 & 0 & 0\\
    1 & 0 & 0 & 0\\
  \end{Array}
\right)$]{}
        (330:1) node[circle, draw, fill=red!50,
                        inner sep=0pt, minimum width=4pt, label=right: $\Lambda\left(
  \begin{Array}{ccc|c}
    0 & 0 & 1 & 0\\
    1 & 0 & 0 & 0\\
  \end{Array}
\right)$]{}
    };
    \draw{ (60:1.1547) node[circle, draw, fill=blue!50,
                        inner sep=0pt, minimum width=4pt]{}
    };
\end{tikzpicture}
}
\caption{Octonion orders associated to the vertices of type $2$.}\label{Type2Vertices}
\end{center}
\end{figure}

In this case, we define our related structures by first defining algebra valuations corresponding to each of these type $2$ points.  Using our geometric intuition from $\mathcal{A}$, and noting that each type $2$ vertex lies in between two type $1$ vertices, we can define the relevant valuation by averaging the two type $1$ valuations.

Following with our previous examples, take $\Lambda = \Lambda\left(
  \begin{Array}{ccc|c}
    1 & 0 & 0 & 0\\
    0 & 1 & 0 & 0\\
  \end{Array}
\right)$, and for our chosen basis $B$ of $\oo$, take the intermediate function
\[ v_2^{\Lambda}(b_i) = \frac{1}{2}v_1^{\Lambda}(b_i), \]
with the $v_1^{\Lambda}$ defined in $(8.1)$ of the last section.  That is:
\[ v_2^{\Lambda}(b_i) = \left\{
  \begin{Array}{rl}
    0, & \quad\text{if } i \in \{\pm4,\pm3\} \\
    1/2, & \quad\text{if } i \in \{2,-1\} \\
    -1/2, & \quad\text{if } i \in \{1,-2\}.
  \end{Array}
\right. \]

We use this to find the valuation of several octonion elements.  For a lattice element $x \in \Lambda = \Lambda\left(
  \begin{Array}{ccc|c}
    1 & 0 & 0 & 0\\
    0 & 1 & 0 & 0\\
  \end{Array}
\right)$ we have:
\begin{align*}
\val_2(x) & = \min \left\{ \val(a_1) - \frac{1}{2}, \val(a_2) + \frac{1}{2}, \val(a_{-1}) + \frac{1}{2}, \val(a_{-2}) -\frac{1}{2}, \val(a_{\pm3}), \val(a_{\pm4}) \right\} \\
    & \geq \min \left\{ \frac{1}{2}, \frac{1}{2}, \frac{1}{2},\frac{1}{2}, 0, 0 \right\} \\
    & = 0.
\end{align*}

For a lattice element $x \in \Lambda^\ast = \Lambda\left(
  \begin{Array}{ccc|c}
    0 & -1 & 0 & 0\\
    -1 & 0 & 0 & 0\\
  \end{Array}
\right)$ we have:
\begin{align*}
\val_2(x) & = \min \left\{ \val(a_1) - \frac{1}{2}, \val(a_2) + \frac{1}{2}, \val(a_{-1}) + \frac{1}{2}, \val(a_{-2}) -\frac{1}{2}, \val(a_{\pm3}), \val(a_{\pm4}) \right\} \\
    & \geq \min \left\{ -\frac{1}{2}, -\frac{1}{2}, -\frac{1}{2},-\frac{1}{2}, 0, 0 \right\} \\
    & = -\frac{1}{2}.
\end{align*}

For a lattice element $x \in p^{-1}\Lambda = \Lambda\left(
  \begin{Array}{ccc|c}
    0 & -1 & -1 & -1\\
    -1 & 0 & -1 & -1\\
  \end{Array}
\right)$ we have:
\begin{align*}
\val_2(x) & = \min \left\{ \val(a_1) - \frac{1}{2}, \val(a_2) + \frac{1}{2}, \val(a_{-1}) + \frac{1}{2}, \val(a_{-2}) -\frac{1}{2}, \val(a_{\pm3}), \val(a_{\pm4}) \right\} \\
    & \geq \min \left\{ -\frac{1}{2}, -\frac{1}{2}, -\frac{1}{2},-\frac{1}{2}, -1, -1 \right\} \\
    & = -1.
\end{align*}

This valuation leads us to the lattice sequence
\[ \cdots \subsetneq p^{-1}\Lambda \subsetneq \Lambda^\ast \subsetneq \Lambda \subsetneq p(\Lambda^\ast) \subsetneq p\Lambda \subsetneq \cdots \]
Again, algebra valuations and lattice sequences for all type $2$ vertices can be found in this way.

\section{Orders of Type $3$}
\label{type3points}

Next, we turn to the vertices of type $3$ in $\mathcal{A}$.  Again, the following proposition is a restatement of Lemma $9.12$ and Theorem $9.14$ of \cite{GanYu2003}.

\begin{proposition}[\cite{GanYu2003}]
\label{type3conditions}
Let $\mathcal{V}_3$ be the set of orders $\Lambda$ in $\oo$ satisfying:
\begin{enumerate}
\item $\Lambda \subsetneq \Lambda^\ast \subsetneq p^{-1}\Lambda$.

\item $M := p \,\Lambda^{\ast 2} + \Lambda$ is a self-dual lattice.
\end{enumerate}
Then $\mathcal{V}_3$ is in bijection with the set of vertices of type $3$ in $\mathcal{B}(G_2)$.  Take $\Lambda_1 \in \mathcal{V}_1$, $\Lambda_2 \in \mathcal{V}_2$, and $\Lambda_3 \in \mathcal{V}_3$.  Let $x_i$ be the type $i$ vertex in $\mathcal{B}(G_2)$ associated to $\Lambda_i$.  Then:
\begin{enumerate}
\item $x_1$ is incident to $x_3$ if and only if $\Lambda_3 \subset \Lambda_1$.

\item $x_2$ is incident to $x_3$ if and only if $\Lambda_3 \subset \Lambda_2$.
\end{enumerate}
\end{proposition}

\begin{corollary}
The triangle formed in $\mathcal{B}(G_2)$ by $\{ x_1,x_2,x_3 \}$ is a chamber if and only if $\Lambda_3 \subset \Lambda_2 \subset \Lambda_1$.
\end{corollary}

To identify these type $3$ orders, it is best to use the incidence condition described in the above proposition.  That is, a type $3$ order must simultaneously be contained in each of the orders which have associated vertices incident to its vertex.  For example, we may examine the vertex marked in blue from Figure \ref{Type1Vertices}. It should be contained in $\underline\Lambda$ and also in each of the following orders:
\begin{align*}
\Lambda\left(
  \begin{Array}{ccc|c}
    1 & 0 & 0 & 0\\
    0 & 1 & 0 & 0\\
  \end{Array}
\right), & \quad \quad \Lambda\left(
  \begin{Array}{ccc|c}
    1 & -1 & 0 & 0\\
    -1 & 1 & 0 & 0\\
  \end{Array}
\right), \\
\Lambda\left(
  \begin{Array}{ccc|c}
    0 & 0 & 1 & 0\\
    0 & 1 & 0 & 0\\
  \end{Array}
\right), & \quad \quad \Lambda\left(
  \begin{Array}{ccc|c}
    0 & -1 & 1 & 0\\
    0 & 1 & -1 & 0\\
  \end{Array}
\right).
\end{align*}
It should also be contained in a sixth, as yet unidentified type $2$ order, but the information we have so far is sufficient to determine that our type $3$ lattice must be the following:
\[ \Lambda = \Lambda\left(
  \begin{Array}{ccc|c}
    1 & 0 & 1 & 0\\
    0 & 1 & 0 & 0\\
  \end{Array}
\right). \]

\begin{proposition}
Let $\Lambda$ be the following lattice in $\oo$:
\[ \Lambda = \Lambda\left(
  \begin{Array}{ccc|c}
    1 & 0 & 1 & 0\\
    0 & 1 & 0 & 0\\
  \end{Array}
\right). \]
Then $\Lambda$ is a type $3$ order in $\oo$; i.e., it satisfies the conditions of Proposition\ref{type3conditions}.
\end{proposition}

\begin{proof}
Just as in the type $2$ case, we may check that $\Lambda$ is an order by verifying the conditions of Proposition \ref{orderconditions}. Next, in this case $\Lambda^\ast = \Lambda\left(
  \begin{Array}{ccc|c}
    0 & -1 & 0 & 0\\
    -1 & 0 & -1 & 0\\
  \end{Array}
\right)$,
so clearly
\[ \Lambda\left(
  \begin{Array}{ccc|c}
    1 & 0 & 1 & 0\\
    0 & 1 & 0 & 0\\
  \end{Array}
\right) \subsetneq \Lambda\left(
  \begin{Array}{ccc|c}
    0 & -1 & 0 & 0\\
    -1 & 0 & -1 & 0\\
  \end{Array}
\right) \subsetneq \Lambda\left(
  \begin{Array}{ccc|c}
    0 & -1 & 0 & -1\\
    -1 & 0 & -1 & -1\\
  \end{Array}
\right). \]
so condition (a) is satisfied.  With an aim to producing $M$, we calculate $\Lambda^{\ast 2}$:
\begin{align*}
\Lambda^{\ast 2} & = \Lambda\left(
  \begin{Array}{ccc|c}
    0 & -1 & 0 & 0\\
    -1 & 0 & -1 & 0\\
  \end{Array}
\right)^2 \\
    & = \left(
  \begin{Array}{cc}
    \zp & \left(
                   \begin{Array}{c}
                     \zp \\
                     p^{-1}\zp \\
                     \zp \\
                   \end{Array}
                 \right) \\
    \left(
                   \begin{Array}{c}
                     p^{-1}\zp \\
                     \zp \\
                     p^{-1}\zp \\
                   \end{Array}
                 \right) & \zp
  \end{Array}
\right)\left(
  \begin{Array}{cc}
    \zp & \left(
                   \begin{Array}{c}
                     \zp \\
                     p^{-1}\zp \\
                     \zp \\
                   \end{Array}
                 \right) \\
    \left(
                   \begin{Array}{c}
                     p^{-1}\zp \\
                     \zp \\
                     p^{-1}\zp \\
                   \end{Array}
                 \right) & \zp
  \end{Array}
\right) \\
    & = \left(
  \begin{Array}{cc}
    p^{-1}\zp & \left(
                   \begin{Array}{c}
                     p^{-1}\zp \\
                     p^{-2}\zp \\
                     p^{-1}\zp \\
                   \end{Array}
                 \right) \\
    \left(
                   \begin{Array}{c}
                     p^{-1}\zp \\
                     \zp \\
                     p^{-1}\zp \\
                   \end{Array}
                 \right) & p^{-1}\zp
  \end{Array}
\right) \\
    & = \Lambda\left(
  \begin{Array}{ccc|c}
    -1 & -2 & -1 & -1\\
    -1 & 0 & -1 & -1\\
  \end{Array}
\right).
\end{align*}
Therefore we have that
\begin{align*}
M & = p\Lambda^{\ast 2} +\Lambda \\
    & = p \,\Lambda\left(
  \begin{Array}{ccc|c}
    -1 & -2 & -1 & -1\\
    -1 & 0 & -1 & -1\\
  \end{Array}
\right) + \Lambda\left(
  \begin{Array}{ccc|c}
    1 & 0 & 1 & 0\\
    0 & 1 & 0 & 0\\
  \end{Array}
\right) \\
    & = \Lambda\left(
  \begin{Array}{ccc|c}
    0 & -1 & 0 & 0\\
    0 & 1 & 0 & 0\\
  \end{Array}
\right) + \Lambda\left(
  \begin{Array}{ccc|c}
    1 & 0 & 1 & 0\\
    0 & 1 & 0 & 0\\
  \end{Array}
\right) \\
    & = \Lambda\left(
  \begin{Array}{ccc|c}
    0 & -1 & 0 & 0\\
    0 & 1 & 0 & 0\\
  \end{Array}
\right).
\end{align*}
Thus $M$ is self-dual and condition (b) is also satisfied, so $\Lambda$ is a type $3$ order in $\oo$.
\end{proof}

\begin{comment}
\[  \dots \subsetneq \Lambda \subsetneq M \subsetneq \Lambda^\ast \subsetneq p^{-1}\Lambda \subsetneq \dots \]
\end{comment}

By applying this method of detecting a type $3$ order based on the surrounding, already identified orders, we can obtain all the type $3$ orders.  We list those nearest to the origin in blue, in Figure \ref{Type3Vertices}.

\begin{figure}[ht!]
\begin{center}
\frame{%
\begin{tikzpicture}[thin,scale=2]%
    \clip (-3.2,-2.2) rectangle (3.2,2.2);
    \draw[help lines] \foreach \z in {0,1,...,5} {
            \foreach \y in {0,1,...,5} {
              \foreach \x in {0,1,...,11} {
        (0:0) -- ++(30*\x:10)
        (60*\y + 30:2*\z) -- ++(30*\x:10)
    }}};
    \draw{ (0:0) node[circle, draw, fill=yellow!50,
                        inner sep=0pt, minimum width=4pt, label=above: $\underline\Lambda$]{}
        (30:2) node[circle, draw, fill=yellow!50,
                        inner sep=0pt, minimum width=4pt]{}
        (90:2) node[circle, draw, fill=yellow!50,
                        inner sep=0pt, minimum width=4pt]{}
        (150:2) node[circle, draw, fill=yellow!50,
                        inner sep=0pt, minimum width=4pt]{}
        (210:2) node[circle, draw, fill=yellow!50,
                        inner sep=0pt, minimum width=4pt]{}
        (270:2) node[circle, draw, fill=yellow!50,
                        inner sep=0pt, minimum width=4pt]{}
        (330:2) node[circle, draw, fill=yellow!50,
                        inner sep=0pt, minimum width=4pt]{}
    };
    \draw{ (30:1) node[circle, draw, fill=red!50,
                        inner sep=0pt, minimum width=4pt]{}
        (90:1) node[circle, draw, fill=red!50,
                        inner sep=0pt, minimum width=4pt]{}
        (150:1) node[circle, draw, fill=red!50,
                        inner sep=0pt, minimum width=4pt]{}
        (210:1) node[circle, draw, fill=red!50,
                        inner sep=0pt, minimum width=4pt]{}
        (270:1) node[circle, draw, fill=red!50,
                        inner sep=0pt, minimum width=4pt]{}
        (330:1) node[circle, draw, fill=red!50,
                        inner sep=0pt, minimum width=4pt]{}
    };
    \draw{ (0:1.1547) node[circle, draw, fill=blue!50,
                        inner sep=0pt, minimum width=4pt, label=right: $\Lambda\left(
  \begin{Array}{ccc|c}
    0 & 0 & 1 & 0\\
    1 & 1 & 0 & 0\\
  \end{Array}
\right)$]{}
        (60:1.1547) node[circle, draw, fill=blue!50,
                        inner sep=0pt, minimum width=4pt, label=above right: $\Lambda\left(
  \begin{Array}{ccc|c}
    1 & 0 & 1 & 0\\
    0 & 1 & 0 & 0\\
  \end{Array}
\right)$]{}
        (120:1.1547) node[circle, draw, fill=blue!50,
                        inner sep=0pt, minimum width=4pt, label=above left: $\Lambda\left(
  \begin{Array}{ccc|c}
    1 & 0 & 0 & 0\\
    0 & 1 & 1 & 0\\
  \end{Array}
\right)$]{}
        (180:1.1547) node[circle, draw, fill=blue!50,
                        inner sep=0pt, minimum width=4pt, label=left: $\Lambda\left(
  \begin{Array}{ccc|c}
    1 & 1 & 0 & 0\\
    0 & 0 & 1 & 0\\
  \end{Array}
\right)$]{}
        (240:1.1547) node[circle, draw, fill=blue!50,
                        inner sep=0pt, minimum width=4pt, label=below left: $\Lambda\left(
  \begin{Array}{ccc|c}
    0 & 1 & 0 & 0\\
    1 & 0 & 1 & 0\\
  \end{Array}
\right)$]{}
        (300:1.1547) node[circle, draw, fill=blue!50,
                        inner sep=0pt, minimum width=4pt, label=below right: $\Lambda\left(
  \begin{Array}{ccc|c}
    0 & 1 & 1 & 0\\
    1 & 0 & 0 & 0\\
  \end{Array}
\right)$]{}
    };
\end{tikzpicture}
}
\caption{Octonion orders associated to the vertices of type $3$.}\label{Type3Vertices}
\end{center}
\end{figure}
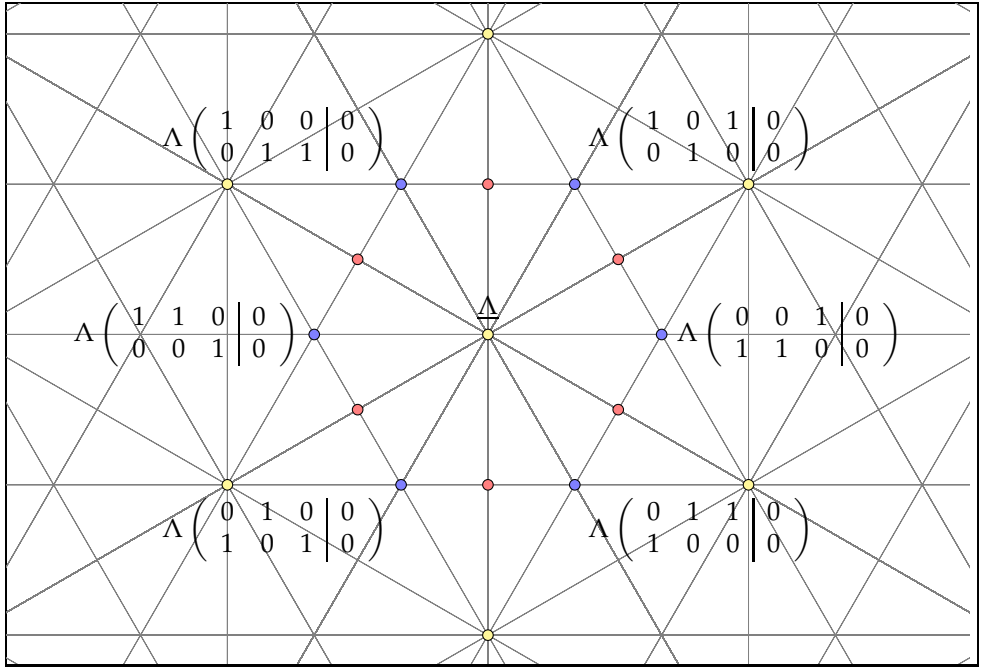

Once more, we define algebra valuations corresponding to the type $3$ vertices.  This time we note that each type $3$ vertex lies at the barycenter of a triangle created by three type $1$ vertices, and therefore we define the relevant valuation by averaging over the three type $1$ valuations.

To use the work of our examples in Section \ref{type1points}, we take $\Lambda = \Lambda\left(
  \begin{Array}{ccc|c}
    1 & 0 & 1 & 0\\
    0 & 1 & 0 & 0\\
  \end{Array}
\right)$, and for our chosen basis $B$ of $\oo$, take the intermediate function $v_3(b_i)$ as follows:
\begin{align*}
v_3^{\Lambda}(b_1) & = \frac{1}{3}(-1+0+0) = -\frac{1}{3}, & v_3^{\Lambda}(b_{-1}) = \frac{1}{3}(1+0+0) = \frac{1}{3},\\
v_3^{\Lambda}(b_2) & = \frac{1}{3}(1+1+0) = \frac{2}{3}, & v_3^{\Lambda}(b_{-2}) = \frac{1}{3}(-1-1+0) = -\frac{2}{3},\\
v_3^{\Lambda}(b_3) & = \frac{1}{3}(0-1+0) = -\frac{1}{3}, & v_3^{\Lambda}(b_{-3}) = \frac{1}{3}(0+1+0) = \frac{1}{3},\\
v_3^{\Lambda}(b_4) & = \frac{1}{3}(0+0+0) = 0, & v_3^{\Lambda}(b_{-4}) = \frac{1}{3}(0+0+0) = 0.
\end{align*}

We again use this to find the valuation of several octonion elements.  For a lattice element $x \in \Lambda = \Lambda\left(
  \begin{Array}{ccc|c}
    1 & 0 & 1 & 0\\
    0 & 1 & 0 & 0\\
  \end{Array}
\right)$ we have:
\begin{align*}
\val_3(x) & \geq \min \left\{ 1-\frac{1}{3}, \frac{2}{3}, 1-\frac{1}{3},\frac{1}{3}, 1- \frac{2}{3},\frac{1}{3},0, 0 \right\} = 0.
\end{align*}

For a lattice element $x \in M = \Lambda\left(
  \begin{Array}{ccc|c}
    0 & -1 & 0 & 0\\
    0 & 1 & 0 & 0\\
  \end{Array}
\right)$ we have:
\begin{align*}
\val_3(x) & \geq \min \left\{ -\frac{1}{3}, \frac{2}{3}-1, -\frac{1}{3},\frac{1}{3}, \frac{2}{3}+1,\frac{1}{3},0, 0 \right\}  = -\frac{1}{3}.
\end{align*}

For a lattice element $x \in \Lambda^\ast = \Lambda\left(
  \begin{Array}{ccc|c}
    0 & -1 & 0 & 0\\
    -1 & 0 & -1 & 0\\
  \end{Array}
\right)$ we have:
\begin{align*}
\val_3(x) & \geq \min \left\{ -\frac{1}{3}, \frac{2}{3}-1, -\frac{1}{3},\frac{1}{3}-1, \frac{2}{3},\frac{1}{3}-1,0, 0 \right\} = -\frac{2}{3}.
\end{align*}

For a lattice element $x \in p^{-1}\Lambda = \Lambda\left(
  \begin{Array}{ccc|c}
    0 & -1 & 0 & -1\\
    -1 & 0 & -1 & -1\\
  \end{Array}
\right)$ we have:
\begin{align*}
\val_3(x) & \geq \min \left\{ -\frac{1}{3}, \frac{2}{3}-1, -\frac{1}{3},\frac{1}{3}-1, \frac{2}{3},\frac{1}{3}-1,-1,-1 \right\}  = -1.
\end{align*}

This valuation leads us to the lattice sequence
\[ \cdots \subsetneq p^{-1}\Lambda \subsetneq \Lambda^\ast \subsetneq M \subsetneq \Lambda \subsetneq p(\Lambda^\ast) \subsetneq pM \subsetneq p\Lambda \subsetneq \cdots \]
Once again, algebra valuations and lattice sequences for all type $3$ vertices can be found in this way.

\section{Other Points in $\mathcal{A}$}
\label{otherpoints}

We conclude this chapter by identifying the algebra valuations and lattice sequences corresponding to a few points in the apartment which do not lie on vertices.  In Figure \ref{OtherOrders}, the points that we will address are marked in orange and green.

\begin{figure}[ht!]
\begin{center}
\frame{%
\begin{tikzpicture}[thin,scale=2]%
    \clip (-3.2,-2.2) rectangle (3.2,2.2);
    \draw[help lines] \foreach \z in {0,1,...,5} {
            \foreach \y in {0,1,...,5} {
              \foreach \x in {0,1,...,11} {
        (0:0) -- ++(30*\x:10)
        (60*\y + 30:2*\z) -- ++(30*\x:10)
    }}};
    \draw{ (0:0) node[circle, draw, fill=yellow!50,
                        inner sep=0pt, minimum width=4pt, label=above: $\underline\Lambda$]{}
        (30:2) node[circle, draw, fill=yellow!50,
                        inner sep=0pt, minimum width=4pt]{}
        (90:2) node[circle, draw, fill=yellow!50,
                        inner sep=0pt, minimum width=4pt]{}
        (150:2) node[circle, draw, fill=yellow!50,
                        inner sep=0pt, minimum width=4pt]{}
        (210:2) node[circle, draw, fill=yellow!50,
                        inner sep=0pt, minimum width=4pt]{}
        (270:2) node[circle, draw, fill=yellow!50,
                        inner sep=0pt, minimum width=4pt]{}
        (330:2) node[circle, draw, fill=yellow!50,
                        inner sep=0pt, minimum width=4pt]{}
    };
    \draw{ (30:1) node[circle, draw, fill=red!50,
                        inner sep=0pt, minimum width=4pt]{}
        (90:1) node[circle, draw, fill=red!50,
                        inner sep=0pt, minimum width=4pt]{}
        (150:1) node[circle, draw, fill=red!50,
                        inner sep=0pt, minimum width=4pt]{}
        (210:1) node[circle, draw, fill=red!50,
                        inner sep=0pt, minimum width=4pt]{}
        (270:1) node[circle, draw, fill=red!50,
                        inner sep=0pt, minimum width=4pt]{}
        (330:1) node[circle, draw, fill=red!50,
                        inner sep=0pt, minimum width=4pt]{}
    };
    \draw{ (0:1.1547) node[circle, draw, fill=blue!50,
                        inner sep=0pt, minimum width=4pt]{}
        (60:1.1547) node[circle, draw, fill=blue!50,
                        inner sep=0pt, minimum width=4pt]{}
        (120:1.1547) node[circle, draw, fill=blue!50,
                        inner sep=0pt, minimum width=4pt]{}
        (180:1.1547) node[circle, draw, fill=blue!50,
                        inner sep=0pt, minimum width=4pt]{}
        (240:1.1547) node[circle, draw, fill=blue!50,
                        inner sep=0pt, minimum width=4pt]{}
        (300:1.1547) node[circle, draw, fill=blue!50,
                        inner sep=0pt, minimum width=4pt]{}
    };
    \draw{
        (90:0.667) node[circle, draw, fill=orange!50,
                        inner sep=0pt, minimum width=4pt]{}
    };
    \draw{
        (60:0.57735) node[circle, draw, fill=green!50,
                        inner sep=0pt, minimum width=4pt]{}
    };
\end{tikzpicture}
}
\caption{Other octonion orders in $\mathcal{A}$.}\label{OtherOrders}
\end{center}
\end{figure}
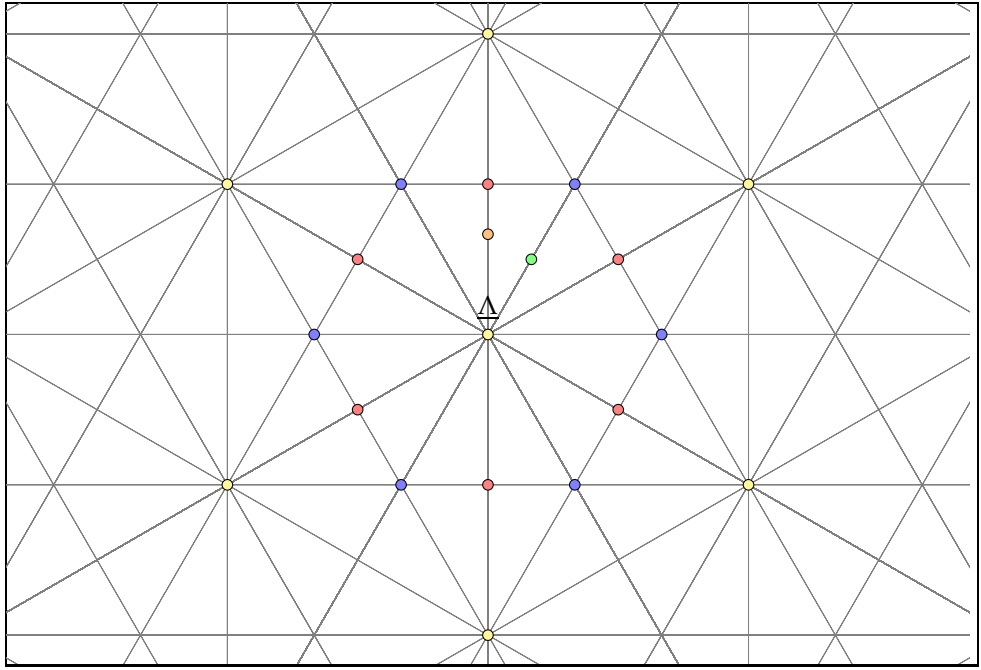

Starting with the point in orange, we note that it is at the barycenter of a triangle formed by $\underline\Lambda$ and two type $3$ vertices.  Therefore we can use the `averaging' strategy we have already employed to construct a new algebra valuation corresponding to that point.  However, it will be simpler for us to note that the point also lies two-thirds of the distance from $\underline\Lambda$ to a type $2$ vertex that we have already identified.  Therefore, for each $b_i \in B$ we can take
\[ v_4^{\Lambda}(b_i) = \frac{2}{3} v_2^{\Lambda}(b_i), \]
where $v_2$ was defined in Section \ref{type2points}. More explicitly:
\[ v_2^{\Lambda}(b_i) = \left\{
  \begin{Array}{rl}
    0, & \quad\text{if } i \in \{\pm4,\pm3\} \\
    1/3, & \quad\text{if } i \in \{2,-1\} \\
    -1/3, & \quad\text{if } i \in \{1,-2\}.
  \end{Array}
\right. \]
Again, our new valuation will be
\[ \val_4(x) := \min \left\{ v_4^{\Lambda}(a_i b_i) \right\}, \quad \text{for } x = \sum a_i b_i \in \oo. \]

Now, we work backwards to find the lattices which will make up our associated lattice sequence.  Let
\[ \Lambda^{\val_4}_0 = \Lambda\left(
  \begin{Array}{ccc|c}
    \lambda_1 & \lambda_2 & \lambda_3 & \lambda_4\\
    \lambda_{-1} & \lambda_{-2} & \lambda_{-3} & \lambda_{-4}\\
  \end{Array}
\right) \]
be the octonion order identified to our point in $\mathcal{A}$ and corresponding to the $r=0$ index in our desired lattice sequence.  Then we should have that $\val_4(x) \geq 0$ for all $x \in \Lambda^{\val_4}_0$.  Recall that each $\lambda_i$ is an integer.  Therefore for
\[ \min \left\{ \lambda_1-\frac{1}{3} , \lambda_2 +\frac{1}{3}, \lambda_3,\lambda_{-1} + \frac{1}{3}, \lambda_{-2}-\frac{1}{3}, \lambda_{-3}, \lambda_{\pm4} \right\} \geq 0, \]
the lowest valuations we may have are $\lambda_1 = \lambda_{-2} = 1$, with the rest equal to zero. Our order is now identified as
\[ \Lambda^{\val_4}_0 = \Lambda\left(
  \begin{Array}{ccc}
    1 & 0 & 0 \\
    0 & 1 & 0 \\
  \end{Array}
\right). \]
We note that the next jump in our lattice sequence occurs at $\Lambda^{\val_4}_{1/3}$, which we may identify in a similar way, by considering:
\[ \min \left\{ \lambda_1-\frac{1}{3} , \lambda_2 +\frac{1}{3}, \lambda_3,\lambda_{-1} + \frac{1}{3}, \lambda_{-2}-\frac{1}{3}, \lambda_{-3}, \lambda_{\pm4} \right\} \geq \frac{1}{3}. \]
In this case we find that
\[ \Lambda^{\val_4}_{1/3} = \Lambda\left(
  \begin{Array}{ccc|c}
    1 & 0 & 1 & 1 \\
    0 & 1 & 1 & 1 \\
  \end{Array}
\right). \]

The following jump in our lattice sequence occurs at $\Lambda^{\val_4}_{2/3}$, which we again identify:
\[ \min \left\{ \lambda_1-\frac{1}{3} , \lambda_2 +\frac{1}{3}, \lambda_3,\lambda_{-1} + \frac{1}{3}, \lambda_{-2}-\frac{1}{3}, \lambda_{-3}, \lambda_{\pm4} \right\} \geq \frac{2}{3}. \]
In this case we find that
\[ \Lambda^{\val_4}_{1/3} = \Lambda\left(
  \begin{Array}{ccc|c}
    1 & 1 & 1 & 1 \\
    1 & 1 & 1 & 1 \\
  \end{Array}
\right). \]

Finally, we identify $\Lambda^{\val_4}_0$:
\[ \min \left\{ \lambda_1-\frac{1}{3} , \lambda_2 +\frac{1}{3}, \lambda_3,\lambda_{-1} + \frac{1}{3}, \lambda_{-2}-\frac{1}{3}, \lambda_{-3}, \lambda_{\pm4} \right\} \geq 1. \]
In this case we find that
\[ \Lambda^{\val_4}_1 = \Lambda\left(
  \begin{Array}{ccc|c}
    2 & 1 & 1 & 1 \\
    1 & 2 & 1 & 1 \\
  \end{Array}
\right) = p\Lambda^{\val_4}_0. \]

Therefore we obtain the lattice sequence corresponding to our orange point in $\mathcal{A}$:
\[ \cdots \subsetneq \Lambda\left(
  \begin{Array}{ccc|c}
    2 & 1 & 1 & 1 \\
    1 & 2 & 1 & 1 \\
  \end{Array}
\right) \subsetneq \Lambda\left(
  \begin{Array}{ccc|c}
    1 & 1 & 1 & 1 \\
    1 & 1 & 1 & 1 \\
  \end{Array}
\right) \subsetneq \Lambda\left(
  \begin{Array}{ccc|c}
    1 & 0 & 1 & 1 \\
    0 & 1 & 1 & 1 \\
  \end{Array}
\right) \subsetneq \Lambda\left(
  \begin{Array}{ccc}
    1 & 0 & 0 \\
    0 & 1 & 0 \\
  \end{Array}
\right) \subsetneq \cdots \]

Finally, we tackle the point labeled in green in Figure \ref{OtherOrders}, which we notice lie halfway between $\underline\Lambda$ and the type $3$ vertex that we identified in Section \ref{type3conditions}.  Therefore we set
\[ v_5^{\Lambda}(b_i) = \frac{1}{2} v_3^{\Lambda}(b_i), \]
where $v_3$ was defined in Section \ref{type3conditions}. Explicitly:
\begin{align*}
v_5^{\Lambda}(b_1) & = -\frac{1}{6}, & v_5^{\Lambda}(b_{-1}) = \frac{1}{6},\\
v_5^{\Lambda}(b_2) & = \frac{1}{3}, & v_5^{\Lambda}(b_{-2}) = -\frac{1}{3},\\
v_5^{\Lambda}(b_3) & = -\frac{1}{6}, & v_5^{\Lambda}(b_{-3}) = \frac{1}{6},\\
v_5^{\Lambda}(b_4) & = 0, & v_5^{\Lambda}(b_{-4})  = 0.
\end{align*}
Again, our new valuation will be
\[ \val_5(x) := \min \left\{ v_5^{\Lambda}(a_i b_i) \right\}, \quad \text{for } x = \sum a_i b_i \in \oo. \]

We again work backwards to find the lattices which will make up our associated lattice sequence.  Again let
\[ \Lambda^{\val_5}_0 = \Lambda\left(
  \begin{Array}{ccc|c}
    \lambda_1 & \lambda_2 & \lambda_3 & \lambda_4\\
    \lambda_{-1} & \lambda_{-2} & \lambda_{-3} & \lambda_{-4}\\
  \end{Array}
\right). \]
be the octonion order identified to our green point in $\mathcal{A}$ and corresponding to the $r=0$ index in our desired lattice sequence.  Then since
\[ \min \left\{ \lambda_1-\frac{1}{6} , \lambda_2 +\frac{1}{3}, \lambda_3-\frac{1}{6},\lambda_{-1} + \frac{1}{6}, \lambda_{-2}-\frac{1}{3}, \lambda_{-3}+\frac{1}{6}, \lambda_{\pm4} \right\} \geq 0, \]
the lowest valuations we may have are $\lambda_1 = \lambda_{-2} = \lambda_3 = 1$, and the rest are equal to zero. Our order is now identified as
\[ \Lambda^{\val_5}_0 = \Lambda\left(
  \begin{Array}{ccc}
    1 & 0 & 1 \\
    0 & 1 & 0 \\
  \end{Array}
\right). \]

The next jump in our lattice sequence occurs at $\Lambda^{\val_4}_{1/6}$, which we identify by considering:
\[ \min \left\{ \lambda_1-\frac{1}{6} , \lambda_2 +\frac{1}{3}, \lambda_3-\frac{1}{6},\lambda_{-1} + \frac{1}{6}, \lambda_{-2}-\frac{1}{3}, \lambda_{-3}+\frac{1}{6}, \lambda_{\pm4} \right\} \geq \frac{1}{6}, \]
In this case we find that
\[ \Lambda^{\val_5}_{1/6} = \Lambda\left(
  \begin{Array}{ccc|c}
    1 & 0 & 1 & 1 \\
    0 & 1 & 0 & 1 \\
  \end{Array}
\right). \]

We spare the reader the remaining arguments to show that the following lattices make up the rest of our lattice sequence:
\begin{align*}
\Lambda^{\val_5}_{1/3} & = \Lambda\left(
  \begin{Array}{ccc|c}
    1 & 0 & 1 & 1 \\
    1 & 1 & 1 & 1 \\
  \end{Array}
\right) \\
\Lambda^{\val_5}_{1/2} = \Lambda^{\val_5}_{2/3}& = \Lambda\left(
  \begin{Array}{ccc|c}
    1 & 1 & 1 & 1 \\
    1 & 1 & 1 & 1 \\
  \end{Array}
\right) \\
\Lambda^{\val_5}_{5/6} & = \Lambda\left(
  \begin{Array}{ccc|c}
    1 & 1 & 1 & 1 \\
    1 & 2 & 1 & 1 \\
  \end{Array}
\right) \\
\Lambda^{\val_5}_{1} & = \Lambda\left(
  \begin{Array}{ccc|c}
    2 & 1 & 2 & 1 \\
    1 & 2 & 1 & 1 \\
  \end{Array}
\right) = p \cdot \Lambda^{\val_5}_0.
\end{align*}

These lattices now form the lattice sequence associated to the green point in our figure.

%%%%%%%%%%%%%%%%%%%%%%%%%%%%%%%%%%%%%%%%%%%%%%%%%%%%%%%%%%%%%%%%%%%%%%%%%%%%%%%%%%%%%%%%%%%%%%%%%%%%%%%%%%

\appendix

\chapter{Tables of Chevalley Constants}

All calculations performed in this section were done with the aid of the SAGE software system.  The precise code created for this task is included in Appendix B for reference.

The tables included in this appendix are meant to display the particular constants $N_{ij}$ which arise from our choices of Chevalley generators for $G_2$.  To remind ourselves, our generators should satisfy the following relations \cite[pg $66$]{Steinberg1967}:

\begin{enumerate}
\item The $e_{\alpha}$ are each homomorphisms from the additive group of $k$ into $G$, that is: \[ e_{\alpha}(s+t) = e_{\alpha}(s)e_{\alpha}(t) \quad \text{ for all } s,t \in k.\]

\item If $\alpha, \beta \in \Phi$ with $\alpha + \beta \neq 0$, then
\[ [e_{\beta}(t) , e_{\alpha}(s)] = \prod e_{i \alpha + j \beta}(N_{ij} s^i t^j), \]
where the product is taken over all (strictly) positive integers $i,j \in \zz$ such that $i \alpha + j \beta \in \Phi$, and the $N_{ij}$ are each integers depending on $\alpha,\beta$, but not on $s,t$.

\item Each $h_{\alpha}$ is multiplicative in $k^{\times}$; i.e., $h_{\alpha}(s)h_{\alpha}(t) = h_{\alpha}(st)$ for all $s,t \in k^{\times}$.
\end{enumerate}

We have already verified relations (a) and (c) from Section \ref{sec:associationsverifications}, so we now concentrate on relation (b).  Figure \ref{autosandrootsG2append} shows our chosen Chevalley generators, and their association to the roots in the $G_2$ root diagram:

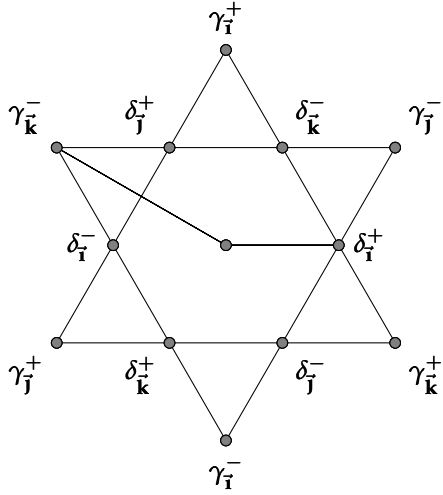
\begin{figure}[!h]
\begin{center}
\begin{tikzpicture}[thin,scale=1.5]%
    \draw \foreach \x in {0,120,240} {
        (\x+90:1.732) -- (\x+210:1.732)
        (\x-90:1.732) -- (\x+30:1.732)
        (0:0) -- (0:1)
        (0:0) -- (150:1.732)
        (0:0) node[circle, draw, fill=black!50,
                        inner sep=0pt, minimum width=4pt]{}
        (0:1) node[circle, draw, fill=black!50,
                        inner sep=0pt, minimum width=4pt, label=right:$\delta_{\ii}^{+}$]{}
        (60:1) node[circle, draw, fill=black!50,
                        inner sep=0pt, minimum width=4pt, label=above right:$\delta_{\kk}^{-}$]{}
        (120:1) node[circle, draw, fill=black!50,
                        inner sep=0pt, minimum width=4pt, label=above left:$\delta_{\jj}^{+}$]{}
        (180:1) node[circle, draw, fill=black!50,
                        inner sep=0pt, minimum width=4pt, label=left:$\delta_{\ii}^{-}$]{}
        (240:1) node[circle, draw, fill=black!50,
                        inner sep=0pt, minimum width=4pt, label=below left:$\delta_{\kk}^{+}$]{}
        (300:1) node[circle, draw, fill=black!50,
                        inner sep=0pt, minimum width=4pt, label=below right:$\delta_{\jj}^{-}$]{}
        (30:1.732) node[circle, draw, fill=black!50,
                        inner sep=0pt, minimum width=4pt, label=above right:$\gamma_{\jj}^{-}$]{}
        (90:1.732) node[circle, draw, fill=black!50,
                        inner sep=0pt, minimum width=4pt, label=above:$\gamma_{\ii}^{+}$]{}
        (150:1.732) node[circle, draw, fill=black!50,
                        inner sep=0pt, minimum width=4pt, label=above left:$\gamma_{\kk}^{-}$]{}
        (210:1.732) node[circle, draw, fill=black!50,
                        inner sep=0pt, minimum width=4pt, label=below left:$\gamma_{\jj}^{+}$]{}
        (270:1.732) node[circle, draw, fill=black!50,
                        inner sep=0pt, minimum width=4pt, label=below:$\gamma_{\ii}^{-}$]{}
        (330:1.732) node[circle, draw, fill=black!50,
                        inner sep=0pt, minimum width=4pt, label=below right:$\gamma_{\kk}^{+}$]{}
};
\end{tikzpicture}
\caption{The association of the automorphisms of $\oo$ to the roots of $G_2$.}\label{autosandrootsG2append}
\end{center}
\end{figure}

In the following series of tables, we label each row with the automorphism $\alpha$ being applied first, and label each column with  automorphisms $\beta$ having non-trivial commutator bracket with the first.  A trivial commutator bracket will mean that either the two automorphisms in question commute, or that the Chevalley relation is vacuous; i.e., that $\alpha + \beta = 0$.
In the product on the right side of relation (b), we have applied first the $\beta$ string in ascending order, and then the $2\beta$ string, and then the $3\beta$ string.  The indices in each $N_{ij}$ correspond to the linear combination $i\alpha+j\beta$; i.e., corresponding to the arguments $N_{ij} s^{i}t^{j}$.

\newpage

\begin{table}[!ht]
\begin{center}
\begin{tabular}{|l|l|l|l|l|l|l|}
  \hline
                   & $\gamma_{\jj}^{+}(t)$ & $\delta_{\kk}^{+}(t)$ & $\delta_{\jj}^{-}(t)$ & $\gamma_{\kk}^{+}(t)$  \\
  \hline
  $\gamma_{\ii}^{+}(s)$&$N_{11}= 1$&$N_{11}= 1$&$N_{11}=-1$&$N_{11}=-1$\\
                       &           &$N_{12}= 1$&$N_{12}=-1$&           \\
                       &           &$N_{13}= 1$&$N_{13}= 1$&           \\
  \hline
\end{tabular}
\caption{Chevalley constants for commutators of type $\left[ - ,\gamma_{\ii}^{+}(s) \right]$.}\label{gammaiplus}
\end{center}
\end{table}

\begin{table}[!ht]
\begin{center}
\begin{tabular}{|l|l|l|l|l|l|l|}
  \hline
                   & $\gamma_{\jj}^{-}(t)$ & $\delta_{\kk}^{-}(t)$ & $\delta_{\jj}^{+}(t)$ & $\gamma_{\kk}^{-}(t)$  \\
  \hline
  $\gamma_{\ii}^{-}(s)$&$N_{11}=-1$&$N_{11}=-1$&$N_{11}= 1$&$N_{11}= 1$\\
                       &           &$N_{12}= 1$&$N_{12}=-1$&           \\
                       &           &$N_{13}=-1$&$N_{13}=-1$&           \\
  \hline
\end{tabular}
\caption{Chevalley constants for commutators of type $\left[ - ,\gamma_{\ii}^{-}(s) \right]$.}\label{gammaiminus}
\end{center}
\end{table}

\begin{table}[!ht]
\begin{center}
\begin{tabular}{|l|l|l|l|l|l|l|}
  \hline
                   & $\gamma_{\ii}^{+}(t)$ & $\delta_{\kk}^{-}(t)$ & $\delta_{\ii}^{+}(t)$ & $\gamma_{\kk}^{+}(t)$  \\
  \hline
  $\gamma_{\jj}^{+}(s)$&$N_{11}=-1$&$N_{11}=-1$&$N_{11}= 1$&$N_{11}= 1$\\
                       &           &$N_{12}=-1$&$N_{12}= 1$&           \\
                       &           &$N_{13}= 1$&$N_{13}= 1$&           \\
  \hline
\end{tabular}
\caption{Chevalley constants for commutators of type $\left[ - ,\gamma_{\jj}^{+}(s) \right]$.}\label{gammajplus}
\end{center}
\end{table}

\begin{table}[!ht]
\begin{center}
\begin{tabular}{|l|l|l|l|l|l|l|}
  \hline
                   & $\gamma_{\ii}^{-}(t)$ & $\delta_{\kk}^{+}(t)$ & $\delta_{\ii}^{-}(t)$ & $\gamma_{\kk}^{-}(t)$  \\
  \hline
  $\gamma_{\jj}^{-}(s)$&$N_{11}= 1$&$N_{11}= 1$&$N_{11}=-1$&$N_{11}=-1$\\
                       &           &$N_{12}=-1$&$N_{12}= 1$&           \\
                       &           &$N_{13}=-1$&$N_{13}=-1$&           \\
  \hline
\end{tabular}
\caption{Chevalley constants for commutators of type $\left[ - ,\gamma_{\jj}^{-}(s) \right]$.}\label{gammajminus}
\end{center}
\end{table}

\begin{table}[!ht]
\begin{center}
\begin{tabular}{|l|l|l|l|l|l|l|}
  \hline
                   & $\gamma_{\ii}^{+}(t)$ & $\delta_{\jj}^{-}(t)$ & $\delta_{\ii}^{-}(t)$ & $\gamma_{\jj}^{+}(t)$  \\
  \hline
  $\gamma_{\kk}^{+}(s)$&$N_{11}= 1$&$N_{11}= 1$&$N_{11}=-1$&$N_{11}=-1$\\
                       &           &$N_{12}= 1$&$N_{12}=-1$&           \\
                       &           &$N_{13}= 1$&$N_{13}= 1$&           \\
  \hline
\end{tabular}
\caption{Chevalley constants for commutators of type $\left[ - ,\gamma_{\kk}^{+}(s) \right]$.}\label{gammakplus}
\end{center}
\end{table}

\begin{table}[!ht]
\begin{center}
\begin{tabular}{|l|l|l|l|l|l|l|}
  \hline
                   & $\gamma_{\ii}^{-}(t)$ & $\delta_{\jj}^{-}(t)$ & $\delta_{\ii}^{+}(t)$ & $\gamma_{\jj}^{-}(t)$  \\
  \hline
  $\gamma_{\kk}^{-}(s)$&$N_{11}=-1$&$N_{11}=-1$&$N_{11}= 1$&$N_{11}= 1$\\
                       &           &$N_{12}= 1$&$N_{12}=-1$&           \\
                       &           &$N_{13}=-1$&$N_{13}=-1$&           \\
  \hline
\end{tabular}
\caption{Chevalley constants for commutators of type $\left[ - ,\gamma_{\kk}^{-}(s) \right]$.}\label{gammakminus}
\end{center}
\end{table}

\newpage

\begin{table}[!h]
\begin{center}
\begin{tabular}{|l|l|l|l|l|l|l|}
  \hline
                   & $\gamma_{\kk}^{-}(t)$ & $\delta_{\jj}^{+}(t)$ & $\delta_{\kk}^{-}(t)$ & $\gamma_{\jj}^{+}(t)$ & $\delta_{\kk}^{+}(t)$ & $\delta_{\jj}^{-}(t)$ \\
  \hline
  $\delta_{\ii}^{+}(s)$&$N_{11}=-1$&$N_{11}=-2$&$N_{11}= 3$&$N_{11}=-1$&$N_{11}= 2$&$N_{11}= 3$\\
                       &$N_{21}= 1$&$N_{21}=-3$&           &$N_{21}=-1$&$N_{21}= 3$&           \\
                       &$N_{31}= 1$&$N_{12}=-3$&           &$N_{31}=-1$&$N_{12}= 3$&           \\
                       &$N_{32}= 2$&           &           &$N_{32}=-2$&           &           \\
  \hline
\end{tabular}
\caption{Chevalley constants for commutators of type $\left[ - ,\delta_{\ii}^{+}(s) \right]$.}\label{gammaconstants}
\end{center}
\end{table}

\begin{table}[!h]
\begin{center}
\begin{tabular}{|l|l|l|l|l|l|l|}
  \hline
                   & $\gamma_{\jj}^{-}(t)$ & $\delta_{\kk}^{-}(t)$ & $\delta_{\jj}^{+}(t)$ & $\gamma_{\kk}^{+}(t)$ & $\delta_{\jj}^{-}(t)$ & $\delta_{\kk}^{+}(t)$ \\
  \hline
  $\delta_{\ii}^{-}(s)$&$N_{11}= 1$&$N_{11}=-2$&$N_{11}=-3$&$N_{11}= 1$&$N_{11}= 2$&$N_{11}=-3$\\
                       &$N_{21}=-1$&$N_{21}= 3$&           &$N_{21}= 1$&$N_{21}=-3$&           \\
                       &$N_{31}= 1$&$N_{12}= 3$&           &$N_{31}=-1$&$N_{12}=-3$&           \\
                       &$N_{32}=-2$&           &           &$N_{32}= 2$&           &           \\
  \hline
\end{tabular}
\caption{Chevalley constants for commutators of type $\left[ - ,\delta_{\ii}^{-}(s) \right]$.}\label{gammaconstants}
\end{center}
\end{table}

\begin{table}[!h]
\begin{center}
\begin{tabular}{|l|l|l|l|l|l|l|}
  \hline
                   & $\gamma_{\kk}^{+}(t)$ & $\delta_{\ii}^{+}(t)$ & $\delta_{\kk}^{-}(t)$ & $\gamma_{\ii}^{-}(t)$ & $\delta_{\kk}^{+}(t)$ & $\delta_{\ii}^{-}(t)$ \\
  \hline
  $\delta_{\jj}^{+}(s)$&$N_{11}=-1$&$N_{11}= 2$&$N_{11}= 3$&$N_{11}=-1$&$N_{11}=-2$&$N_{11}= 3$\\
                       &$N_{21}=-1$&$N_{21}= 3$&           &$N_{21}= 1$&$N_{21}=-3$&           \\
                       &$N_{31}=-1$&$N_{12}= 3$&           &$N_{31}= 1$&$N_{12}=-3$&           \\
                       &$N_{32}=-2$&           &           &$N_{32}= 2$&           &           \\
  \hline
\end{tabular}
\caption{Chevalley constants for commutators of type $\left[ - ,\delta_{\jj}^{+}(s) \right]$.}\label{gammaconstants}
\end{center}
\end{table}

\begin{table}[!h]
\begin{center}
\begin{tabular}{|l|l|l|l|l|l|l|}
  \hline
                   & $\gamma_{\ii}^{+}(t)$ & $\delta_{\kk}^{-}(t)$ & $\delta_{\ii}^{+}(t)$ & $\gamma_{\kk}^{-}(t)$ & $\delta_{\ii}^{-}(t)$ & $\delta_{\kk}^{+}(t)$ \\
  \hline
  $\delta_{\jj}^{-}(s)$&$N_{11}= 1$&$N_{11}= 2$&$N_{11}=-3$&$N_{11}= 1$&$N_{11}=-2$&$N_{11}=-3$\\
                       &$N_{21}= 1$&$N_{21}=-3$&           &$N_{21}=-1$&$N_{21}= 3$&           \\
                       &$N_{31}=-1$&$N_{12}=-3$&           &$N_{31}= 1$&$N_{12}= 3$&           \\
                       &$N_{32}= 2$&           &           &$N_{32}=-2$&           &           \\
  \hline
\end{tabular}
\caption{Chevalley constants for commutators of type $\left[ - ,\delta_{\jj}^{-}(s) \right]$.}\label{gammaconstants}
\end{center}
\end{table}

\begin{table}[!h]
\begin{center}
\begin{tabular}{|l|l|l|l|l|l|l|}
  \hline
                   & $\gamma_{\ii}^{+}(t)$ & $\delta_{\jj}^{+}(t)$ & $\delta_{\ii}^{-}(t)$ & $\gamma_{\jj}^{-}(t)$ & $\delta_{\ii}^{+}(t)$ & $\delta_{\jj}^{-}(t)$ \\
  \hline
  $\delta_{\kk}^{+}(s)$&$N_{11}=-1$&$N_{11}= 2$&$N_{11}= 3$&$N_{11}=-1$&$N_{11}=-2$&$N_{11}= 3$\\
                       &$N_{21}=-1$&$N_{21}= 3$&           &$N_{21}= 1$&$N_{21}=-3$&           \\
                       &$N_{31}=-1$&$N_{12}= 3$&           &$N_{31}= 1$&$N_{12}=-3$&           \\
                       &$N_{32}=-2$&           &           &$N_{32}= 2$&           &           \\
  \hline
\end{tabular}
\caption{Chevalley constants for commutators of type $\left[ - ,\delta_{\kk}^{+}(s) \right]$.}\label{gammaconstants}
\end{center}
\end{table}

\begin{table}[!h]
\begin{center}
\begin{tabular}{|l|l|l|l|l|l|l|}
  \hline
                   & $\gamma_{\jj}^{+}(t)$ & $\delta_{\ii}^{-}(t)$ & $\delta_{\jj}^{+}(t)$ & $\gamma_{\ii}^{-}(t)$ & $\delta_{\jj}^{-}(t)$ & $\delta_{\ii}^{+}(t)$ \\
  \hline
  $\delta_{\kk}^{-}(s)$&$N_{11}= 1$&$N_{11}= 2$&$N_{11}=-3$&$N_{11}= 1$&$N_{11}=-2$&$N_{11}=-3$\\
                       &$N_{21}= 1$&$N_{21}=-3$&           &$N_{21}=-1$&$N_{21}= 3$&           \\
                       &$N_{31}=-1$&$N_{12}=-3$&           &$N_{31}= 1$&$N_{12}= 3$&           \\
                       &$N_{32}= 2$&           &           &$N_{32}=-2$&           &           \\
  \hline
\end{tabular}
\caption{Chevalley constants for commutators of type $\left[ - ,\delta_{\kk}^{-}(s) \right]$.}\label{gammaconstants}
\end{center}
\end{table}

\chapter{Computations in SAGE}

\begin{lstlisting}
#OCTONION MUTLIPLICATION

R.<A1,A2,A3,A4,A5,A6,A7,A8,B1,B2,B3,B4,B5,B6,B7,B8> = PolynomialRing(QQ)
def octonion_mult(A1,A2,A3,A4,A5,A6,A7,A8,B1,B2,B3,B4,B5,B6,B7,B8):
    C1,C8 = var('C1,C8')
    alpha1 = vector([A2,A3,A4])
    alpha2 = vector([A5,A6,A7])
    beta1 = vector([B2,B3,B4])
    beta2 = vector([B5,B6,B7])
    gamma1 = (A1*beta1) + (B8*alpha1) - (alpha2.cross_product(beta2))
    gamma2 = (B1*alpha2) + (A8*beta2) + (alpha1.cross_product(beta1))
    C1 = (A1*B1) + (alpha1.dot_product(beta2))
    C8 = (alpha2.dot_product(beta1)) + (A8*B8)
    product = vector([C1,gamma1[0],gamma1[1],gamma1[2],gamma2[0],gamma2[1],gamma2[2],C8])
    return product
\end{lstlisting}

\begin{lstlisting}
#AUTOMORPHISM DEFINITIONS

def AutoG2(type,sign,embedding,constant,A):
    S = matrix([[1,((sign+1)%2)*constant],[(sign%2)*constant,1]])
    if type == 0:
        if embedding == 0:
            M1 = vector([A[0],A[1],0,0,A[4],0,0,A[7]])
            M2 = matrix([[A[2],A[6]],[-A[3],A[5]]])
            U = S*M2
            V = octonion_mult(U[0,0],U[0,1],0,0,U[1,0],0,0,U[1,1],0,0,1,0,0,1,0,0)
            return M1 + V
        elif embedding == 1:
            M1 = vector([A[0],0,A[2],0,0,A[5],0,A[7]])
            M2 = matrix([[A[3],A[4]],[-A[1],A[6]]])
            U = S*M2
            V = octonion_mult(U[0,0],0,U[0,1],0,0,U[1,0],0,U[1,1],0,0,0,1,0,0,1,0)
            return M1 + V
        elif embedding == 2:
            M1 = vector([A[0],0,0,A[3],0,0,A[6],A[7]])
            M2 = matrix([[A[1],A[5]],[-A[2],A[4]]])
            U = S*M2
            V = octonion_mult(U[0,0],0,0,U[0,1],0,0,U[1,0],U[1,1],0,1,0,0,1,0,0,0)
            return M1 + V
        else:
            print "Please enter a '0' for a 'i'-type embedding, a '1' for a 'j'-type embedding, or a '2' for a 'k'-type embedding."
    elif type == 1:
        if embedding == 0:
            M1 = matrix([[A[0],A[1]],[A[4],A[7]]])
            M2 = matrix([[A[2],A[6]],[-A[3],A[5]]])
            T = (S)*(M1)*(S^-1)
            U = (M2)*(S^-1)
            V = octonion_mult(U[0,0],U[0,1],0,0,U[1,0],0,0,U[1,1],0,0,1,0,0,1,0,0)
            M = vector([T[0,0],T[0,1],0,0,T[1,0],0,0,T[1,1]])
            return M + V
        elif embedding == 1:
            M1 = matrix([[A[0],A[2]],[A[5],A[7]]])
            M2 = matrix([[A[3],A[4]],[-A[1],A[6]]])
            T = (S)*(M1)*(S^-1)
            U = (M2)*(S^-1)
            V = octonion_mult(U[0,0],0,U[0,1],0,0,U[1,0],0,U[1,1],0,0,0,1,0,0,1,0)
            M = vector([T[0,0],0,T[0,1],0,0,T[1,0],0,T[1,1]])
            return M + V
        elif embedding == 2:
            M1 = matrix([[A[0],A[3]],[A[6],A[7]]])
            M2 = matrix([[A[1],A[5]],[-A[2],A[4]]])
            T = (S)*(M1)*(S^-1)
            U = (M2)*(S^-1)
            V = octonion_mult(U[0,0],0,0,U[0,1],0,0,U[1,0],U[1,1],0,1,0,0,1,0,0,0)
            M = vector([T[0,0],0,0,T[0,1],0,0,T[1,0],T[1,1]])
            return M + V
        else:
            print "Please enter '0' for an 'i'-type embedding, '1' for a 'j'-type embedding, or '2' for a 'k'-type embedding."
    else:
        print "Please enter '0' for a gamma-type automorphism or '1' for a delta-type automorphism."
\end{lstlisting}

\begin{lstlisting}
#CHEVALLEY COMMUTATOR

def Chevalley_Commutator(type1,sign1,embedding1,constant1,type2,sign2,embedding2,constant2,X):
    Z1 = AutoG2(type1,sign1,embedding1,constant1,X)
    Z2 = AutoG2(type2,sign2,embedding2,constant2,Z1)
    Z3 = AutoG2(type1,sign1,embedding1,-constant1,Z2)
    Z4 = AutoG2(type2,sign2,embedding2,-constant2,Z3)
    return Z4
\end{lstlisting}

\begin{lstlisting}
#DISPLAY FUNCTION

def Print_Entries(B):
    print
    print "******************************"
    print
    print B[0]
    print B[1]
    print B[2]
    print B[3]
    print B[4]
    print B[5]
    print B[6]
    print B[7]
\end{lstlisting}

\begin{lstlisting}
#GAMMA I PLUS

N11,N12,N13,N21,N23,N31,N32 = var('N11,N12,N13,N21,N23,N31,N32')
R.<s,t,a,b,c,d,e,f,g,h> = PolynomialRing(QQ)
X = vector([a,b,c,d,e,f,g,h])

V = Chevalley_Commutator(0,0,0,s,0,0,1,t,X)
W = AutoG2(0,1,2,N11*s*t,X)
B = V - W
print Print_Entries(B)

V = Chevalley_Commutator(0,0,0,s,1,0,2,t,X)
W1 = AutoG2(1,0,1,N11*s*t,X)
W2 = AutoG2(1,1,0,N12*s*t^2,W1)
W3 = AutoG2(0,0,1,N13*s*t^3,W2)
B = V - W3
print Print_Entries(B)

V = Chevalley_Commutator(0,0,0,s,1,1,1,t,X)
W1 = AutoG2(1,1,2,N11*s*t,X)
W2 = AutoG2(1,0,0,N12*s*t^2,W1)
W3 = AutoG2(0,0,2,N13*s*t^3,W2)
B = V - W3
print Print_Entries(B)

V = Chevalley_Commutator(0,0,0,s,0,0,2,t,X)
W = AutoG2(0,1,1,N11*s*t,X)
B = V - W
print Print_Entries(B)

#GAMMA I PLUS
\end{lstlisting}

\begin{lstlisting}
#GAMMA I MINUS

N11,N12,N13,N21,N23,N31,N32 = var('N11,N12,N13,N21,N23,N31,N32')
R.<s,t,a,b,c,d,e,f,g,h> = PolynomialRing(QQ)
X = vector([a,b,c,d,e,f,g,h])

V = Chevalley_Commutator(0,1,0,s,0,1,1,t,X)
W = AutoG2(0,0,2,N11*s*t,X)
B = V - W
print Print_Entries(B)

V = Chevalley_Commutator(0,1,0,s,1,1,2,t,X)
W1 = AutoG2(1,1,1,N11*s*t,X)
W2 = AutoG2(1,0,0,N12*s*t^2,W1)
W3 = AutoG2(0,1,1,N13*s*t^3,W2)
B = V - W3
print Print_Entries(B)

V = Chevalley_Commutator(0,1,0,s,1,0,1,t,X)
W1 = AutoG2(1,0,2,N11*s*t,X)
W2 = AutoG2(1,1,0,N12*s*t^2,W1)
W3 = AutoG2(0,1,2,N13*s*t^3,W2)
B = V - W3
print Print_Entries(B)

V = Chevalley_Commutator(0,1,0,s,0,1,2,t,X)
W = AutoG2(0,0,1,N11*s*t,X)
B = V - W
print Print_Entries(B)

#GAMMA I MINUS
\end{lstlisting}

\begin{lstlisting}
#GAMMA J PLUS

V = Chevalley_Commutator(0,0,1,s,0,0,0,t,X)
W = AutoG2(0,1,2,N11*s*t,X)
B = V - W
print Print_Entries(B)

V = Chevalley_Commutator(0,0,1,s,1,1,2,t,X)
W1 = AutoG2(1,1,0,N11*s*t,X)
W2 = AutoG2(1,0,1,N12*s^2*t,W1)
W3 = AutoG2(0,0,0,N13*s^3*t,W2)
B = V - W3
print Print_Entries(B)

V = Chevalley_Commutator(0,0,1,s,1,0,0,t,X)
W1 = AutoG2(1,0,2,N11*s*t,X)
W2 = AutoG2(1,1,1,N12*s*t^2,W1)
W3 = AutoG2(0,0,2,N13*s*t^3,W2)
B = V - W3
print Print_Entries(B)

V = Chevalley_Commutator(0,0,1,s,0,0,2,t,X)
W = AutoG2(0,1,0,N11*s*t,X)
B = V - W
print Print_Entries(B)

#GAMMA J PLUS
\end{lstlisting}

\begin{lstlisting}
#GAMMA J MINUS

V = Chevalley_Commutator(0,1,1,s,0,1,0,t,X)
W = AutoG2(0,0,2,N11*s*t,X)
B = V - W
print Print_Entries(B)

V = Chevalley_Commutator(0,1,1,s,1,0,2,t,X)
W1 = AutoG2(1,0,0,N11*s*t,X)
W2 = AutoG2(1,1,1,N12*s*t^2,W1)
W3 = AutoG2(0,1,0,N13*s*t^3,W2)
B = V - W3
print Print_Entries(B)

V = Chevalley_Commutator(0,1,1,s,1,1,0,t,X)
W1 = AutoG2(1,1,2,N11*s*t,X)
W2 = AutoG2(1,0,1,N12*s*t^2,W1)
W3 = AutoG2(0,1,2,N13*s*t^3,W2)
B = V - W3
print Print_Entries(B)

V = Chevalley_Commutator(0,1,1,s,0,1,2,t,X)
W = AutoG2(0,0,0,N11*s*t,X)
B = V - W
print Print_Entries(B)

#GAMMA J MINUS
\end{lstlisting}

\begin{lstlisting}
#GAMMA K PLUS

N11,N12,N13,N21,N23,N31,N32 = var('N11,N12,N13,N21,N23,N31,N32')
R.<s,t,a,b,c,d,e,f,g,h> = PolynomialRing(QQ)
X = vector([a,b,c,d,e,f,g,h])

V = Chevalley_Commutator(0,0,2,s,0,0,0,t,X)
W = AutoG2(0,1,1,N11*s*t,X)
B = V - W
print Print_Entries(B)

V = Chevalley_Commutator(0,0,2,s,1,0,1,t,X)
W1 = AutoG2(1,0,0,N11*s*t,X)
W2 = AutoG2(1,1,2,N12*s*t^2,W1)
W3 = AutoG2(0,0,0,N13*s*t^3,W2)
B = V - W3
print Print_Entries(B)

V = Chevalley_Commutator(0,0,2,s,1,1,0,t,X)
W1 = AutoG2(1,1,1,N11*s*t,X)
W2 = AutoG2(1,0,2,N12*s*t^2,W1)
W3 = AutoG2(0,0,1,N13*s*t^3,W2)
B = V - W3
print Print_Entries(B)

V = Chevalley_Commutator(0,0,2,s,0,0,1,t,X)
W = AutoG2(0,1,0,N11*s*t,X)
B = V - W
print Print_Entries(B)

#GAMMA K PLUS
\end{lstlisting}

\begin{lstlisting}
#GAMMA K MINUS

N11,N12,N13,N21,N23,N31,N32 = var('N11,N12,N13,N21,N23,N31,N32')
R.<s,t,a,b,c,d,e,f,g,h> = PolynomialRing(QQ)
X = vector([a,b,c,d,e,f,g,h])

V = Chevalley_Commutator(0,1,2,s,0,1,0,t,X)
W = AutoG2(0,0,1,N11*s*t,X)
B = V - W
print Print_Entries(B)

V = Chevalley_Commutator(0,1,2,s,1,1,1,t,X)
W1 = AutoG2(1,1,0,N11*s*t,X)
W2 = AutoG2(1,0,2,N12*s*t^2,W1)
W3 = AutoG2(0,1,0,N13*s*t^3,W2)
B = V - W3
print Print_Entries(B)

V = Chevalley_Commutator(0,1,2,s,1,0,0,t,X)
W1 = AutoG2(1,0,1,N11*s*t,X)
W2 = AutoG2(1,1,2,N12*s*t^2,W1)
W3 = AutoG2(0,1,1,N13*s*t^3,W2)
B = V - W3
print Print_Entries(B)

V = Chevalley_Commutator(0,1,2,s,0,1,1,t,X)
W = AutoG2(0,0,0,N11*s*t,X)
B = V - W
print Print_Entries(B)

#GAMMA K MINUS
\end{lstlisting}

\begin{lstlisting}
#DELTA I PLUS

N11,N12,N13,N21,N23,N31,N32 = var('N11,N12,N13,N21,N23,N31,N32')
R.<s,t,a,b,c,d,e,f,g,h> = PolynomialRing(QQ)
X = vector([a,b,c,d,e,f,g,h])

V = Chevalley_Commutator(1,0,0,s,1,1,2,t,X)
W1 = AutoG2(0,1,1,N11*s*t,X)
B = V - W1
print Print_Entries(B)

V = Chevalley_Commutator(1,0,0,s,1,0,1,t,X)
W1 = AutoG2(1,1,2,N11*s*t,X)
W2 = AutoG2(0,1,1,N21*s^2*t,W1)
W3 = AutoG2(0,0,0,N12*s*t^2,W2)
B = V - W3
print Print_Entries(B)

V = Chevalley_Commutator(1,0,0,s,0,1,2,t,X)
W1 = AutoG2(1,0,1,N11*s*t,X)
W2 = AutoG2(1,1,2,N21*s^2*t,W1)
W3 = AutoG2(0,1,1,N31*s^3*t,W2)
W4 = AutoG2(0,0,0,N32*s^3*t^2,W3)
B = V - W4
print Print_Entries(B)

V = Chevalley_Commutator(1,0,0,s,1,1,1,t,X)
W1 = AutoG2(0,0,2,N11*s*t,X)
B = V - W1
print Print_Entries(B)

V = Chevalley_Commutator(1,0,0,s,1,0,2,t,X)
W1 = AutoG2(1,1,1,N11*s*t,X)
W2 = AutoG2(0,0,2,N21*s^2*t,W1)
W3 = AutoG2(0,1,0,N12*s*t^2,W2)
B = V - W3
print Print_Entries(B)

V = Chevalley_Commutator(1,0,0,s,0,0,1,t,X)
W1 = AutoG2(1,0,2,N11*s*t,X)
W2 = AutoG2(1,1,1,N21*s^2*t,W1)
W3 = AutoG2(0,0,2,N31*s^3*t,W2)
W4 = AutoG2(0,1,0,N32*s^3*t^2,W3)
B = V - W4
print Print_Entries(B)

#DELTA I PLUS
\end{lstlisting}

\begin{lstlisting}
#DELTA I MINUS

N11,N12,N13,N21,N23,N31,N32 = var('N11,N12,N13,N21,N23,N31,N32')
R.<s,t,a,b,c,d,e,f,g,h> = PolynomialRing(QQ)
X = vector([a,b,c,d,e,f,g,h])

V = Chevalley_Commutator(1,1,0,s,1,0,1,t,X)
W1 = AutoG2(0,1,2,N11*s*t,X)
B = V - W1
print Print_Entries(B)

V = Chevalley_Commutator(1,1,0,s,1,1,2,t,X)
W1 = AutoG2(1,0,1,N11*s*t,X)
W2 = AutoG2(0,1,2,N21*s^2*t,W1)
W3 = AutoG2(0,0,0,N12*s*t^2,W2)
B = V - W3
print Print_Entries(B)

V = Chevalley_Commutator(1,1,0,s,0,1,1,t,X)
W1 = AutoG2(1,1,2,N11*s*t,X)
W2 = AutoG2(1,0,1,N21*s^2*t,W1)
W3 = AutoG2(0,1,2,N31*s^3*t,W2)
W4 = AutoG2(0,0,0,N32*s^3*t^2,W3)
B = V - W4
print Print_Entries(B)

V = Chevalley_Commutator(1,1,0,s,1,0,2,t,X)
W1 = AutoG2(0,0,1,N11*s*t,X)
B = V - W1
print Print_Entries(B)

V = Chevalley_Commutator(1,1,0,s,1,1,1,t,X)
W1 = AutoG2(1,0,2,N11*s*t,X)
W2 = AutoG2(0,0,1,N21*s^2*t,W1)
W3 = AutoG2(0,1,0,N12*s*t^2,W2)
B = V - W3
print Print_Entries(B)

V = Chevalley_Commutator(1,1,0,s,0,0,2,t,X)
W1 = AutoG2(1,1,1,N11*s*t,X)
W2 = AutoG2(1,0,2,N21*s^2*t,W1)
W3 = AutoG2(0,0,1,N31*s^3*t,W2)
W4 = AutoG2(0,1,0,N32*s^3*t^2,W3)
B = V - W4
print Print_Entries(B)

#DELTA I MINUS
\end{lstlisting}

\begin{lstlisting}
#DELTA J PLUS

N11,N12,N13,N21,N23,N31,N32 = var('N11,N12,N13,N21,N23,N31,N32')
R.<s,t,a,b,c,d,e,f,g,h> = PolynomialRing(QQ)
X = vector([a,b,c,d,e,f,g,h])

V = Chevalley_Commutator(1,0,1,s,1,1,2,t,X)
W1 = AutoG2(0,0,0,N11*s*t,X)
B = V - W1
print Print_Entries(B)

V = Chevalley_Commutator(1,0,1,s,1,0,0,t,X)
W1 = AutoG2(1,1,2,N11*s*t,X)
W2 = AutoG2(0,0,0,N21*s^2*t,W1)
W3 = AutoG2(0,1,1,N12*s*t^2,W2)
B = V - W3
print Print_Entries(B)

V = Chevalley_Commutator(1,0,1,s,0,0,2,t,X)
W1 = AutoG2(1,0,0,N11*s*t,X)
W2 = AutoG2(1,1,2,N21*s^2*t,W1)
W3 = AutoG2(0,0,0,N31*s^3*t,W2)
W4 = AutoG2(0,1,1,N32*s^3*t^2,W3)
B = V - W4
print Print_Entries(B)

V = Chevalley_Commutator(1,0,1,s,1,1,0,t,X)
W1 = AutoG2(0,1,2,N11*s*t,X)
B = V - W1
print Print_Entries(B)

V = Chevalley_Commutator(1,0,1,s,1,0,2,t,X)
W1 = AutoG2(1,1,0,N11*s*t,X)
W2 = AutoG2(0,1,2,N21*s^2*t,W1)
W3 = AutoG2(0,0,1,N12*s*t^2,W2)
B = V - W3
print Print_Entries(B)

V = Chevalley_Commutator(1,0,1,s,0,1,0,t,X)
W1 = AutoG2(1,0,2,N11*s*t,X)
W2 = AutoG2(1,1,0,N21*s^2*t,W1)
W3 = AutoG2(0,1,2,N31*s^3*t,W2)
W4 = AutoG2(0,0,1,N32*s^3*t^2,W3)
B = V - W4
print Print_Entries(B)

#DELTA J PLUS
\end{lstlisting}

\begin{lstlisting}
#DELTA J MINUS

N11,N12,N13,N21,N23,N31,N32 = var('N11,N12,N13,N21,N23,N31,N32')
R.<s,t,a,b,c,d,e,f,g,h> = PolynomialRing(QQ)
X = vector([a,b,c,d,e,f,g,h])

V = Chevalley_Commutator(1,1,1,s,1,0,0,t,X)
W1 = AutoG2(0,0,2,N11*s*t,X)
B = V - W1
print Print_Entries(B)

V = Chevalley_Commutator(1,1,1,s,1,1,2,t,X)
W1 = AutoG2(1,0,0,N11*s*t,X)
W2 = AutoG2(0,0,2,N21*s^2*t,W1)
W3 = AutoG2(0,1,1,N12*s*t^2,W2)
B = V - W3
print Print_Entries(B)

V = Chevalley_Commutator(1,1,1,s,0,0,0,t,X)
W1 = AutoG2(1,1,2,N11*s*t,X)
W2 = AutoG2(1,0,0,N21*s^2*t,W1)
W3 = AutoG2(0,0,2,N31*s^3*t,W2)
W4 = AutoG2(0,1,1,N32*s^3*t^2,W3)
B = V - W4
print Print_Entries(B)

V = Chevalley_Commutator(1,1,1,s,1,0,2,t,X)
W1 = AutoG2(0,1,0,N11*s*t,X)
B = V - W1
print Print_Entries(B)

V = Chevalley_Commutator(1,1,1,s,1,1,0,t,X)
W1 = AutoG2(1,0,2,N11*s*t,X)
W2 = AutoG2(0,1,0,N21*s^2*t,W1)
W3 = AutoG2(0,0,1,N12*s*t^2,W2)
B = V - W3
print Print_Entries(B)

V = Chevalley_Commutator(1,1,1,s,0,1,2,t,X)
W1 = AutoG2(1,1,0,N11*s*t,X)
W2 = AutoG2(1,0,2,N21*s^2*t,W1)
W3 = AutoG2(0,1,0,N31*s^3*t,W2)
W4 = AutoG2(0,0,1,N32*s^3*t^2,W3)
B = V - W4
print Print_Entries(B)

#DELTA J MINUS
\end{lstlisting}

\begin{lstlisting}
#DELTA K PLUS

N11,N12,N13,N21,N23,N31,N32 = var('N11,N12,N13,N21,N23,N31,N32')
R.<s,t,a,b,c,d,e,f,g,h> = PolynomialRing(QQ)
X = vector([a,b,c,d,e,f,g,h])

V = Chevalley_Commutator(1,0,2,s,1,1,0,t,X)
W1 = AutoG2(0,0,1,N11*s*t,X)
B = V - W1
print Print_Entries(B)

V = Chevalley_Commutator(1,0,2,s,1,0,1,t,X)
W1 = AutoG2(1,1,0,N11*s*t,X)
W2 = AutoG2(0,0,1,N21*s^2*t,W1)
W3 = AutoG2(0,1,2,N12*s*t^2,W2)
B = V - W3
print Print_Entries(B)

V = Chevalley_Commutator(1,0,2,s,0,0,0,t,X)
W1 = AutoG2(1,0,1,N11*s*t,X)
W2 = AutoG2(1,1,0,N21*s^2*t,W1)
W3 = AutoG2(0,0,1,N31*s^3*t,W2)
W4 = AutoG2(0,1,2,N32*s^3*t^2,W3)
B = V - W4
print Print_Entries(B)

V = Chevalley_Commutator(1,0,2,s,1,1,1,t,X)
W1 = AutoG2(0,1,0,N11*s*t,X)
B = V - W1
print Print_Entries(B)

V = Chevalley_Commutator(1,0,2,s,1,0,0,t,X)
W1 = AutoG2(1,1,1,N11*s*t,X)
W2 = AutoG2(0,1,0,N21*s^2*t,W1)
W3 = AutoG2(0,0,2,N12*s*t^2,W2)
B = V - W3
print Print_Entries(B)

V = Chevalley_Commutator(1,0,2,s,0,1,1,t,X)
W1 = AutoG2(1,0,0,N11*s*t,X)
W2 = AutoG2(1,1,1,N21*s^2*t,W1)
W3 = AutoG2(0,1,0,N31*s^3*t,W2)
W4 = AutoG2(0,0,2,N32*s^3*t^2,W3)
B = V - W4
print
print Print_Entries(B)

#DELTA K PLUS
\end{lstlisting}

\begin{lstlisting}
#DELTA K MINUS

N11,N12,N13,N21,N23,N31,N32 = var('N11,N12,N13,N21,N23,N31,N32')
R.<s,t,a,b,c,d,e,f,g,h> = PolynomialRing(QQ)
X = vector([a,b,c,d,e,f,g,h])

V = Chevalley_Commutator(1,1,2,s,1,0,1,t,X)
W1 = AutoG2(0,0,0,N11*s*t,X)
B = V - W1
print Print_Entries(B)

V = Chevalley_Commutator(1,1,2,s,1,1,0,t,X)
W1 = AutoG2(1,0,1,N11*s*t,X)
W2 = AutoG2(0,0,0,N21*s^2*t,W1)
W3 = AutoG2(0,1,2,N12*s*t^2,W2)
B = V - W3
print Print_Entries(B)

V = Chevalley_Commutator(1,1,2,s,0,0,1,t,X)
W1 = AutoG2(1,1,0,N11*s*t,X)
W2 = AutoG2(1,0,1,N21*s^2*t,W1)
W3 = AutoG2(0,0,0,N31*s^3*t,W2)
W4 = AutoG2(0,1,2,N32*s^3*t^2,W3)
B = V - W4
print Print_Entries(B)

V = Chevalley_Commutator(1,1,2,s,1,0,0,t,X)
W1 = AutoG2(0,1,1,N11*s*t,X)
B = V - W1
print Print_Entries(B)

V = Chevalley_Commutator(1,1,2,s,1,1,1,t,X)
W1 = AutoG2(1,0,0,N11*s*t,X)
W2 = AutoG2(0,1,1,N21*s^2*t,W1)
W3 = AutoG2(0,0,2,N12*s*t^2,W2)
B = V - W3
print Print_Entries(B)

V = Chevalley_Commutator(1,1,2,s,0,1,0,t,X)
W1 = AutoG2(1,1,1,N11*s*t,X)
W2 = AutoG2(1,0,0,N21*s^2*t,W1)
W3 = AutoG2(0,1,1,N31*s^3*t,W2)
W4 = AutoG2(0,0,2,N32*s^3*t^2,W3)
B = V - W4
print Print_Entries(B)

#DELTA K MINUS
\end{lstlisting}

\nocite{Casselman}
\nocite{Knapp2002}
\nocite{Jacquet1975}
\nocite{BushHen2006}
\nocite{GanYu2003}
\nocite{Involutions1998}
\bibliographystyle{plain}
\bibliography{ThesisTokorcheck}

\end{document}